\documentclass[12pt,a4paper,final,reqno]{amsart}
\usepackage[utf8]{inputenc}
\usepackage{amsmath}
\usepackage{amsfonts}
\usepackage{amssymb}
\usepackage{graphicx}
\usepackage{amssymb,amsmath,amscd, amsthm,stmaryrd,verbatim,
	mathrsfs,tikz,color , amsfonts, 
}
\usetikzlibrary{patterns}
\usepackage[colorlinks,linkcolor=blue,urlcolor=cyan,citecolor=green]{hyperref}
\usepackage{mathtools}
\mathtoolsset{showonlyrefs}
\allowdisplaybreaks
\newcommand{\af}{ \mathbf{a} }
\newcommand{\mb}[1]{\mathbf{#1}}
\newcommand{\mc}[1]{\mathcal{#1}}

\renewcommand{\subset}{\subseteq}

\newcommand{\nt}{{}^{{ N}}\! T}
\newcommand{\nF}{{}^{{ N}}\! F}
\newcommand{\nE}{{}^{{ N}}\! E}
\newcommand{\nf}{{}^{{ N}}\! f}
\newcommand{\np}{{}^{{ N}}\! p}
\newcommand{\nq}{{}^{{ N}}\! q}
\newcommand{\nc}{{}^{{ N}}\! C}
\newcommand{\ntau}{{}^{{ N}}\! \tau}
\newcommand{\nb}{{}^{{ N}}\! \beta}
\newcommand{\wn}{W_{\leq N}}

\newcommand{\wnn}{W_{>N}}
\newcommand{\hn}{\mc{H}_{\leq N}}

\newcommand{\bd}{\mathbf{\cdot }}

\newtheorem{thm}{Theorem}[section]
\newtheorem{prop}[thm]{Proposition}
\newtheorem{lem}[thm]{Lemma}
\newtheorem{cor}[thm]{Corollary}
\newtheorem{ass}[thm]{Assumption}

\newtheorem{conj}[thm]{Conjecture}
\numberwithin{equation}{section}

\begin{document}
		\bibliographystyle{alpha}
	\author{Xun Xie}
	\title[Conjectures P1-P15]{Conjectures P1-P15 for Coxeter groups with complete graph}
	\address{School of Mathematics and Statistics, Beijing Institute of Technology, Beijing 100081, China}
	\email{xieg7@163.com}
	\date{\today}
	


\maketitle
\begin{abstract}
	We prove  Lusztig's conjectures P1-P15   for  Coxeter groups with complete graph, using decreasing induction on $ \af  $-values  and a kind of factorization formula of Kazhdan-Lusztig basis elements. As a byproduct, we give a description of the left, right, and two-sided cells. In the appendix, we prove P1-P15 for right-angled Coxeter groups by the same methods.
\end{abstract}

\setcounter{tocdepth}{1}
\tableofcontents
\section{Introduction}

\subsection{Background}
Lusztig stated a series of conjectures in \cite[14.2]{lusztig2003hecke}, called P1-P15, for general Coxeter groups with a (positive) weight function. These conjectures mainly concern some properties of cells and the $ \af $-function defined in terms of Kazhdan-Lusztig basis of the Hecke algebra with unequal parameters. The main goal of this paper is to prove P1-P15 for Coxeter groups with complete (Coxeter) graph.

The cells of a Coxeter group (in the equal parameter case) are defined in \cite{kazhdan_lusztig79representation} for the study of representations of Hecke algebras. Cells of finite and affine Weyl groups also appear naturally in other contexts of representation theory. Left cells of a finite Weyl group are in bijection with  primitive ideals of the enveloping algebra of the corresponding semisimple Lie algebra, see \cite[1.6(c)]{kazhdan_lusztig79representation}. Cells of finite Weyl groups (including some special unequal parameter case) play an important role in classification of characters of finite groups of Lie type, see \cite{lusztig1984char}. The two-sided cells of the affine Weyl groups also have connection with modular representations of Lie algebras and algebraic groups, see \cite{Hum2002modular,achar2018cells}.
In his works \cite{lusztig1985cellsI,lusztig1987cellsII,lusztig1987III,lusztig1989cellsIV} on cells of affine Weyl groups, Lusztig introduced  $ \af $-functions and asymptotic rings, which are applied to study representations of affine Hecke algebras, see also \cite{xi1994book,xi2007rep}.
Lusztig  proved that cells of the affine Weyl groups are in bijection with unipotent conjugacy classes of algebraic groups over $ \mathbb{C} $, and that values of the $ \af $-function can be given by the dimensions of springer fibers. Based on these works, Lusztig summarized P1-P15 for general Coxeter group with weight function in \cite{lusztig2003hecke}. In the equal parameter case, P1-P15 can be proved by using the positivity conjecture (\cite{kazhdan_lusztig79representation}) of Kazhdan-Lusztig polynomial and  the boundedness conjecture (\cite[13.4]{lusztig2003hecke}) of the $ \af $-function, see \cite[\S15]{lusztig2003hecke}. Up to now, the positivity conjecture has been proved for general Coxeter groups, see \cite{kazhdan-lusztig1980positivity,elias-williamson2014positivity}. However, the boundedness conjecture is only known for  finite Coxeter groups, affine Weyl groups (\cite{lusztig1985cellsI}),  Coxeter groups with complete graph (\cite{xi2012afunction,Shi-Yang2016}),  the rank 3 case (\cite{zhou,gao2016rank3}), universal Coxeter groups \cite{Shi-Yang2015universal} and Coxeter groups with non-3-edge-labeling graph (\cite{li-Shi-non3edge}).  In other words, even in equal parameter case, P1-P15 are  known only for these Coxeter groups.

The Kazhdan-Lusztig basis and cells can be defined for Hecke algebras with unequal parameters (\cite{lusztig1983leftCell}). An important difference from the  equal parameter case  is that there is no  positivity for Kazhdan-Lusztig polynomials. It is an interesting question to know whether P1-P15 hold for the unequal parameter case. In \cite{lusztig2003hecke}, Lusztig proved P1-P15 for the quasi-split case and infinite dihedral groups.
For Weyl group of type $ B_n $ with asymptotic parameters, Bonnaf{\'e} and Iancu  defined an analogue of the Robinson-Schensted algorithm, and gave a description of left cells, see \cite{bonnafe-iancu2003leftcell}. Later on Bonnaf{\'e} determined the two-sided cells,  see \cite{bonnafe2006twosidedcell}. Geck and Iancu proved conjectures P1-P15 except P9, P10, P15 in \cite{geck-iancu2006afunction} using the method of ``leading matrix coefficients" introduced by Geck \cite{geck2002leading}. Later on Geck proved P9, P10 and a weak version of P15 in\cite{geck2006relative}, and  the proof of P15 was  given in \cite[Lem. 4.7]{geck2011rank2}. In the same paper,  P1-P15   are proved for finite dihedral groups  and Coxeter groups of type $ F_4 $, see \cite[Prop. 5.1 and 5.2]{geck2011rank2}. Therefore, for finite Coxeter groups, P1-P15 are still open only for Weyl groups of type $ B_n $ with non-asymptotic parameters. See \cite{bonnafe-geck-iancu-lam2010} for a conjectural description of their cells.

In \cite{guilhot2008computedrank2,guilhot2010cells}, Guilhot  explicitly determined  the left and two-sided cells  of affine Weyl groups of  types $ \tilde{B}_2 $ (or $ \tilde{C}_2 $) and $ \tilde{G}_2 $. Based on the cell partitions, Guilhot and Parkinson gave a proof of P1-P15 for affine Weyl groups of type  $ \tilde{C}_2 $ and $ \tilde{G}_2 $, see \cite{guilhot-parkinson2019G2,guilhot-parkinson2019C2}. They introduced a notion, called a balanced system of cell representations.
Moreover, they found an interesting connection of cells with Plancherel Theorem. Conjectures  P1-P15 for universal Coxeter groups (which are called free Coxeter groups in \cite[Ch.24]{bonnafe2017book}) are also proved in \cite{Shi-Yang2015universal}.

 The lowest two-sided cell of the affine Weyl group is a typical cell. In the equal parameter case, the lowest two-sided cell has an  explicit description, see \cite{shi1987lowest-I,shi1988lowest-II,bedard1988lowest}. For the unequal parameter case, the lowest two-sided cell has a similar description, see \cite[Ch.3]{xi1994book} and \cite{bremke1997lowest,guilhot2008lowest}. In fact,  there always exists a unique lowest two-sided cell for any Coxeter group with weight function if the boundedness conjecture is true, see \cite[Thm. 1.5]{xi2012afunction} for the equal parameter case and \cite[Thm. 2.1]{xie2017complete} for a straightforward generalization to the unequal parameter case.

 In \cite{xi1990based_ring,xi1994based}, Xi proved a conjecture on the structure of the asymptotic ring (also called based ring) of the affine Weyl group in the case of the lowest two-sided cell, and applied it to study certain representations of the affine Hecke algebra. In \cite{xie2015lowest}, we tried to generalize Xi's works to the unequal parameter case. To define the asymptotic ring, we need first to prove  P1-P15 for the lowest two-sided cell. To describe the structure of the asymptotic ring, we need to generalize  a kind of factorization formula for the Kazhdan-Lusztig basis elements corresponding to the lowest two-sided cell (see \cite[Lem.2.7 and Thm.2.9]{xi1990based_ring} and \cite{blasiak2009factorization}) to the unequal parameter case. We found that this factorization formula can be used to prove P1-P15 for the lowest two-sided cell.
Motivated by this, in  \cite{xie2015dec-rank2}, we determined a kind of factorization formula for all the Kazhdan-Lusztig basis elements of affine Hecke algebras of type $ \tilde{B}_2 $ and $ \tilde{G}_2 $, and proved P1-P15 under some assumptions.

Xi \cite{xi2012afunction} proved the boundedness conjecture for Coxeter groups with complete graph, see also \cite{Shi-Yang2016} for the unequal parameter case. It turns out these Coxeter groups are relatively easy to deal with, partly because the reduced expressions of their elements can be described explicitly (see \cite{shi-2015-reduced-strict,Shi2016reduced}).  Based on these works, we proved P1-P15 for the lowest two-sided cell of the Coxeter group with complete graph, and gave a description of the structure of its asymptotic ring, see \cite{xie2017complete}. 

The main goal of this paper is to prove P1-P15 for Coxeter groups with complete graph. Some ideas for the proof of P1-P15 here originate from our previous works \cite{xie2015lowest,xie2015dec-rank2,xie2017complete} on the lowest two-sided cell and the factorization formula.
\subsection{Main idea}

In this subsection, $ (W,S,L) $ is a  weighted Coxeter group with complete graph. Let $ N\in\mathbb{N} $. Define $$  W_{\geq N}=\{w\in W\mid \af(w)\geq N \}  ,$$ and similarly  $ W_{\leq N} $, $ W_{>N} $, $ W_N $. 
Let $ D $ be the set consisting of: 
\begin{itemize}
\item $ w_J $, where $ J\subset S $ such that the parabolic subgroup $ W_I $ is finite, and 
\item $ sw_I $, where $ I=\{s,t\}\subset S $, $ 4\leq m_{st}<\infty $, $ L(s)<L(t) $.
\end{itemize}
 Let $ \af':D\to \mathbb{N} $ be a function given by \begin{align*}
\af'(w_J)&=L(w_J),\\
\af'(sw_I)&=L(t)+(\frac{m_{st}}{2}-1)(L(t)-L(s)).
\end{align*}
Define $ D_{\geq N}=\{d\in D \mid \af'(d)\geq N\} $ and  $ D_N=D_{\geq N}\setminus D_{\geq N+1} $. Define 
\[ 
\Omega_{\geq N} =\left\{w\in W\,\Big |\,  
w=x\bd d\bd y \text{ for some }d\in D_{\geq N}, x,y\in W
\right\},
\]
and  $ \Omega_{ N}=\Omega_{\geq N}\setminus \Omega_{\geq N+1}$. For $ d\in D_N $, define
\[U_d=\{y\in W\mid dy=d\bd y\in \Omega_{N}\},\]
\[B_d=\{b\in U_d^{-1} \,|\,
\text{if } bd=w\bd v \text{ with }  v\neq e,\text{ then }  w\in \Omega_{<N}  \}.
\]
Here we use  notation \eqref{convention}.

Assume that $ W_{>N} $ is $ \prec_{ {LR}} $ closed. Then we can consider the quotient algebra $ \hn=\mc{H}/\mc{H}_{>N} $, where $ \mc{H} $ is the Hecke algebra and  $ \mc{H}_{>N} $ is the two-sided ideal spanned by $C_w  $, $ w\in W_{>N} $. Denote the image of $ T_w $ and $ C_w $ in $ \hn $ by $ \nt_w $ and $ \nc_w $ for any $ w\in W $. Then $$  \{\nt_w\mid w\in\wn \},\quad \{\nc_w\mid w\in\wn \}  $$ form two basis of $ \hn $. Define the degree of an element of $ \hn $ to be the maximal degree of its coefficients with respect to the basis $ \nt_w $, $ w\in\wn $.

The starting point of this paper is to show that if $ \wnn=\Omega_{>N} $ is $ \prec_{ {LR}} $ closed, then we have inequalities about degrees of products:
\begin{itemize}
\item [(i)] $ \deg \nt_x\nt_y\leq N $ for any $ x,y\in\wn $, and the equality holds only if $ x,y\in\Omega_{ N} $;
\item [(ii)] $ \deg \nt_{x}\nt_v\nt_y \leq -\deg p_{v,d}$ for any $ d\in D_N $, $ x\in U_d^{-1}$,  $  y\in U_d $, $ v\leq d $;
\item [(iii)] $ \deg \nt_{b}\nt_v\nt_y < -\deg p_{v,d}$ for any $ d\in D_N $, $ b\in B_d $, $ y\in U_d $, $ v<d $.
\end{itemize}
Using these, we prove a factorization formula\[
\nc_{bdy}=\nE_b\nc_d\nF_{y}\text{ in }\hn,
\]
where $  \nE_b$, $ \nF_{y} $ are elements such that $ \nc_{bd}=\nE_b\nc_d $, $ \nc_{dy}=\nc_d\nF_{y} $. The key point here is that $ \nE_b $ (resp. $ \nF_y $) is independent of $  y$ (resp. $ b $), and $ \nc_d\nc_d=\eta_d\nc_d $ with $ \deg \eta_d=N $. Note that  (i) is a kind of generalization of the boundedness conjecture.

The main strategy  of this paper is using decreasing induction on $ N $  to prove    P1-P11,  P13-P15 for $ W_{\geq N} $ and $ W_{\geq N}=\Omega_{\geq  N} $. It holds for $ N $ large enough, since $  W_{\geq N}=\Omega_{\geq  N}=\emptyset $ by the boundedness conjecture which has been proved in \cite{xi2012afunction}. We deal with P12 alone.
It is worth mentioning that general facts in section \ref{sec:general} play an important role in our proof, and Lemma \ref{lem:compute} can be used to  compute $ \af $-values.

Recently, Gao and the author have proved P1-P15 for hyperbolic Coxeter groups of rank 3 based on ideas of this paper. In this case, the counterpart of Lemma \ref{lem:shiyang} becomes  complicated.  
\subsection{Organization}

In section \ref{sec:conj}, we fix some basic notations, and clarify  the precise meaning of ``P1-P15 for $ W_{\geq N} $". 
In section \ref{sec:general}, we consider the quotient algebra $ \hn $ and prove that $ \nt_w $, $ w\in\wn $ and $ \nc_w $, $ w\in\wn $ form two-basis of $ \hn $. We prove a cyclic property (Lemma \ref{lem:cyc}(iii)), which is useful in determining  left cells, and prove Lemma \ref{lem:compute}, which can be used to compute the $ \af $-values.
In section \ref{sec:rank2}, we fix some notations about finite dihedral group that we use frequently, and  prove Proposition \ref{prop:key} on some properties in finite dihedral groups, which will be used in the proof of Proposition \ref{prop:key1}.

In section \ref{sec:complete}, we recall some basic properties about Coxeter groups with complete graph.
Section \ref{sec:main} is the main part of this paper. We prove the factorization formula (Theorem \ref{thm:dec}),  and its corollaries Theorem  \ref{thm:cor} and \ref{thm:cor2}. Then we prepare two propositions for the proof of P1-P15. Section  \ref{sec:proof} is devoted to the proof of P1-P15. In section \ref{sec:cells}, we describe  two-sided cells of $ W $. In section \ref{sec:generalized}, we conjecture that our methods  can be generalized to Coxeter groups  whose finite parabolic subgroups have irreducible components of rank 1 or 2.

In appendix \ref{ap:dihedral}, we give a new proof of P1-P15 for finite dihedral groups. In  appendix \ref{ap:ra} we prove P1-P15 for right-angled Coxeter groups. 

\section{Conjectures $ (\text{P1-P15})_{\geq N} $}\label{sec:conj}

Let $ (W,S) $ be a Coxeter group. Throughout this  article, we always assume that $ S $ is a finite set. For $ s,t\in S $, let $ m_{st} \in \mathbb{N}\cup\{\infty\}$ be the order of $ st $ in $ W $. For example, if $ m_{st}=1 $, then $ s=t $; if $ m_{st}=2 $, then $ st=ts $.
The neutral element of the group $ W $ is denoted by $ e $.
Associated to $ (W,S) $, we can define a graph, called Coxeter graph. We call $ (W,S) $  a Coxeter group with complete graph if its Coxeter graph is complete, or equivalently  $ m_{st}\geq 3 $ for any $ s\neq t $ in $ S $. For $ I\subset S $, we have a parabolic subgroup $ W_I $, which is the subgroup of $ W $ generated by $ I $. If $ W_I $ is a finite group, then we denote by $ w_I $ the longest element of $ W_I $. If $ I=\{ s,t\} $ with $ m_{st}<\infty $, we also write $ w_I $ simply as $ w_{st} $.

For a Coxeter group $ (W,S) $, we denote the length function  by $l: W\to \mathbb{N} $. 
For $ w,w_1,w_2,\dots, w_k\in W $, we write
\begin{equation}\label{convention}
w=w_1\bd w_2\bd \dots \bd w_k \text{ if } w=w_1w_2\dots w_k \text{ and } l(w)=\sum _{i=1}^{k}l(w_i).
\end{equation}
For simplicity,  we often use  \[xy=x\bd y  \text{ (resp.  }xy<x\bd y)\] to represent $ l(xy)=l(x)+l(y) $ (resp. $ l(xy)<l(x)+l(y) $) with $ x,y\in W $. 
We also use similar notations for the product of more than two elements.

A weight function on $ W $ is  a  function  $ L:W\to \mathbb{Z} $ such that  $$  L(ww')=L(w)+L(w')  $$ when $ ww'=w\bd w' $. Unless otherwise stated, the weight function in this paper is assumed to be positive, i.e. $ L(s)>0 $ for any $ s\in S $.

Let $ \mc{A}=\mathbb{Z}[q,q^{-1}] $. Associated to $ (W,S,L) $, we have an algebra  $ \mc{H} $ over $ \mc{A} $, called the Hecke algebra. It has an $ \mc{A} $-basis $ \{T_w\mid w\in W \} $ with relations: $$  T_{ww'}=T_wT_{w'} \text{ if } ww'=w\bd w',  $$ $$ \text{ and } T_s^2=1+\xi_s T_s,  \text{ where } \xi_s=q^{L(s)}-q^{-L(s)} \in\mc{A}.$$
For $ 0\neq a=\sum_i\alpha_i q^i\in \mc{A} $ with $ \alpha_i \in \mathbb{Z}$, we define $$  \deg a=\max\{i\mid \alpha_i\neq 0 \} . $$ For $ 0\in \mc{A} $, we define $ \deg 0=-\infty $. For $  h=\sum_{w\in W} a_w T_w$ with $ a_w\in\mc{A} $, we define $$  \deg h=\max\{ \deg a_w \mid w\in W\}  .$$ This gives a function $ \deg : \mc{H}\to \mathbb{N}\cup \{-\infty\} $.

There is a unique $ \mc{A} $-basis $ \{C_w\mid w\in W \} $ of $ \mc{H} $, called the Kazhdan-Lusztig basis, such that \begin{itemize}
	\item [(1)] $ C_w\equiv T_w\mod \mc{H}_{<0} $, where $ \mc{H}_{<0}=\bigoplus_{w\in W}\mc{A}_{<0}T_w $ with $ \mc{A}_{<0}=q^{-1}\mathbb{Z}[q^{-1}] $,
	\item [(2)] and $ C_w $ is invariant under the bar involution $\bar{\cdot}$, which is a $ \mathbb{Z} $-algebra endomorphism on $ \mc{H} $ such that $ \overline{q}=q^{-1} $ and $ \overline{T_w}=T_{w^{-1}}^{-1} $.
\end{itemize}
Let $ p_{y,w}\in\mc{A} $ be the Kazhdan-Lusztig polynomial, which is given by $ C_w=\sum_{y\in W} p_{y,w}T_y $. By  definition, we have  $ p_{w,w}=1 $, $ y\leq w $ if $ p_{y,w}\neq0 $, and $ \deg p_{y,w}< 0 $ if $ y<w $.
Using Kazhdan-Lusztig basis, one can define preorders  $  \prec_{ {L}}$, $ \prec_{ {R}} $, $ \prec_{ {LR}} $ on $ W $, and  corresponding equivalence relations $ \sim_{ {L}} $, $ \sim_{ {R}} $, $ \sim_{ {LR}} $ on $ W $. The associated equivalence classes are called respectively left cells, right cells and two-sided cell. See for example \cite[\S 8]{lusztig2003hecke}.

Define $ f_{x,y,z} \in\mc{A}$ and $ h_{x,y,z} \in\mc{A}$ by \[
T_xT_y=\sum_{z\in W} f_{x,y,z}T_z,\quad C_xC_y=\sum_{z\in W} h_{x,y,z}C_z.
\]
For $ w\in W $, define $$  \af(w):=\max\{\deg h_{x,y,w}\mid x,y\in W \}  .$$ Then $ \af:W\to \mathbb{N}\cup\{\infty \} $ is called the Lusztig's $ \af $-function. Define $ \gamma_{x,y,z^{-1}} $ to be the coefficient of $ q^{\af(z)} $ in $ h_{x,y,z} $.
About boundedness of the $ \af $-function, there is a conjecture as follows, which is still open in general.
\begin{conj}{\normalfont (\cite[13.4]{lusztig2003hecke})} \label{conj:bound}
	Assume that  $ (W,S) $  is  a Coxeter group  with weight function $ L $.  Define \(
	N_0=\max\{L(w_I)\mid I\subset S\text{ with }W_I\text{ finite}\}.
	\)
Then the following equivalent statements hold:
\begin{itemize}
	\item $ \af(w)\leq N_0 $ for any $ w\in W $, 
	\item  $\deg T_xT_y\leq N_0  $ for any $ x,y\in W $,
	\item $ \deg h_{x,y,z}\leq N_0 $ for any $ x,y,z\in W $,
	\item $ \deg f_{x,y,z}\leq N_0 $ for any $ x,y,z\in W $.
\end{itemize}
\end{conj}

%

For $ w\in W $, integers $ \Delta(w) $ and $ n_w $ are 
defined by\[
p_{e,w}=n_wv^{-\Delta(w)}+\text{ terms of lower degree, and } n_w\neq0.
\]
Let $ \mc{D}=\{z\in W \mid \af(z)=\Delta(z) \} $.

For $ N\in\mathbb{N} $, we denote by
\[
W_{\geq N}:=\{w\in W\mid \af(w)\geq N \},
\]
\[
W_{>N}:=W_{\geq (N+1)},\quad W_{N}:=W_{\geq N}\setminus W_{>N},
\]
and similarly  define  $ W_{\leq N} $,  $ W_{<N} $. Note that it is possible that $ W_{N} $ is an empty set. Let  $ \mc{D}_{\geq N}=\mc{D}\cap W_{\geq N} $, and similarly define $  \mc{D}_{N}  $ etc.
\begin{conj}\label{conj15} Let $ N\in\mathbb{N} $.
	
	$ (\text{P1})_{\geq N} $. For any  $w\in W_{\geq N} $, we have $\af(w)\leq\Delta(w)$.
	
$ (\text{P2})_{\geq N} $. If $ z\in\mc{D}_{\geq N} $ and $ x,y\in W $ such that $ \gamma_{x,y,z}\neq0 $, then $ x=y^{-1} $.
	
$ (\text{P3})_{\geq N} $. If $ y\in W_{\geq N} $, there exists a unique $ z\in \mc{D} $ such that $ \gamma_{y^{-1},y,z} \neq0$.
	
$ (\text{P4})_{\geq N} $. If $ w'\prec_{LR} w $ with $ w\in W_{\geq N} $, then $ \af(w')\geq\af(w) $.
	
$ (\text{P5})_{\geq N} $. If $z\in\mathcal{D}_{\geq N}$, $ y\in W $, $\gamma_{y^{-1},y,z}\neq0$, then $\gamma_{y^{-1},y,z}=n_z=\pm1$.
	
$ (\text{P6})_{\geq N} $. For any $z\in\mathcal{D}_{\geq N}$, we have  $z^2=e$.
	
$ (\text{P7})_{\geq N} $.  For any  $x,y,z\in W$ with one of them belonging to $ W_{\geq N} $, we have $\gamma_{x,y,z}=\gamma_{y,z,x}=\gamma_{z,x,y}$.
	
$ (\text{P8})_{\geq N} $.    For any $x,y,z\in W$ with one of them belonging to $ W_{\geq N} $,  $\gamma_{x,y,z}\neq0$ implies that  $x\sim_{L}y^{-1}$, $y\sim_{ L}z^{-1}$, $z\sim_{ L}x^{-1}$.
	
$ (\text{P9})_{\geq N} $.  If  $w'\prec_{ {L}}w$ with $ w\in W_{\geq N} $ and $\af(w')=\af(w)$, then $w'\sim_{ {L}}w$.
	
$ (\text{P10})_{\geq N} $. If  $w'\prec_{ {R}}w$ with $ w\in W_{\geq N} $ and $\af(w')=\af(w)$, then $w'\sim_{ {R}}w$.
	
$ (\text{P11})_{\geq N} $.  If  $w'\prec_{ {LR}}w$ with $ w\in W_{\geq N} $ and $\af(w')=\af(w)$, then $w'\sim_{ {LR}}w$.
	
$ (\text{P12})_{\geq N} $. For any $I\subset S$ and $y\in W_I\cap W_{\geq N}$,  the $ \mb{a} $-value of $ y $ computed in  $W_I$ is equal to that  in $W$. 
	
$ (\text{P13})_{\geq N} $. Any left cell  $\Gamma\subset W_{\geq N}$ contains  a unique element $ z$ in  $\mathcal{D}$. Moreover, for any $y\in\Gamma$, we have $\gamma_{y^{-1},y,z}\neq0$. 
	
$ (\text{P14})_{\geq N} $. For any $w\in W_{\geq N}$,  we have  $w\sim_{ {LR}}w^{-1}$.
	
$ (\text{P15})_{\geq N} $. For   $w,w'\in W$ and  $ x,y\in  W_{\geq N} $ such that  $ \af(x)=\af(y)$, we have\[
\sum_{z\in W}h_{w,x,z}\otimes h_{z,w',y}  =\sum_{z\in W} h_{w,z,y}\otimes h_{x,w',z} \in  \mc{A}\otimes_\mathbb{Z}\mc{A}.	\]
\end{conj}

We  refer to Conjecture \ref{conj15} as $ (\text{P1-P15})_{\geq N} $. Similarly, we use $ (\text{P1-P15})_{>N} $ (resp. $ (\text{P1-P15})_{N} $)
to denote the statements  obtained by replacing $\geq N $ by $   >N $ (resp. $  N  $) in  Conjecture \ref{conj15}.
When $ N=0 $, $ W_{\geq 0}=W $, and hence $  (\text{P1-P15})_{\geq 0}  $ is just Lusztig's conjectures P1-P15 from \cite[\S14.2]{lusztig2003hecke}.

The following lemma is immediate.
\begin{lem}\label{lem:closed}
	If $ (\text{P4})_{\geq N} $ (resp. $ (\text{P4})_{>N} $) holds, then  $ W_{\geq N} $ (resp. $ W_{>N} $) is closed under the preorder $ \prec_{ {LR}} $.
\end{lem}

The main goal of this paper is to prove P1-P15 for Coxeter groups with complete graph. Roughly speaking, our main strategy is using decreasing induction on $ N $: assuming  $ (\text{P1-P15})_{>N} $ and then proving $ (\text{P1-P15})_{N} $.

\section{Some general facts}\label{sec:general}

\subsection{The quotient algebra $ \mc{H}_{\leq N} $}\label{subsec:basis}

\begin{ass}\label{ass1}
	In this subsection, $ (W,S) $ is any Coxeter group, and we fix an integer $ N $ such that $ W_{> N} $ is closed under the preorder $ \prec_{ {LR}} $.
\end{ass}

Let $ \mc{H}_{>N} $ be the subspace of $ \mc{H} $ spanned by $ \{ C_w\mid w\in W_{>N} \} $ over $ \mc{A} $. Then $ \mc{H}_{>N} $ is a two-sided ideal of $ \mc{H} $ since $ W_{>N} $ is $ \prec_{LR} $ closed.
 Let $ \mc{H}_{\leq N} $  be the quotient algebra $ \mc{H}/\mc{H}_{> N} $. For any $ w\in W $, denote by $ \nt_w $ the image of $ T_w $ under the quotient map \(
 \mc{H}\to \mc{H}_{\leq N}.
\)

\begin{lem}\label{lem:independent}
	The set \(
 \{\nt_w\mid w\in W_{\leq N} \} 
 \) forms an $ \mc{A} $-basis of the quotient algebra $ \mc{H}_{\leq N}  $.  
\end{lem}
\begin{proof}
	If $ z\in\wnn $, we have $ C_z\in \mc{H}_{>N}  $, and hence
\begin{equation}\label{eq:minus}
	\nt_z=-\sum_{y<z}p_{y,z} \nt_y.
\end{equation}
Then using induction on the Bruhat order,  we know that elements $  \nt_w$, $w\in W_{\leq N}   $ indeed span $ \hn $.
	
	Assume that $  \nt_w, w\in W_{\leq N}   $  are not linearly independent. Then $$  \sum_{w\in W_{\leq N}} a_w\nt_w=0 $$ for some $ a_w\in\mc{A} $ and $ a_w\neq0 $ for some $ w $. Thus
\begin{equation}\label{eq:thus}
	\sum_{w\in W_{\leq N}} a_w T_w=\sum_{y\in W_{>N}}b_y C_y
\end{equation}
	 for some $ b_y\in\mc{A} $. We have $ b_y \neq0 $ for some $ y\in W_{>N} $. Let $ y_0 $ be the maximal element in $ W_{>N} $ such that $ b_{y_0} \neq 0$.  The coefficient of $ T_{y_0} $ is $ b_{y_0} \neq 0$ on the right-hand side of \eqref{eq:thus}, but it is  zero  on the left-hand side since $ y_0\notin\wn $. This is a contradiction.  Thus $  \nt_w, w\in W_{\leq N}   $  are linearly independent.	
\end{proof}

For $ x,y,z\in W_{\leq N} $, define $ \nf_{x,y,z}\in \mc{A} $  by the expansion
\begin{equation}\label{eq:fxyz}
\nt_x\nt_y =\sum_{z\in W_{\leq N} } \nf_{x,y,z}\nt_z.
\end{equation}
For $ h=\sum_{z\in\wn}b_z\nt_z \in\hn$, we define $$  \deg(h):= \max\{\deg b_{z}\mid z\in\wn \}.  $$ This gives a function $ \deg:\hn\to \mathbb{N}\cup\{-\infty\}$. Using \eqref{eq:minus} and induction, one can prove that \begin{equation}\label{eq:degree}
\deg(\nt_z)<0 \text{ for }z\in W_{>N}.
\end{equation}

For any $ w\in W $, we denote by $ \nc_w $ the image of $ C_w $ in the quotient algebra $ \hn $. We have $ \nc_w=\sum_{y\leq w} p_{y,w} \nt_y $ for any $ w\in W $.  Note that $ \nc_w=0 $ for $ w\in\wnn $. By applying \eqref{eq:minus} and using induction on length, for $ w\in W_{\leq N} $, we have  
\begin{equation}\label{eq:npyw}
	\nc_w=\sum_{y\in W_{\leq N}}\np_{y,w}\nt_y,
\end{equation}
for some $ \np_{y,w}\in\mc{A} $ such that 
$ \np_{y,w}=0 $ unless $ y\leq w $, $ \np_{w,w}=1 $,  and $ \deg \np_{y,w}<0 $ for $ y<w $. In particular,  by Lemma \ref{lem:independent}, $ \{\nc_w\mid w\in\wn \} $ is an $ \mc{A} $-basis of $ \hn $. 

Note that the two-sided ideal $  \mc{H}_{>N}   $ is bar invariant, and hence we have a bar involution on $ \hn $ which is induced naturally from the one on $ \mc{H} $. 

\begin{lem}\label{lem:char}
	The elements $ \nc_w $, $ w\in \wn $ form  an $ \mc{A} $-basis of $ \hn $, and they are characterized as the unique elements of $ \hn $ such that \[
	\nc_w\equiv \nt_w\mod (\hn)_{<0}, \text{ and } \nc_w\text{ is bar invariant,}
	\]where  $ (\hn)_{<0} =\bigoplus_{w\in \wn}\mc{A}_{<0}\nt_w$.
\end{lem}
\begin{proof}
It only remains to prove uniqueness. For this,  it suffices to prove the claim that \begin{equation}\label{claim:uniqueness}
 \text{if }\sum_{y\in\wn}a_y\nt_y\in (\hn)_{<0}	\text{ and is bar invariant, then } a_y=0\text{ for all }y.
	\end{equation}
	
Take $ h=\sum_{y\in\wn} a_y T_y\in \mc{H} $. Since the image of $ h $ in $ \hn $ is bar invariant, we have $ \bar{h}-h\in\mc{H}_{>N}  $. Write $$  \bar{h}-h=\sum _{y\in\wnn} b_y C_y  \text{ with } b_y\in\mc{A}  .$$ Obviously, $ \overline{b_y}=-b_y $, and hence $ b_y=\overline{q_y}-q_y $ for some $ q_y\in \mc{A}_{<0} $. Then consider the element $h'=h-\sum_{y\in\wnn}q_y C_y $.  Since both $ a_y $ and $ q_y $ belong to $ \mc{A}_{<0} $, we have $h'\in \mc{H}_{<0} $. On the other hand, by our construction, $ h' $ is bar invariant. These force $ h' =0 $, see for example \cite[5.2(e)]{lusztig2003hecke}. Thus $ \sum_{y\in\wn}a_y\nt_y=0 $ in $ \hn $. By Lemma \ref{lem:independent}, $ a_y=0 $. This proves claim \eqref{claim:uniqueness}.
\end{proof}

It is easy to see that for $ x,y\in \wn $, we have \begin{equation}\label{eq:nhxyz}
  \nc_x\nc_y=\sum_{z\in\wn} h_{x,y,z} \nc_z. 
\end{equation}
By triangularity,
for any $ w\in\wn $, we have
\begin{equation}\label{eq:inverse}
\nt_w=\sum_{y\in\wn}\nq_{y,w}\nc_y,
\end{equation}
 for some $ \nq_{y,w}\in \mc{A} $ such that $ \nq_{y,w}\neq0 $ implies that $ y\leq w $, $ \nq_{w,w}=1 $ and $ \deg \nq_{y,w}<0 $ for $ y<w $. By \eqref{eq:fxyz}, \eqref{eq:npyw}, \eqref{eq:nhxyz} and \eqref{eq:inverse}, for $ x,y,z\in \wn $, we have 
 \begin{equation}\label{eq:hf}
 h_{x,y,z}=\sum_{x',y',z'}\np_{x',x}\np_{y',y}\nf_{x',y',z'}\nq_{z,z'}
\end{equation}
 \begin{equation}\label{eq:fh}
\nf_{x,y,z}=\sum_{x',y',z'}\nq_{x',x}\nq_{y',y}h_{x',y',z'}\np_{z,z'}
\end{equation}
 where $ x',y',z' $ run though $ \wn $.
 
\begin{lem}\label{lem:wn} 	Keep Assumption \ref{ass1}.
	For any $ x,y\in\wn $, we have \begin{equation}\label{key}
	 \deg (\nt_x\nt_y)\leq N ,
	\end{equation} 
\begin{equation}\label{eq:wn}
	W_N=\{z\in \wn\mid \deg (\nf_{x,y,z})=N \text{ for some } x,y\in \wn \}.  
\end{equation}
Write \[
\nf_{x,y,z}=\nb_{x,y,z^{-1}}q^N+\text{terms of lower degree}.
\]
Then $ \nb_{x,y,z}=\gamma_{x,y,z} $  holds for $ x,y,z\in \wn $ such that $ \nb_{x,y,z}\neq0 $ or $ \af(z)=N $.
\end{lem}
 \begin{proof}
For all $ x,y,z\in\wn $, we have $ \deg h_{x,y,z}\leq N $. By \eqref{eq:fh}, we have   $ \deg \nf_{x,y,z}\leq N $. Thus $ \deg (\nt_x\nt_y)\leq N $ alway holds.

If $ \deg \nf_{x,y,z}=N $, then by \eqref{eq:hf}, $ \deg h_{x,y,z}=N $. Hence $ \af(z)\geq N $, but the assumption $ z\in \wn $ forces $ \af(z)=N $.

 Conversely, if $ \af(z)=N $, then there exist some $ u,v\in W $ such that $ \deg h_{u,v,z}=N $. Since  $ W_{>N} $ is $ \prec_{ {LR}} $ closed, we have 
$ h_{a,b,z}=0\text{ if }a \text{ or }b\in \wnn. $ Thus, $ u,v\in \wn $. By \eqref{eq:fh}, $\deg \nf_{u,v,z}=N$. Then we have proved \eqref{eq:wn}.

If  $ \af(z)=N $, then $ \gamma_{x,y,z} $ is the coefficient of $ q^N $ in $ h_{x,y,z^{-1}} $. Using \eqref{eq:fh}, we see that $ \nb_{x,y,z}=\gamma_{x,y,z} $. 

If $ \nb_{x,y,z}\neq0 $, then $ \deg\nf_{x,y,z^{-1}}=N$. By \eqref{eq:wn},  $ \af(z)=N $. From the last paragraph, we know  $ \nb_{x,y,z}=\gamma_{x,y,z} $. 
\end{proof}

Note that combining Lemma \ref{lem:wn} with claim \eqref{eq:degree}, we actually have 
\begin{equation}\label{eq:strict}
\deg \nt_x\nt_y\leq N \text{ for all } x,y\in W,
\end{equation}
and if $ x$ or  $y $ belongs to $ W_{>N} $, then 
\begin{equation}\label{eq:strict1}
\deg \nt_x\nt_y< N.
\end{equation}

 \subsection{Cyclic property}\label{subsec:cyc}
 
\begin{ass}\label{ass2}
	In this subsection, $ (W,S) $ is any Coxeter group, and $ N $ is an integer  such that $ (\text{P1})_{>N} $, $ (\text{P4})_{>N} $, $ (\text{P8})_{>N} $ hold. In particular, $ \wnn $ is $ \prec_{ {LR}} $ closed.
\end{ass}

Let $ \tau:\mc{H}\to\mc{A} $ be the $ \mc{A} $-linear map such that $ \tau(T_w)=\delta_{e,w} $. It is known that $ \tau(T_{x}T_y)=\delta_{x,y^{-1}} $, and for $ h,h'\in\mc{H} $, we have $ \tau(hh')=\tau(h'h) $. Then $ f_{x,y,z^{-1}}=\tau(T_xT_yT_z) $. Since $ \tau(T_xT_yT_z) =\tau(T_yT_zT_x) =\tau(T_zT_xT_y)  $, we have $ f_{x,y,z^{-1}}=f_{y,z,x^{-1}}=f_{z,x,y^{-1}} $.

Let $ \ntau:\hn\to\mc{A} $ be the map defined by $  \ntau(h) =b_e $ for  $   h=\sum_{w\in \wn}b_w\nt_w $ with  $b_w\in \mc{A} .$ 

\begin{lem}
	For $ x,y\in\wn $, we have \begin{equation}\label{eq:tau1}
	\ntau(\nc_x\nc_y)\in\delta_{x,y^{-1}}+\mc{A}_{<0},
	\end{equation}
\begin{equation}\label{eq:tnegative}
\text{ and }	\ntau(\nt_x\nt_y)\in\delta_{x,y^{-1}}+\mc{A}_{<0}.
	\end{equation}
\end{lem}
\begin{proof}
	By using $ \np_{y,w} $ and $ \nq_{y,w} $, it is easy to see that \eqref{eq:tau1} and \eqref{eq:tnegative} are equivalent. Since $ \tau(T_xT_y)=\delta_{x,y^{-1}} $, we have $  \tau(C_xC_y)\in \delta_{x,y^{-1}} +\mc{A}_{<0}$. Thus to prove \eqref{eq:tau1}, it suffices to prove that $ \ntau(\nc_x\nc_y) \equiv \tau(C_xC_y)\mod\mc{A}_{<0}$.
	We have  \begin{align}
	\tau(C_xC_y)&=\sum_{z\in \wn}\tau( h_{x,y,z}C_z) +\sum_{z\in\wnn}\tau(h_{x,y,z}C_z),\\
	\ntau(\nc_x\nc_y)&=\sum_{z\in \wn}\ntau( h_{x,y,z}\nc_z).
	\end{align}
	
	For $ z\in \wnn $,  we have $\deg\tau(h_{x,y,z}C_z) \leq0 $, since $$  \deg h_{x,y,z}\leq \af(z)\leq \Delta(z)=-\deg p_{e,z}  $$ by $ (\text{P1})_{>N} $. In fact, $\deg\tau(h_{x,y,z}C_z) <0 $; otherwise, we have
	$ \gamma_{x,y,z}\neq0 $, which implies that $ \af(x)=\af(z^{-1})=\af(z)>N $ by $ (\text{P4})_{>N} $ and $ (\text{P8})_{>N} $,
 a contradiction with $ x\in\wn $.  Thus, we always have $\sum_{z\in \wnn}\tau(h_{x,y,z}C_z) \in\mc{A}_{<0} $. 
	
To prove the lemma, it is remains to prove that, for $z\in\wn  $,	\begin{equation}\label{eq:remain}
\tau(h_{x,y,z}C_z)\equiv \ntau( h_{x,y,z}\nc_z )\mod \mc{A}_{<0}.
\end{equation}

By definition, for $ e\neq y\in W_{\leq N} $, $ \ntau(\nt_y) =0$. Then
for $ w\in\wnn $, applying $ \ntau $ to \eqref{eq:minus}, we have $$   \ntau(\nt_w)=-p_{e,w}-\sum_{y<w,y\in\wnn}p_{y,w}\ntau(\nt_y) . $$ By $ (\text{P1})_{>N} $, we have   $$  \deg p_{e,w}=-\Delta(w)\leq -\af(w)<-N  .$$ Then using induction on the length, we can prove  \(
	\deg \ntau(\nt_w)<-N\text{ for all }w\in \wnn.
	\)

For any $ z\in \wn $, we have \begin{align*}
 \ntau(\nc_z) &=\sum_{y\leq z} p_{y,z}\ntau(\nt_y)\\
& =p_{e,z}+\sum_{\substack{y< z\\y\in\wnn}} p_{y,z}\ntau(\nt_y)\\
&\equiv p_{e,z} \mod q^{-N-1}\mathbb{Z}[q^{-1}],
\end{align*}
by the last paragraph.
Since $ \tau(C_z)=p_{e,z} $, we obtain \begin{equation}\label{eq:tau}
	\tau(C_z)\equiv\ntau(\nc_z) \mod q^{-N-1}\mathbb{Z}[q^{-1}].
	\end{equation}
	Since $ z\in\wn $, we have $\deg h_{x,y,z} \leq N$, and hence $$  \tau(h_{x,y,z}C_z)\equiv \ntau( h_{x,y,z}\nc_z )\mod \mc{A}_{<0}.  $$ This proves \eqref{eq:remain}.
\end{proof}


 \begin{lem}\label{lem:cyc}
\begin{itemize}
	\item[(i)] For $ h,h'\in\hn $ such that $ \deg h\leq m $ and $ \deg h'\leq m' $, we have $ \ntau(hh')\equiv \ntau(h'h)\mod q^{m+m'-1}\mathbb{Z}[q^{-1}] $.
	\item [(ii)] For $ x,y,z\in\wn $, we have\[
	\nf_{x,y,z^{-1}}=\ntau(\nt_x\nt_y\nt_z) \mod q^{N-1}\mathbb{Z}[q^{-1}].
	\]This implies that 
	\begin{equation}\label{key}
	\nb_{x,y,z}=\nb_{y,z,x}=\nb_{z,x,y}.
	\end{equation}
\item [(iii)] If $ \nb_{x,y,z}\neq0 $ with $ x,y,z\in\wn $, then
\begin{equation}\label{eq:betagamma}
	\nb_{x,y,z}=\nb_{y,z,x}=\nb_{z,x,y}=\gamma_{x,y,z}=\gamma_{y,z,x}=\gamma_{z,x,y},
\end{equation}
and 
$$ y\sim_{ {L}} z^{-1} ,z\sim_{ {L}} x^{-1} ,x\sim_{ {L}}y^{-1} ,\af(x)=\af(y)=\af(z)=N . $$
\item [(iv)] If $ \gamma_{x,y,z}\neq0 $ with $ x,y\in\wn $ and  $ \af(z)=N $, then $$  \nb_{x,y,z}=\gamma_{x,y,z}\neq0  ,$$ and hence conditions and conclusions of assertion (iii) all hold.
\end{itemize}
 \end{lem}
\begin{proof}
	By \eqref{eq:tnegative}, we have $ \ntau(\nt_x\nt_y)\equiv \ntau(\nt_y\nt_x) \mod \mc{A}_{<0}$. Then assertion (i) follows immediately.
	
	We have $ \ntau(\nt_x\nt_y\nt_z)=\sum_{w\in\wn} \nf_{x,y,w}\ntau(\nt_w\nt_z)$. By  Lemma \ref{lem:wn}, we have $ \deg\nf_{x,y,w}  \leq N$. If $ w^{-1}\neq z $, then $$  \nf_{x,y,w}\ntau(\nt_w\nt_z)\in \nf_{x,y,w}\mc{A}_{<0} \subset q^{N-1}\mathbb{Z}[q^{-1}] .$$ If $ w^{-1}= z  $, then $$  \nf_{x,y,w}\ntau(\nt_w\nt_z)\in \nf_{x,y,w}(1+\mc{A}_{<0}) \subset  \nf_{x,y,z^{-1}}+q^{N-1}\mathbb{Z}[q^{-1}] .$$ Hence $   \ntau(\nt_x\nt_y\nt_z)\equiv \nf_{x,y,z^{-1}}\mod q^{N-1}\mathbb{Z}[q^{-1}]$.
	
Using Lemma \ref{lem:wn} and assertion (i), we have $$  \ntau(\nt_x\nt_y\nt_z) \equiv\ntau(\nt_z\nt_x\nt_y)\equiv\ntau(\nt_y\nt_z\nt_x)\mod q^{N-1}\mathbb{Z}[q^{-1}]  .$$
	Hence $$  \nf_{x,y,z^{-1}} \equiv\nf_{y,z,x^{-1}}\equiv\nf_{z,x,y^{-1}}\mod q^{N-1}\mathbb{Z}[q^{-1}]  .$$
	By taking  coefficients of $ q^N $, we have \begin{equation*} 
	\nb_{x,y,z}=\nb_{y,z,x}=\nb_{z,x,y}.
	\end{equation*} This proves assertion (ii).
	
Now we prove assertion (iii).	If $ \nb_{x,y,z}\neq0 $, we have $ \nb_{x,y,z}=\gamma_{x,y,z} $ by Lemma \ref{lem:wn}. Using assertion (ii), we obtain \eqref{eq:betagamma}. Now $ \gamma_{x,y,z}\neq0 $ implies $ z^{-1}\prec_{ {L}} y $, and $\gamma_{z,x,y}\neq0  $ implies $ y^{-1}\prec_R z $. Hence  $ y\sim_{ {L}} z^{-1} $. Similarly, we have $ z\sim_{ {L}} x^{-1} $, $ x\sim_{ {L}}y^{-1} $. At last, since $ \nb_{x,y,z}\neq0 $, we have $ \deg\nf_{x,y,z^{-1}}=N$. Using \eqref{eq:wn}, we have $ \af(z)=N $. Similarly, using $ \nb_{y,z,x}=\nb_{z,x,y}\neq 0 $, we have $ \af(x)=\af(y)=N $. This proves (iii).

Now we prove (iv).
Since $ \af(z)=N $, we have $   \nb_{x,y,z}=\gamma_{x,y,z}$ by Lemma \ref{lem:wn}. Thus $  \nb_{x,y,z}\neq 0 $ and  we can use (iii).
\end{proof}

\begin{lem}\label{lem:compute} Under Assumption \ref{ass2},
 we have \[
W_N=\{ x\in\wn\mid \deg \nt_x\nt_y=N \text{ for some }y\in\wn \}.
\]
\end{lem}\begin{proof}
This follows from Lemma \ref{lem:wn} and the cyclic property \eqref{eq:betagamma}.
\end{proof}
\subsection{Parabolic subgroups}
Let $ J \subset S$. The restriction of the weight function $ L $ on the parabolic subgroup $ W_J$ is still denoted by $ L $. Let $  \mc{H}_J $ be the Hecke algebra corresponding to $ (W_J,J,L) $. Then $ \mc{H}_J $ is naturally isomorphic to the subalgebra of $ \mc{H} $ spanned by $\{ T_w \mid w\in W_J\}$ over $ \mc{A} $. We will identify $ \mc{H}_J $ with this subalgebra. By definition, the Kazhdan-Lusztig basis elements indexed by $ w\in W_J $ in $ \mc{H}_J $ and $ \mc{H} $ coincide, and we use the same notation $ C_w $. 

We can define the preorders $ \prec_{ {L}}^J $, $ \prec_{ {R}}^J $ and $ \prec_{ {LR}}^J $ on $ W_J $ in the same way as those 
on $ W $.
Let $ \af_J: W_J\to \mathbb{N} \cup\{\infty \}$ be the  $ \af $-function defined in terms of  $ W_J $.  Using $ \af_J $ one can define  $(W_J)_{>N}  $, $(W_J)_{N}  $  and so on. For example, $$  ({W_J})_{>N} =\{w\in W_J\mid \af_J(w)>N \}  .$$
If $  ({W_J})_{>N} $ is $ \prec_{LR}^J$ closed, then one can 
define  two-sided ideal $ (\mc{H}_J)_{>N} $ and quotient algebra $ (\mc{H}_J)_{\leq N} $ of $ \mc{H} _J$ in the same way as $ \mc{H} _{>N}$ and $ \mc{H}_{\leq N} $.

\begin{lem} \label{lem:nfxyz}
	Assume that $ ({W_J})_{>N}  $ (resp. $ W_{>N} $) is $ \prec_{LR}^J $ (resp. $\prec_{ {LR}}$) closed. If $ (W_J)_{\leq N}\subset W_{\leq N} $, then $ (\mc{H}_J)_{\leq N} $ can be naturally embedded into $  \mc{H}_{\leq N} $, and hence for $ x,y,z\in (W_J)_{\leq N} $, $ \nf_{x,y,z} \in\mc{A}$ that computed in $ \mc{H}_{\leq N} $ coincide with that  in  $ (\mc{H}_J)_{\leq N}$.
%
\end{lem}
\begin{proof}
By definition of the $ \af $-function, we have $ ({W_J})_{>N} \subset W_{>N} $. Then $$  (\mc{H}_J)_{>N} \subset\mc{H}_{>N}   .$$ Thus we have a homomorphism $$   (\mc{H}_J)_{\leq N}\to \mc{H}_{\leq N} $$ induced from the inclusion $ \mc{H}_{J}\to \mc{H} $. Since we assume $ (W_J)_{\leq N}\subset W_{\leq N} $,   the homomorphism $  (\mc{H}_J)_{\leq N}\to \mc{H}_{\leq N}$ is injective by Lemma \ref{lem:independent}. This implies the lemma.
\end{proof}

%

\section{Finite dihedral groups}\label{sec:rank2}

\begin{ass}
	In this section,  $ (W_I,I) $ is   a dihedral group with a positive weight function $ L $ such that  $ I=\{s,t\}$ and  $ 3\leq m_{st}<\infty $.
\end{ass}

Here are some notations. 

Let $ L(s)=a $, $ L(t)=b $. Let $ \xi_a=\xi_s=q^a-q^{-a} $, $ \xi_b=\xi_t=q^b-q^{-b} $. 

For $ r\in I $ and $ 0<n\leq m_{st} $, we denote by $ w(r,n) $ (resp. $ w(n,r) $) the element $ w $ of  $ W_I $ satisfying $ l(w)=n $ and $ r \in\mc{L}(w)$ (resp. $ r\in\mc{R}(w) $). By convention, $ w(r,0)=w(0,r)=e $.

In the case of $ a\neq b  $, $ m_{st} $ is even, and we set $ m_{st}=2m $ for some $ m\in\mathbb{N} $. In this case, we denote by $ d_I $  the element $ w(r,2m-1) $, where $ r \in I$ is determined by $ \{r,r'\}=\{s,t\} $ and $ L(r) >L(r')$. In this case, we define a new weight function on $ W_I $  \begin{equation}\label{eq:l'}
L':W_I\to \mathbb{Z} 
\end{equation} by $ L'(r)=L(r) $ and $ L'(r')=-L(r') $. 

For example, if $ L(s)=a<b=L(t) $, then $ d_I =ts\cdots t$ with $ 2m-1 $ factors, and $ L'(d_I)=mb-(m-1)a $.



%
\begin{lem}\label{lem:degd}
Assume $ a\neq b $.	For any $ v\leq d_I $, we have $$  \deg p_{v,d_I}=L'(v)-L'(d_I)  .$$

If $ a<b $, then \[
 C_tC_{d_I} =(q^b+q^{-b})C_{d_I},\quad C_sC_{d_I}=C_{w_I},
\]
which implies that 
\begin{equation}\label{eq:sdi}
T_tC_{d_I}=q^bC_{d_I} , \quad T_sC_{d_I}=-q^{-a}C_{d_I}+C_{w_I}.
\end{equation}
\end{lem}
\begin{proof}
The degree of $ p_{v,d_I} $ and the equality $ C_sC_{d_I}=C_{w_I} $ follow from 7.4, 7.5(a) and 7.6 of \cite{lusztig2003hecke}, where an explicit formula of $ C_{d_I} $ is given.	Other assertions are straightforward.
\end{proof}

Note that   \eqref{eq:sdi}  implies that $ \{ d_I,w_I \} $ is $ \prec_{LR} $ closed when $ a\neq b $.

\begin{lem}\label{lem:eta}
Assume $ a\neq b $. 
	We have $ \deg h_{d_I,d_I,d_I} =L'(d_I) $.
\end{lem}

\begin{proof}
	This follows from 7.8 of \cite{lusztig2003hecke}, where an expansion of $ C_{d_I}C_{d_I} $ is given.
\end{proof}

\begin{lem}\label{lem:plusdeg}
Let $ u,v\in W_I $. Then $ f_{u,v,w_I}\neq 0 $ if and only if
\begin{itemize}
\item [(i)] $ \mc{R}(u)\cap\mc{L}(v)=\emptyset $, and $ l(u)+l(v)\geq l(w_I) $, or
\item [(ii)] $ \mc{R}(u)\cap\mc{L}(v)\neq\emptyset $, and  $ l(u)+l(v)\geq l(w_I) +1$.
\end{itemize}
In both cases, we always  have   \begin{equation}\label{eq:deg11}
\deg f_{u,v,w_I}=L(u)+L(v)-L(w_I)  . 
\end{equation}
\end{lem}
\begin{proof}
It is easy to see that
\begin{itemize}
	\item  if $l(u)+l(v)<l(w_I)  $, we have $ f_{u,v,w_I}=0 $; 
	\item if $l(u)+l(v)=l(w_I)  $ but $ \mc{R}(u)\cap\mc{L}(v)\neq\emptyset $, we also have $ f_{u,v,w_I}=0 $. 
\end{itemize}
Then the `` only  if " part follows.

Assume that we are in case (i). Then we have $ u=u_1\bd u_2 $ for some $ u_1 ,u_2\in W_I$ and $ w_I=u_2v $. Thus $ f_{u,v,w_I}=f_{u_1,w_I,w_I} $, whose degree is $$  L(u_1)=L(u)+L(v)-L(w_I)  .$$ In particular, $ f_{u,v,w_I} $ is  nonzero.

Assume that we are in case (ii). If $ v=w_I $, then \eqref{eq:deg11} is obvious. Thus in the following we assume $ v<w_I $. Let $ r\in \mc{R}(u)\cap\mc{L}(v) $ and $ u=u'r $, $ v=rv' $. Then $ f_{u,v,w_I}=f_{u',v',w_I}+\xi_r f_{u',v,w_I} $ and $  \mc{R}(u')\cap\mc{L}(v)=\emptyset $. Since $ l(u')+l(v)\geq l(w_I) $, by the last paragraph, we have $$  \deg\xi_r f_{u',v,w_I}=L(u)+L(v)-L(w_I)  .$$ One can see that $  f_{u',v',w_I} $ is $ 0 $ or has degree $L(u')+L(v')-L(w_I)$ (by using induction on length). Thus, $ \deg f_{u,v,w_I}= L(u)+L(v)-L(w_I)$. In particular, $ f_{u,v,w_I}\neq 0 $ in this case. The `` if " part follows. 
\end{proof}

\begin{lem}\label{lem:degfp2}
Assume that $ a=L(s)\leq L(t)=b $,
$ u,v\in W_I\setminus\{w_I \} $ and $ z\in\{e,s,t,st,ts \} $. For  $ \delta=\deg f_{u,v,w_I}p_{z,w_I} $, we must be in one of the following situations.
\begin{itemize}
\item [(1)] $ \delta\leq 0 $;
\item [(2)]  $ \delta=b-2a>0 $,  $ z=t $, and $ u=v=d_I$;
\item [(3)]  $ \delta=b-a>0 $, $ z=st $ or $ ts $,  and $ u=v=d_I$.
\end{itemize}
\end{lem}
\begin{proof}
	 Without loss of generality, we can assume $ f_{u,v,w_I}\neq 0 $.
 By Lemma \ref{lem:plusdeg},  \begin{align}
  \delta=&L(u)+L(v)-L(w_I)+L(z)-L(w_I) \\=&L(z)-(L(w_I)-L(u))-(L(w_I)-L(v)). 
 \end{align}

Since $ u,v\neq w_I $, the possible values of $ L(w_I)-L(u)$ (resp. $ L(w_I)-L(v) $) are $$  a,b,a+b,2a+b,a+2b,\cdots  .$$
\begin{itemize}
\item If $ z\in\{e,s\} $, then $ L(z)\leq a $, and we always have $ \delta< 0 $.
\item If $ z=t $ and $ \delta>0 $, then $  L(w_I)-L(u)=L(w_I)-L(v)=a$,   which implies $ \delta =b-2a>0 $, $ u=v=d_I $.
\item If $ z=\{st , ts\} $ and $ \delta>0 $, then  $  L(w_I)-L(u)=L(w_I)-L(v)=a$, which implies $ \delta =b-a$, $ u=v=d_I $.
\end{itemize}
This completes the proof. 
\end{proof}

\begin{lem}\label{lem:Fuv}
Assume $ L(s)<L(t) $.
Let $ u,v\in W_I\setminus\{w_I\} $, $ p=l(u),$  $ q=l(v) $, and\[
F(u,v)=f_{u,v,d_I}-p_{d_I,w_I}f_{u,v,w_I}.
\]
The following properties hold. \begin{itemize}
\item[(i)] If  $ vs<v $, then $ F(u,v)=-q^{-a}F(u,vs) $.
\item[(ii)] If  $ su<u $, then $ F(u,v)=-q^{-a}F(su,v) $.
\item [(iii)]Let  $ su>u $ and $ vs>v $. We have the following cases:
 \begin{itemize}
\item[(1)] if $ p+q< 2m-1 $, then $ F(u,v)=0 $;
\item [(2)]if $ p+q=2m-1 $, then $ F(u,v)=1 $;
\item[(3)] if $ p+q=2m $, then $$  F(u,v)=  \begin{cases}
\xi_a &\text{ if } p,q \text{ are even},\\
\xi_b &\text{ if } p,q \text{ are odd};
\end{cases}$$
\item [(4)]if $ p+q>2m $, then $ \deg F(u,v)=L'(u)+L'(v)-L'(d_I) $.
\end{itemize}
\end{itemize}
\end{lem}

\begin{proof} In this proof, we abbreviate $ d_I $ as $ d $.
Note that \begin{align}
 f_{u,v,d}-p_{d,w_I}f_{u,v,w_I}=&f_{v,d,u^{-1}}-q^{-a}f_{v,w_I,u^{-1}}\\ =&\tau(T_v(T_d-q^{-a}T_{w_I})T_u),
\end{align} which is the coefficient of $ T_{u^{-1}} $ in the product $ T_v(T_d-q^{-a}T_{w_I}) $.  In the following, we compute $ T_v(T_d-q^{-a}T_{w_I}) $. 

We use notations: $$  U_{2m-1}= T_d-q^{-a}T_{w_I},\quad U_{2m-2}=  T_{td}-q^{-a}T_{tw_I},\quad\cdots \quad U_0=T_e-q^{-a}T_s  .$$
In other words, $$  U_k=T_{w(k,t)}-q^{-a}T_{w(k+1,s)} \text{ for }0\leq k\leq 2m-1  .$$ Straightforward computations show that 
\[
T_tU_k=\begin{cases}
U_{k+1} &\text{ if } k\text{ is even,}\\
U_{k-1}+\xi_b U_{k} &\text{ if } k\text{ is odd,}
\end{cases}
\]
\[
T_sU_k=\begin{cases}
 U_{k+1}&\text{ if } k\text{ is odd and  }\neq  2m-1,\\
U_{k-1}+\xi_a U_{k} &\text{ if } k\text{ is even and }\neq 0,\\
(-q^{-a})U_k & \text{ if } k= 0, 2m-1.
\end{cases}
\]

Let $ (\lambda_{i,j}) _{0\leq i,j\leq 2m-1}$ be a $ 2m\times 2m $ matrix  with entries in $ \mc{A} $ such that\[
T_{w(i,t) }U_{2m-1}=\sum_{0\leq j\leq 2m-1}\lambda_{i,j}U_{2m-1-j}.
\]
Then  
\[ \lambda_{i,i}=1 \text{ for }  i\geq 0,\text{ and } \lambda_{i,j}=0 \text{ for } i<j, \] and we have a recursive formula, for $ i\geq 1 $,
\begin{equation}\label{eq:induc}
\lambda_{i,j}=\begin{cases}
(-q^{-a})\lambda_{i-1,j}&\text{ if } i \text{ is even and } j=0,2m-1,\\
  \lambda_{i-1,j-1}&\text{ if } i+j \text{ is even and } j\neq 0,\\
\xi_a\lambda_{i-1,j}+\lambda_{i-1,j+1}  &\text{ if } i \text{ is even,  } j\text{ is odd, and }  j \neq 2m-1,\\
\xi_b\lambda_{i-1,j}+\lambda_{i-1,j+1}  &\text{ if } i \text{ is odd and } j\text{ is even}.
\end{cases}
\end{equation}

Set $ \mu_i=\lambda_{i,0} $ for $0\leq i\leq 2m-1 $, and $ \mu_i=0 $ for $ i<0 $. 

Using \eqref{eq:induc}, we can express $ \lambda_{i,j} $ in terms of $ \mu_i $ as follows.
\begin{itemize}
	\item [(I)] If $i+j $ is even, we have $ \lambda_{i,j}=\mu_{i-j} $.
	\item [(II)] If $ i+j $ is odd and $ j\neq 0 $, we have\[
	\lambda_{i,j}=\begin{cases}
	\xi_a\sum_{n\geq0}\mu_{i-j-1-4n}+\xi_b\sum_{n\geq0}\mu_{i-j-3-4n}&\text{ if } i \text{  is even},\\
		\xi_b\sum_{n\geq0}\mu_{i-j-1-4n}+\xi_a\sum_{n\geq0}\mu_{i-j-3-4n}&\text{ if } i \text{  is odd}.
	\end{cases}
	\]
	\item [(III)] We have $ \mu_0=1 $,    and for $ i>0 $,  \[\mu_i=\begin{cases}
	(-q^{-a})\mu_{i-1}& \text{ if } i \text{ is even},\\
	\xi_b\sum_{n\geq0}\mu_{i-1-4n}+\xi_a\sum_{n\geq0}\mu_{i-3-4n}& \text{ if } i \text{ is odd}.
	\end{cases}
	 \] 
\end{itemize}

For $ k\in\mathbb{Z}_{\geq 1} $, we set $$  L'_k=b-a+b-a+b-\cdots  $$ with $ k $ terms appearing, and set $ L'_{0}=0 $. In other words,  $$ L'_{2r}=r(b-a) ,\quad L'_{2r+1} =r(b-a)+b .$$ Since $ b>a $, we have $L'_k>L'_{k_1}  $ if $ k $ is odd and $ k>k_1 $. Using (III), one can prove inductively that \[
\deg \mu_k=L'_k.
\]

Combining (I)(II)(III),  if $ i+j $ is even, or if $ i $ is odd and $ j $ is even, we have $$  \lambda_{i,j}=\mu_{i-j}  .$$ 

 Assume that $ i $ is even and $ j $ is odd. To determine $ \lambda_{i,j} $, let us compute $ 	\xi_a\mu_{2k}+\xi_b\mu_{2k-2} $ for $ k\geq 0$. 

If $ k=0 $, then $ \xi_a\mu_{0}+\xi_b\mu_{-2} =\xi_a $.
Assume $ k\geq 1 $.
 By (III), and using $ 1+(-q^{-a})\xi_a=q^{-2a} $, we have 
\begin{align*}
	&\phantom{=}\xi_a\mu_{2k}+\xi_b\mu_{2k-2}\\
&=(-q^{-a})\xi_a\mu_{2k-1}+\xi_b\mu_{2k-2}\\
	&=(-q^{-a})\xi_a\xi_b\sum_{n\geq0}\mu_{2k-2-4n}+(-q^{-a})\xi_a\xi_a\sum_{n\geq0}\mu_{2k-4-4n}+\xi_b\mu_{2k-2}\\
&=q^{-2a}\xi_b\mu_{2k-2}+(-q^{-a})\xi_a\left [\xi_a\sum_{n\geq0}\mu_{2k-4-4n}+\xi_b\sum_{n\geq0}\mu_{2k-6-4n}\right ].
\end{align*}
Then we have
\begin{align*}
&\phantom{=} \xi_a\sum_{n\geq 0}\mu_{2k-4n}+\xi_b\sum_{n\geq 0}\mu_{2k-2-4n}\\&=(\xi_a\mu_{2k}+\xi_b\mu_{2k-2})+(\xi_a\sum_{n\geq0}\mu_{2k-4-4n}+\xi_b\sum_{n\geq0}\mu_{2k-6-4n})\\
&=q^{-2a}\xi_b\mu_{2k-2}+q^{-2a}\left [\xi_a\sum_{n\geq0}\mu_{2k-4-4n}+\xi_b\sum_{n\geq0}\mu_{2k-6-4n}\right ]\\
&=q^{-2a}\left (\xi_b \sum_{n\geq 0} \mu_{2k-2-4n} +\xi_a \sum_{n\geq 0} \mu_{2k-4-4n} \right )\\
&=q^{-2a}\mu_{2k-1}.
\end{align*}
By (II) again, we have $$  \lambda_{i,j}=q^{-2a}\mu_{i-j-2}  ,$$ with $ i $ even, $ j $ odd and $ i-j\geq 3 $. If $ i-j=1 $ and $ i $ is even, then by (II) we have $$  \lambda_{i,i-1}=\xi_a  .$$ 

 Therefore, 
\[
\lambda_{i,j}=\begin{cases}
\mu_{i-j} &\text{ if } i+j \text{ is even, or } i \text{ is odd and } j \text{ is even},  \\
q^{-2a}\mu_{i-j-2}&\text{ if $ i $ is even, } j \text{ is odd, and } i-j\geq 3,\\
\xi_a&\text{ if } i  \text{ is even},  i-j=1.
\end{cases}
\]
Recall that 
\begin{equation}
f_{u,v,d}-p_{d,w_I}f_{u,v,w_I}= \tau(T_v(T_d-q^{-a}T_{w_I})T_u),
\end{equation} 
\begin{equation} 
T_sU_{2m-1}=(-q^{-a})U_{2m-1}=U_{2m-1}T_s .
\end{equation}
If $ s\in\mc{R}(v) $, $ v=v's $, then  
\[
f_{u,v,d}-q^{-a}f_{u,v,w_I}=(-q^{-a})(f_{u,v',d}-q^{-a}f_{u,v',w_I}).
\]
If $ s\in\mc{L}(u) $, $ u=su' $,   then  
\[
f_{u,v,d}-q^{-a}f_{u,v,w_I}=(-q^{-a})(f_{u',v,d}-q^{-a}f_{u',v,w_I}).
\]
By definition, for $ u=w(t,2m-1-j) $ and $ v=w(i,t) $, we have\[
f_{u,v,d}-q^{-a}f_{u,v,w_I}=\lambda_{i,j}.
\]
Using these results, one can determine $f_{u,v,d}-q^{-a}f_{u,v,w_I}  $ for any $ u,v\in W_I\setminus\{w_I \} $.


For $ i\geq j $, we have 
\[
\deg\lambda_{i,j}=\begin{cases}
L'_{i-j} &\text{ if } i+j \text{ is even, or } i \text{ is odd and } j \text{ is even},  \\
L'_{i-j-2}-2a&\text{ if $ i $ is even, } j \text{ is odd, and } i-j\geq 3,\\
a&\text{ if } i  \text{ is even},  i-j=1.
\end{cases}
\]
(Recall that $ \lambda_{i,j}=0 $ for $ i<j $.)

 Let $ u=w(t,2m-1-j) $, $ v=w(i,t) $, $ i\geq j $.
If $ i $ is even, $ j $ is odd, and $ i-j\geq 3 $, then \begin{align*}
L'(u)+L'(v)-L'(d)&=\frac{2m-1-j}{2}(b-a) +\frac{i}{2}(b-a)-b-(m-1)(b-a)\\
&=\frac{i-j+1}{2}(b-a)-b\\
&=L'_{i-j+1}-b\\
&=(L'_{i-j-2}-a+b-a)-b\\
&=L'_{i-j-2}-2a.
\end{align*}
If $ i $ is even, $ j $ is odd and $ i-j= 1 $, then \begin{align*}
L'(u)+L'(v)-L'(d)=-a
\end{align*}
If $ i $ is odd, $ j $ is even and $ i> j $, then \begin{align*}
L'(u)+L'(v)-L'(d)&=b+\frac{2m-1-j-1}{2}(b-a) \\
&\phantom{=}+b+\frac{i-1}{2}(b-a)-b-(m-1)(b-a)\\
&=\frac{i-j-1}{2}(b-a)+b\\
&=L'_{i-j}.
\end{align*}
Similarly, one can prove  $ L'(u)+L'(v)-L'(d)=L'_{i-j} $ when $ i+j$ is even and $ i\geq j $. 

Now the lemma follows.
\end{proof}

\begin{cor}\label{cor:minus2}
Assume $ L(s)\neq L(t) $.
For $ u,v\in W_I\setminus \{d_I,w_I\} $, we have\[
\deg (f_{u,v,d_I}-p_{d_I,w_I}f_{u,v,w_I})\leq L'(d_I)-2L'(st).
\]
\end{cor}
\begin{proof}
Without loss of generality, assume  $ a<b $.

If $ m_{st}\geq 6 $, by Lemma \ref{lem:Fuv},
the maximal degree is $ L'(d_I)-2L'(st) $ and is taken by $ u=v=w(t,2m-3) $ (note that we require $ u,v\neq d_I $). 

If $ m_{st}=4 $, the maximal degree is $ a = L'(d_I)-2L'(st) $ and is taken by $ u=v^{-1}=ts $.
\end{proof}

\begin{lem}\label{lem:gammaleq0}
Assume that $ a=L(s)<L(t)=b $. Let $ u,v\in W_I\setminus \{d_I,w_I\} , $ $$  z\in\{e,s,t,st,ts\} \text{ and }\gamma= \deg (f_{u,v,d_I}-p_{d_I,w_I}f_{u,v,w_I})p_{z,d_I} . $$ Then we are in one of the following situations.
\begin{itemize}
\item [(1)] $ \gamma\leq 0 $.
\item [(2)] $ \gamma\leq 2a-b >0$, $ z=t $, $ tu<u $ and $ vt<v $.
\end{itemize}
\end{lem}
\begin{proof}By Corollary \ref{cor:minus2} and Lemma \ref{lem:degd}, we have
\[
\gamma\leq L'(z)-2L'(st).
\]
Assume $ \gamma>0 $. By the above inequality, we must have $ z=t $. Hence $$  \gamma\leq 2a-b>0  .$$
Using (i) and (ii) of Lemma \ref{lem:Fuv} and the above inequality, if $ su<u $ or $ vs<v $, we will have \[
\gamma\leq 2a-b-a=a-b<0.
\]
Thus when $ \gamma>0 $, we must have
 $ tu<u $ and $ vt<v $.
\end{proof}

\begin{lem} \label{lem:deg1}
	Let  $ u,v\in W_I $.
	\begin{itemize}
		\item [(i)] If $ L(s)\neq L(t) $, then the possible values of $ f_{u,v,st} $ are $ \xi_s \xi_t$, $ \xi_s $, $ \xi_t $, $ 1 $, $ 0 $. Moreover,
		\begin{itemize}
			\item if $ f_{u,v,st}=\xi_s\xi_t $, then $ u=v=w_I $;
			\item if $ f_{u,v,st}=\xi_s $, then $ su<u $ and $ v=u^{-1}t$;
			\item if $ f_{u,v,st}=\xi_t $, then $ vt< v$ and $ u=sv^{-1} $.
		\end{itemize}
	\item [(ii)]If $ L(s)=L(t) $, then the possible values of $ f_{u,v,st} $ are $ \xi_s ^2$, $ \xi_s $, $ 1 $, $ 0 $. 
	Moreover,
	\begin{itemize}
		\item if $ f_{u,v,st}=\xi_s^2 $, then $ u=v=w_I $;
		\item if $ f_{u,v,st}=\xi_s $, then $ su<u $, $ v=u^{-1}t$, or $ vt< v$, $ u=sv^{-1} $.
	\end{itemize}
\item [(iii)]If $ r\in I $, then  the possible values of $ f_{u,v,r} $ are $ \xi_r $, $ 1 $, or $ 0 $. Moreover, if $ f_{u,v,r}=\xi_r $, then $ ru<u $ and $ vr<v $.
	\end{itemize}
\end{lem}
\begin{proof}
	We only prove (i) here. Other assertions are proved in a similar way.
	
	We have $ f_{u,v,st}=f_{ts,u,v^{-1}} $. The possible values of $ f_{u,v,st} $ immediately follows from computing $ T_{ts}T_{u} $.
	If $ f_{ts,u,v^{-1}} =\xi_s\xi_t  $, then $ su<u $ and $ tu<u $, i.e. $ \mc{L}(u)=I $. Hence $ u=v=w_I $.		
	If $ f_{ts,u,v^{-1}} =\xi_s   $, then $ su<u $ and $ tu=v^{-1} $. 	
	 If $  f_{ts,u,v^{-1}} =\xi_t   $, then $ tsu<su $ and $ su=v^{-1} $, which implies that $ u=sv^{-1}$ and $ vt<v $. This proves (i).
\end{proof}

\begin{lem} \label{lem:ppd}
	
	 Conjectures P1-P15 hold for the finite dihedral group $ W_I $.

		 If  $ L(s)=L(t) $, the two-sided cells of $ W_I $ are $ \{e\} $, $ W_I\setminus \{e, w_I\} $, $ \{w_I\} $, and  the corresponding $ \af $-values are  $ 0 $, $ L(s) $, $ L(w_I) $.
		 
		 If $ L(s)<L(t) $,   the two-sided cells of $ W_I $ are $ \{e\} $, $ \{s\} $, $ W\setminus\{e,s,d_I, w_I \} $, $ \{d_I\} $, $ \{w_I \}$, and  the corresponding  $ \af $-values are  $ 0 $, $ L(s) $, $ L(t) $, $ L'(d_I) $, $ L(w_I) $.

\end{lem}
\begin{proof}
	Two-sided cells are given in  \cite[8.8]{lusztig2003hecke}, and the $ \af $-values are given in \cite[13.11]{lusztig2003hecke}.
	Conjectures P1-P15 for $ (W_I,L) $ are proved in \cite[\S 15]{lusztig2003hecke} for equal parameters, and in \cite[Thm.5.3]{geck2011rank2} for unequal parameters.
\end{proof}

By Lemmas \ref{lem:wn} and \ref{lem:compute}, we have the following lemma.
\begin{lem}\label{lem:degI}
	For any integer $ N $ and  $ u,v\in (W_I)_{\leq N} $, we have $ \deg\nt_u\nt_v\leq N $ in $ (\mc{H}_{I})_{\leq N} $, and the equality holds only if $ u,v\in (W_I)_{N} $.
\end{lem}
The following proposition is based on Lemma \ref{lem:ppd}. It will be used in the proof of Proposition \ref{prop:key1}.
\begin{prop}\label{prop:key}
	Assume that $ N\in \mathbb{N}$,  $ u,v,z\in (W_I)_{\leq N} $, and $ z\in\{e,s,t,st,ts\} $. All the possible degrees of $ \nf_{u,v,z} $ are given as follows.
	\begin{itemize}
		\item [(i)]We have $ \deg \nf_{u,v,e}\leq 0 $.
		\item [(ii)]For $ r\in I $, $ \deg \nf_{u,v,r}\leq L(r) $, and the equality holds only if $ vr<v $ and $ ru<u $.
		\item [(iii)] Let   $ w=s_1s_2 $ with $ \{s_1,s_2\} =\{s,t\}$.
		Then we are in one of the following situations:
		\begin{itemize}
			\item [(1)]$ \deg \nf_{u,v,w} \leq 0$;
			\item [(2)] $ \deg \nf_{u,v,w}=L(w)$,   and $ u=v=w_I $;
			\item [(3)] $ \deg \nf_{u,v,w}=L(s_1) $, and $ s_1u<u $, $ v=u^{-1}s_2 $;
			\item [(4)] $ \deg \nf_{u,v,w}=L(s_2) $, and $ vs_2<v $, $ u=s_1v^{-1} $;
			\item [(5)] $ \deg\nf_{u,v,w}=|L(s_1)- L(s_2)|>0$,  and $ u=v=d_I $.
		\end{itemize}
	\end{itemize}
\end{prop}
\begin{proof}
	 Without loss of generality, we can assume that $ a=L(s)\leq b=L(t) $. Let $ z\in\{e,s,t,st,ts\} $.
According to $ (W_I)_{>N} $, the proof is divided into the following cases.

Case (I): $ (W_I)_{>N}=\emptyset $. Then we have $ \nf_{u,v,z}=f_{u,v,z} $, and the proposition follows from Lemma \ref{lem:deg1}.

Case (II): $ (W_I)_{>N}=\{w_I\} $. Then we have $ \nf_{u,v,z}=f_{u,v,z}-f_{u,v,w_I}p_{z,w_I} $. Let
$ \delta=\deg f_{u,v,w_I}p_{z,w_I} $. By Lemma \ref{lem:degfp2}, there are three situations as follows.
\begin{itemize}
\item [(1)] $ \delta\leq 0 $. Then $\nf_{u,v,z} =f_{u,v,z} \mod{\mathbb{Z}[q^{-1}]}$, and the proposition follows from Lemma \ref{lem:deg1}.
\item [(2)]  $ \delta=b-2a>0 $,  $ z=t $, and $ u=v=d_I$. We have $ f_{d_I,d_I,t}=\xi_t $, and $ \deg \nf_{d_I,d_I,t}=L(t) $, which is consistent with (ii).
\item [(3)]  $ \delta=b-a>0 $, $ z=st $ or $ ts $,  and $ u=v=d_I$. We have $ f_{d_I,d_I,z}=0 $, and $ \deg \nf_{d_I,d_I,z}=b-a $, which is case (iii,5).
\end{itemize}

Case (III): $ (W_I)_{>N}=\{d_I, w_I\} $. In this case, $ m_{st} $ is even, and $ L(s)=a<b=L(t) $.
We have \[
	\nf_{u,v,z} =f_{u,v,z}-f_{u,v,w_I}p_{z,w_I}-(f_{u,v,d_I}-f_{u,v,w_I}p_{d,w_I})p_{z,d_I}.
	\]
 Let $\gamma= \deg (f_{u,v,d_I}-p_{d_I,w_I}f_{u,v,w_I})p_{z,d_I} $. By Lemma \ref{lem:gammaleq0}, we have two situations as follows.
\begin{itemize}
\item [(1)] $ \gamma\leq 0 $. The third term does not affect the positive degree part of $ \nf_{u,v,z} $, and the proposition follows from case (II).
\item [(2)] $ \gamma\leq 2a-b >0$, $ z=t $, $ tu<u $ and $ vt<v $. Since $ 2a-b<b $, we have $ \deg \nf_{u,v,t}\leq L(t) $, which is consistent with (ii).
\end{itemize}

Case (IV): $ (W_I)_{\leq N}\subset\{e,s\} $. In this case, the proposition is obvious.
\end{proof}

\section{Basic properties of Coxeter groups with complete graph}
\label{sec:complete}

\begin{ass}
	In this section, $ (W,S) $ is a Coxeter group with complete graph, i.e. $ m_{st}\geq 3 $ for any $ s\neq t\in S $.
\end{ass}

By   \cite[Lem.2.2 and Lem. 2.6]{xi2012afunction}, we have the following lemma.  
\begin{lem}\label{lem:xi}
	Let  $ x\in W $ and $ s,t\in S $ such that $ xst=x\bd s\bd t  $.
	\begin{itemize}
		\item[(i)] We have $ \mc{R}(xst)=\{t\}$ or $ \{s,t\} $.
		\item [(ii)] If $ \mc{R}(xst)=\{t\}$, $ \mc{R}(xs)=\{s\}$ and $ \{s,t\}\cap\mc{L}(y) =\emptyset$, then \[
		xsty=x\bd s\bd t\bd y.
		\]
	\end{itemize}
\end{lem}

The following lemma is useful in this paper. Note that (ii) of the following lemma appears in  \cite[Lem. 2.5(1)]{Shi-Yang2016}. We give a proof here based on Lemma \ref{lem:xi}.

\begin{lem}\label{lem:shiyang}
	Let $ I=\{s,t\}\subset S $ such that $ s\neq t $, $ m_{st}<\infty $. Let  $ z\in W_I $ and $ x,y\in W $ such that $ \mc{R}(x)\cup\mc{L}(y)\subset S\setminus I $.
	\begin{itemize}
		\item [(i)] If $ l(z)\geq 3 $ then $xzy=x\bd z\bd y $.
		\item [(ii)] If $ l(z)=2 $ and $ z=st $, then $xzy<x\bd z \bd y $ if and only if there exists $ r\in S $ such  that $ \mc{R}(xs)=\{s,r\} $,  $ \mc{L}(ty)=\{t,r\} $.
		
		\item [(iii)]  Suppose that conditions in (ii) hold.  Then$$  T_{xz}T_y=\xi_r T_{xsrty}+T_{x_1}T_{z_1} T_{y_1}  $$ for some $ x_1,y_1,z_1 $ such that  $ z_1=rw_{tr} \in W_{I_1}$ with $ I_1 =\{ t,r \}$, $ l(y_1)<l(y) $, and $ \mc{R}(x_1)\cup\mc{L}(y_1)\subset S\setminus I_1  $.
		If  $ x_1z_1y_1<x_1\bd z_1\bd y_1 $, then $$  m_{st}=m_{sr}=m_{tr}=3  .$$  
We have
\begin{equation}\label{eq:inddeg}
		\deg T_{xz}T_y=L(r).
\end{equation}
	\end{itemize}
\end{lem}
\begin{proof}
	  See \cite[Lem.2.5]{xi2012afunction} for assertion (i). 
	
	We prove assertion (ii).
	Assume that $ xzy<x\bd z\bd y$. By Lemma \ref{lem:xi} (i), we have $ \mc{R}(xst)=\{t\} $ or $ \{s,t\} $. If the latter case happens, then $t\in \mc{R}(x)$ since $ m_{st}\geq 3 $, which contradicts with the assumption $ \mc{R}(x)\subset S\setminus I $. Thus $ \mc{R}(xst)=\{t\} $. By Lemma \ref{lem:xi} (ii), $ xzy<x\bd z\bd y $ implies that $ \mc{R}(xs)=\{s,r_1\} $ for some $ r_1\in S $. Similarly, $ y^{-1}z^{-1}x^{-1}<y^{-1}\bd z^{-1}\bd x^{-1} $ implies that $ \mc{L}(ty)=\{t,r_2\} $ for some $ r_2 \in S$.
	
	Let $ xs=x'w_{r_1s} $, $ ty= w_{tr_2}y'$. Assume that $ r_1\neq r_2 $. Then $ \mc{R}(xs)\cap\{t,r_2\}=\emptyset $. This implies that $ xsw_{tr_2}=(xs)\bd w_{tr_2}$. Then  by assertion (i), we have $$  l(xsty)=l(xsw_{tr_2}y') =l(xs)+l(w_{tr_2}y')=l(xs)+l(ty)=l(x)+l(y)+2 ,$$ which  contradicts with $ xzy<x\bd z\bd y  $. Hence $ r_1=r_2 $, denoted by $ r $. This proves the `` only if " part of (ii). The `` if " part is obvious.
	
	Now we  prove assertion (iii). We have\[
	T_{xz}T_y=\xi_r T_{x'(w_{rs}r)}T_{w_{tr}y'}+T_{x'(w_{rs}r)}T_{(rw_{tr})y'}.
	\]
	By Lemma \ref{lem:xi}(i), we have $ \mc{R}(x'(w_{rs}r))=\{s\} $, and hence $$  l(x'(w_{rs}r)w_{tr})=l(x'(w_{rs}r))+l(w_{tr}) . $$ Then by assertion (i), we have $$   l(x'(w_{rs}r)w_{tr}y')=l(x'(w_{rs}r))+l(w_{tr})+l(y') .$$ Hence the first term $ \xi_r T_{x'(w_{rs}r)}T_{w_{tr}y'} $ is equal to $ \xi_r T_{xsrty} $. Take $$  x_1=x'(w_{rs}r),\quad  z_1= rw_{tr},\quad y_1=y'  .$$ Then the second term is $ T_{x_1}T_{z_1}T_{y_1} $, $ l(y_1)<l(y) $, and $ \mc{R}(x_1)\cup\mc{L}(y_1)$ has no intersection with $ I_1=\{t,r\}  $.
	
	Assume that $ x_1z_1y_1<x_1\bd z_1\bd y_1$. Then by assertion (i) and (ii),  $ m_{tr}=3 $ and $ \mc{R}(x_1t)=\{t, r'\} $ and $ \mc{L}(ty_1)=\{r,r'\} $ for some $ r'\in S $. Similarly,  considering $ (x',w_{rs}r,(rw_{tr})y') $, we have $ m_{sr}=3 $. Since $ x_1=x'rs$, we have $ \mc{R}(x_1t)=\{t\} $ or $ \{t,s\} $. Hence  $ r'=s $ and $ \mc{R}(x_1t)= \{t,s\} $. If $ m_{st}\geq 4 $, then by Lemma \ref{lem:xi}(i) $ \mc{R}(x_1t)=\mc{R}(x'rst)=\{t,s\} $ implies that $  \mc{R}(x'r)=\{s,t\} $,  a contradiction with $ r\in \mc{R}(x'r) $. Hence $ m_{st}=3 $. 
	
	At last, we prove \eqref{eq:inddeg} using induction on  $ l(y) $. If $ x_1z_1y_1=x_1\bd z_1\bd y_1 $, then   $ \deg (T_{x_1}T_{z_1}T_{y_1})=0 $. If $ x_1z_1y_1<x_1\bd z_1\bd y_1 $, using the induction hypothesis
	 or $ \deg (T_{x_1}T_{z_1}T_{y_1})=L(s) =L(r)=L(t)$ since $ m_{st}=m_{sr}=3 $.
	 Thus we always have
	$$  \deg (T_{xz}T_y)=\max\{L(r),\deg (T_{x_1}T_{z_1}T_{y_1})\}
	=L(r) .  $$ 
\end{proof}


\section{Factorization formulas}\label{sec:main}
\begin{ass}\label{ass3}
In this section, $ (W,S) $ is a Coxeter group with complete graph, i.e. $ m_{st}\geq 3 $ for any $ s\neq t\in S $, and $ L: W\to\mathbb{Z} $ is a fixed  positive weight function. 
\end{ass}

The boundedness conjecture \ref{conj:bound} holds for Coxeter groups with complete graph, by \cite{xi2012afunction} and \cite{Shi-Yang2016}. In particular, $ \wnn=\emptyset $ for large enough $ N $, and $ \af(w)<\infty $ for any $ w\in W $.

Denote by $ D $  the set consisting of involutions \begin{itemize}
	\item[(i)] $ w_J $ for some $ J\subset S $, noting  $ |J|=0,1,2 $, and 
	\item [(ii)]  $ d=w(t,m_{st}-1) $ for some $ s,t\in S $ with   $ L(s)<L(t) $.
\end{itemize}
We define a function $ \af' : D\to\mathbb{N}$ such  that $ \af'(d)= L(d)$ in case (i) and $ \af'(d)= L'(d) $ in case (ii) (see \eqref{eq:l'} for the definition of $ L' $).  Define \[
D_{\geq N}=\{d\in D\mid \af'(d)\geq N \},
\]and  $ D_{>N}=D_{\geq N+1} $, $ D_{N} =D_{\geq N}\setminus D_{>N}$.  Define
\begin{equation}\label{eq:omegaN}
\Omega_{\geq N} =\left\{w\in W\, |\,  
w=x\bd d\bd y \text{ for some }d\in D_{\geq N}, x,y\in W
\right\},
\end{equation}
and  $ \Omega_{>N} =\Omega_{\geq N+1}$, $ \Omega_{N} =\Omega_{\geq N}\setminus\Omega_{>N}$. By definition  the following facts hold:
\begin{itemize}
	\item $ \Omega_{ \geq N} $ and $ \Omega_{ N} $ are stable by taking inverse;
	\item if $ w\in \Omega_{ \geq N} $, then for all $ u,v\in W $ such that $ uwv=u\bd w\bd v $ we have $ uwv\in\Omega_{ \geq N} $.
\end{itemize}



For any subset $ J\subset S $, we define  $  (\Omega_J)_{ N}$ etc. by replacing $ W $ by $ W_J$ in the above definitions. It is easy to verify  that \begin{equation*}
 (\Omega_J)_{\geq N}= \Omega_{\geq N}\cap W_J \text{ for any } N.
\end{equation*}
Hence \begin{equation}\label{eq:omegaj}
(\Omega_J)_{ N}= \Omega_{N}\cap W_J\text{ for any } N.
\end{equation}

 For  $ N=0 $,   $ W_N=\{e\} $ (see 13.7 in \cite{lusztig2003hecke}), $ \Omega_N=\{e\}  $,
and 
all the results in this section hold trivially.
Thus in the following we assume  $ N>0 $.

\begin{ass}\label{ass4}
Until Theorem \ref{thm:dec} (included), we fix an integer $ N>0 $ such that  $ \wnn=\Omega_{>N} $ and $ \wnn $ is $ \prec_{ {LR}} $ closed. 
\end{ass}
By this assumption, it is obvious that
\begin{equation}\label{eq:empty}
\wn\cap \Omega_{> N} =\emptyset,\text{ and } \wn=\Omega_{\leq N}.
\end{equation}


We claim that \begin{equation}\label{eqempty}
 D_N\cap \Omega_{> N}=\emptyset .
\end{equation}
Assume that $ d\in D_N\cap \Omega_{> N}$. Then $ d\in W_I $ for some finite dihedral subgroup $ W_I $, and there exists some $ d'\in D_{>N} , x,y\in W$ such that $ d=x\bd d'\bd y $. Then $ d=x\bd d'\bd y \in W_I$ implies that $ d'\in W_I $, using a result due to Matsumoto and Tits (see for example \cite[Thm. 1.9]{lusztig2003hecke}). 
Now the claim follows from the fact that the following cannot happen in a finite dihedral group:
\[
\af'(d)<\af'(d')\text{ and } l(d)\geq l(d').
\]

For $ d\in D_N $, we have $ \af(d)\geq N $, since $ \deg h_{d,d,d}=\af'(d) $ (Lemma \ref{lem:eta}). If $\af(d)> N  $, then $ d\in \wnn=\Omega_{>N} $. This  contradicts with claim \eqref{eqempty}. Thus $ \af(d)=N $, i.e.  $ D_N\subset W_N $, and $ \af':D_N\to \mathbb{N} $ is the restriction of $ \af: W\to \mathbb{N} $.

Assumption \ref{ass4} guarantees that we can use results from section \ref{subsec:basis}.  

If $ W_I $ is a finite dihedral subgroup of $ W $, then by the explicit cell partitions  of $ W_I $ and their $ \af $-values, we have
\begin{equation}\label{eq:woi}
(W_I)_{N}=(\Omega_I)_{N} \text{ for any  } N.
\end{equation}
Combined with \eqref{eq:omegaj} and \eqref{eq:empty}, $$  (W_I)_{\leq N}=(\Omega_I)_{\leq N} =\Omega_{\leq N}\cap W_I=W_{\leq N} \cap W_I . $$
Then by Lemma \ref{lem:nfxyz}, for $ u,v,w\in W_I $,  $ \nf_{u,v,w} $ that computed in $ (\mc{H}_I)_{\leq N} $ coincides with that computed in $ \mc{H}_{\leq N} $. In particular, we can apply results in section \ref{sec:rank2}  about $ \nf_{u,v,w} $. In the following we  use this without  mention  it.

\subsection{Degrees of products}\label{subsec:degest}

By Lemma \ref{lem:wn},  for all $ x,y\in\wn $, we have \begin{equation} 
\deg\nt_x\nt_y\leq N,
\end{equation}
see also \eqref{eq:strict} and \eqref{eq:strict1}.
The following proposition gives a necessary condition for the  equality $ \deg\nt_x\nt_y= N$.
\begin{prop}\label{prop:key1} 
Keep Assumptions \ref{ass3} and \ref{ass4}.
	For  $ x,y\in\wn $, the equality $ \deg\nt_x\nt_y= N $ holds only if $x, y\in \Omega_{ N} $.
\end{prop}
\begin{proof}
	If $  \deg\nt_x\nt_y=N $ implies $y\in \Omega_{\geq N}   $, then it also implies $x\in \Omega_{\geq N}   $ since $ \deg\nt_x\nt_y=\deg\nt_{y^{-1}}\nt_{x^{-1}} $. If we know $ x,y\in\Omega_{\geq N} $, then $ x,y\in\Omega_{ N} $ since $ x,y\in \wn=\Omega_{ \leq N} $.
	Thus it suffices to  prove that  $ y\in\Omega_{\geq N} $ is a necessary condition of the equality.
	
	We prove the proposition by induction on the length of $ y $. If $ l(y)=0 $, the proposition is obvious. If $ l(y)=1 $, then $ y=r $ for some $ r\in S $. Then \[
	T_xT_y=\begin{cases}
	\xi_r T_x+T_{xr}&\text{ if }xr<x;\\
	T_{xr} &\text{ if }xr>x.
	\end{cases}
	\]
Thus $ \deg\nt_x\nt_y\leq L(r) $. Note that $ L(r)\leq N $, since $ y=r\in \wn=\Omega_{\leq N}$. If $ \deg\nt_x\nt_y=N $, then $ y\in \Omega_{N} $. The proposition follows in this case.   

Assume now that $ l(y)\geq 2 $ and the proposition has been proved for all $ y' \in\wn$ such that $ l(y')<l(y) $.

If $ r\in \mc{L}(y)\setminus\mc{R}(x) $, then $ \nt_x\nt_y=\nt_{xr}\nt_{ry} $.
By the induction hypothesis and \eqref{eq:strict1}, the equality  $ \deg \nt_x\nt_y=N $ holds only if $ ry\in \Omega_{\geq N} $, which implies that $ y=r(ry)\in \Omega_{\geq N} $. The proposition follows in this case.  

If $ \mc{L}(y)\subset\mc{R}(x) $ and $ r\in \mc{L}(y) $ is such that $ xry=x\bd (ry) $, then $$  \nt_x\nt_y=\nt_{xr}\nt_{ry}+\xi_r \nt_{xry}  .$$
Since $ y=r(ry) \in\wn=\Omega_{\leq N}$, we have $ L(r)\leq N $. Hence $$ \deg \xi_r \nt_{xry}  \leq L(r)\leq N . $$
By \eqref{eq:strict}, $\deg \nt_{xr}\nt_{ry}\leq N $; the equality holds only if $ ry\in\Omega_{\geq N} $ by the inductive hypothesis. Hence the equality $ \deg\nt_x\nt_y= N $  holds only if $ ry\in\Omega_{\geq N} $ or $ L(r)=N $, which implies that $ y\in\Omega_{\geq N} $.

It remains to deal with the case where $  \mc{L}(y)\subset\mc{R}(x)  $ and   $ xry<x\bd (ry) $ for any $ r\in \mc{L}(y) $. 
In this case, for a   reduced expression $ y=t_1 t_2\cdots t_k $  of $ y $, we have  $ xt_1<x$, and we can find an integer $ m \geq 2$ such that $ l(xt_2\cdots t_{m-1})=l(x)+m-2 $ and $ xt_2\cdots t_m<xt_2\cdots t_{m-1}$. We can assume such an $ m $ is minimal among similar integers for all reduced expressions of $ y $. Then
by \cite[Lem. 2.3]{xi2012afunction}, $ t_1 t_2\cdots t_m $ is in a finite parabolic subgroup of $ W $. In particular,   $ y=t_1 t_2\cdots t_k $ is  a   reduced expression of $ y $ such that $ m_{t_1t_2}<\infty $. Let $ I=\{t_1,t_2\} $.

Write $ x=x_1u $, $ y=vy_1 $ with $ x_1 $ minimal in $ xW_I $, $ y_1 $ minimal in $ W_Iy $, and $ u,v\in W_I $. Then \[
\nt_x\nt_y=\sum_{w\in {(W_I)}_{\leq N}}\nf_{u,v,w}\nt_{x_1}\nt_w\nt_{y_1}.
\]

In the following,  we  prove that  $$  \deg ( \nf_{u,v,w}\nt_{x_1}\nt_w\nt_{y_1})\leq N $$  and the equality holds only if $ y\in\Omega_{\geq N} $. This will complete the proof.

\textbf{Case (i)}: $ l(w)\geq 3 $. 

By Lemma \ref{lem:shiyang} (i), we have $ T_{x_1w}T_{y_1}=T_{x_1wy_1} $. 
Then by Lemma \ref{lem:degI}, \eqref{eq:woi} and \eqref{eq:omegaj},
  $$  \deg(\nf_{u,v,w}\nt_{x_1w}\nt_{y_1})\leq \deg \nf_{u,v,w}\leq N, $$  and the equalities hold only if $ v\in (W_I)_{N}= (\Omega_I)_{N}\subset\Omega_{N}$,  which implies that  $ y=vy_1\in\Omega_{\geq N} $.

\textbf{Case (ii)}: $ l(w)=2 $.

 Let $ w=s_1s_2 $ with $ \{s_1,s_2\}=I $. If $ x_1wy_1=x_1\bd w\bd y_1 $, then we can use the same method as case (i). In the following we assume that $ x_1wy_1<x_1\bd w\bd y_1$. By Lemma \ref{lem:shiyang} (ii) and (iii), we have $ \mc{R}(x_1s_1)=\{s_1,r \} $, $ \mc{L}(s_2y_1)=\{s_2,r\} $ for some $ r\in S \setminus I$ and  \[\deg T_{x_1w}T_{y_1}=L(r) .\]

According to Proposition \ref{prop:key} (iii), there are 5 possible cases about  $\deg \nf_{u,v,w} $ as follows:

\begin{itemize}
\item [(1):] $ \deg \nf_{u,v,w} \leq 0$. 
In this case we have $ \deg ( \nf_{u,v,w}\nt_{x_1}\nt_w\nt_{y_1})\leq L(r) $. Since $ \mc{L}(s_2y_1)=\{s_2,r\} $, $ m_{rs_2}\geq 3 $,  then $ r $ appears in a reduced expression of $ y_1 $. Then $ y\in \wn =\Omega_{\leq N}$ implies that $ L(r)\leq N $. Hence $ \deg ( \nf_{u,v,w}\nt_{x_1}\nt_w\nt_{y_1})\leq N$, and the equality holds only if $ L(r)=N $, which implies that $ y\in \Omega_{\geq N} $.

\item [(2):] $ \deg \nf_{u,v,w}=L(w)$,  and $ u=v=w_I $. 
Then $ y=w_Iy_1=z_1\bd w_{rs_2}\bd z_2 $ for some $ z_1,z_2\in W $, since $ \mc{L}(s_2y)=\{s_2,r\} $. Then $ y\in \wn =\Omega_{\leq N}$ implies that \begin{align}\label{eq:3r1}
  L(w_I)\leq N, \quad L(w_{rs_2})\leq N  .
\end{align} Similarly $ x\in\wn $ implies that  \begin{align}\label{eq:3r2}
L(w_{rs_1})\leq N  .
\end{align} 
Then we have
\begin{equation}\label{eq:half}
\begin{aligned}
L(s_1)+\frac12L(s_2)&\leq \frac12 N,\\
 L(s_2)+\frac12L(r)&\leq \frac12 N,\\
 L(r)+\frac12L(s_1)&\leq \frac12 N.
\end{aligned}
\end{equation}
Then
\[
\deg(\nf_{u,v,w}\nt_{x_1w}\nt_{y_1})\leq L(s_1)+L(s_2)+L(r)\leq N.
\]
If $ \deg(\nf_{u,v,w}\nt_{x_1w}\nt_{y_1})=N $, then all the equalities in \eqref{eq:half} holds. This implies that $ m_{s_1s_2}=m_{rs_1}=m_{rs_2}=3 $, $ w_I\in D_N $, and hence $ y\in\Omega_{\geq N} $.

\item [(3):] $ \deg \nf_{u,v,w}=L(s_1)  $, and $ s_1u<u $. Since $ \mc{R}(x_1s_1)=\{s_1,r\}$, we have $ x=x_1u=x_1s_1(s_1u)=z\bd w_{rs_1}\bd (s_1u) $  for some $ z $. Then $ x\in \Omega_{\leq N} $ implies that $ L(w_{rs_1}) \leq N$. Thus \[
\deg(\nf_{u,v,w}\nt_{x_1w}\nt_{y_1})\leq L(s_1)+L(r)< N.
\]
Note that the last inequality is strict.

\item [(4):] $ \deg \nf_{u,v,w}=L(s_2) $, and $ vs_2<v $. The proof is  similar to  (3).


\item [(5):] $ \deg\nf_{u,v,w}=|L(s_1)- L(s_2)|>0$,  and $ u=v=d_I $. If $ L(s_1)>L(s_2) $, then $ s_1u<u $ since $u= d_I $. By the same reason as case (3), we have
\[
\deg(\nf_{u,v,w}\nt_{x_1w}\nt_{y_1})\leq L(s_1)-L(s_2)+L(r)< L(s_1)+L(r)<N.
\]
If $ L(s_1)<L(s_2) $, then $ vs_2<v $. By the same reason as case (4), we have
\[
\deg(\nf_{u,v,w}\nt_{x_1w}\nt_{y_1})\leq L(s_2)-L(s_1)+L(r)< L(s_2)+L(r)<N.
\]
\end{itemize}
This completes the proof of case (ii).

\medspace

\textbf{Case (iii)}: $ l(w)=1 $, i.e.  $ w=r $ for some $ r\in S $. We have
\begin{equation}\label{eq:tr}
\nt_{x_1r}\nt_{ry_1}=\xi_r\nt_{x_1r}\nt_{y_1}+\nt_{x_1}\nt_{y_1}.
\end{equation}
By claim \eqref{eq:strict}, we have 
\begin{equation}\label{eq:twoleq}
\deg (\nt_{x_1r}\nt_{ry_1})\leq N,\quad \deg (\nt_{x_1}\nt_{y_1})\leq N,
\end{equation}
and  \[
\deg(\xi_r\nt_{x_1r}\nt_{y_1})\leq \max\{\deg (\nt_{x_1r}\nt_{ry_1}), \deg (\nt_{x_1}\nt_{y_1}) \}\leq N.
\]
By Proposition \ref{prop:key} (ii), we have \[
\deg (\nf_{u,v,r}\nt_{x_1r}\nt_{y_1})\leq \deg (\xi_r\nt_{x_1r}\nt_{y_1})\leq N.
\]
Hence $ \deg (\nf_{u,v,r}\nt_{x_1r}\nt_{y_1})= N $ occurs only if 
\begin{itemize}
	\item $ \deg \nf_{u,v,r}=L(r) $, which implies that $ vr<v $, and 
	\item  $ \deg (\nt_{x_1r}\nt_{ry_1})= N $ or $ \deg (\nt_{x_1}\nt_{y_1})= N $, which implies that $$  y_1\in \Omega_{\geq N} \text{ or }ry_1 \in\Omega_{\geq N} $$ by  \eqref{eq:strict1} and the induction hypothesis (noting that  $ l(y_1)\leq l(y)-2 $).
\end{itemize}
Thus if $ \deg (\nf_{u,v,r}\nt_{x_1r}\nt_{y_1})=N $, then $ y=vy_1=(vr)(ry_1)\in\Omega_{\geq N} $.

\textbf{Case (iv)}: $ l(w)=0 $, i.e $ w=e $. By Proposition \ref{prop:key} (i) and \eqref{eq:strict},
\[
\deg (\nf_{u,v,e}\nt_{x_1}\nt_{y_1})\leq \deg(\nt_{x_1}\nt_{y_1})\leq N,
\]
and by  the induction hypothesis, the equality $ \deg (\nf_{u,v,e}\nt_{x_1}\nt_{y_1})=N $ holds only if $ y_1\in\Omega_{\geq N} $, which implies that $ y=vy_1\in\Omega_{\geq N} $.
\end{proof}

For $ d\in D_N $, define  
   \begin{equation}\label{eq:ud}
U_d=\{y\in W\mid dy=d\bd y\in \Omega_{N}\},
\end{equation}
\begin{equation}\label{eq:bd}
B_d=\{b\in U_d^{-1} \,|\,
\text{if } bd=w\bd v \text{ with }  v\neq e,\text{ then }  w\in \Omega_{<N}  \}.
\end{equation}
In particular, $  B_d\subset U_d^{-1} $, $ B_dd\subset \Omega_{N} $ and $ dU_d\subset \Omega_{N} $.

\begin{lem}\label{lem:bdy-basic}
\begin{itemize}
	\item [(i)] Let $ d\in D_N $. For any $ y\in U_d $, $ w\leq d $, we have \begin{equation}\label{eq:ddd}
	wy=w\bd y.
	\end{equation} 
	\item [(ii)]  For any $ w\in \Omega_{ N} $, there exist $ d\in D_N $, $ b\in B_d $, $ y\in U_d $ such that $ w=b\bd d\bd y $. In particular, \begin{equation}\label{eq:omega}
	\quad\Omega_{ N}\subset \bigcup_{d\in D_N}B_d d U_d.
	\end{equation}
\end{itemize}
\end{lem}
\begin{proof}  
	
We prove (i).  If $ d=w_J $ for some $ J\subset S $, assertion (i) is well known. If $ d=w(t,2m-1) $ for some $ \{s,t\} \subset S$  with $ m_{s,t}=2m $ and $ L(s)<L(t) $, then $ ty>y$ is obvious, and $ sy>y $ also holds (otherwise $ sy<y $ and  $ dy=(w_{st})\bd (sy)\in \Omega_{>N} $ since $ w_{st}\in D_{>N} $,  which is a contradiction). Thus $ y $ is of minimal length in $ W_Iy $ with $ I=\{s,t\} $. This proves \eqref{eq:ddd}.

We  prove (ii) by induction on $ l(w) $.
Since  $w \in \Omega_N  $,  by definition we have  $ w=x_1\bd d_1\bd y_1 $  for some $ d_1\in D_N $, $ x_1\in U_{d_1}^{-1} $, $ y_1\in U_{d_1} $.    
If $ w $ is of minimal length in $ \Omega_N $, then $  w=d_1  $ since $ x_1d_1 $ and $ d_1y_1 $ are in $ \Omega_{ N} $. If $ x_1\in B_{d_1} $, then we are done.
If $ x_1\notin B_{d_1} $, then by   definition   we can find 
$ w_2\in \Omega_{\geq N} $ and $ v_2\neq e $ such that $ x_1d_1=w_2\bd v_2 $. Note that $ w_2 \in \Omega_{ N}$ and $ l(w_2)<l(w) $.
Using the induction hypothesis, we have $ w_2=b\bd d\bd y_2  $ for some $ d\in D_N $, $ b\in B_d $, $ y_2\in U_d $. Taking  $ y=y_2v_2y_1 $, we have
$ w=bdy_2v_2y_1=b\bd d\bd y $ and $ y\in U_d $. This proves (ii).
\end{proof}

\begin{lem}\label{lem:length}
	For any $ d\in D_N $, $ b\in B_d $, $ y\in U_d $, we have \[
bdy=b\bd d \bd y.
	\]
\end{lem}
\begin{proof}
	Note that $ l(d)\geq 3 $ or $ l(d)=1 $. If $ l(d)\geq 3 $, then it follows from Lemma \ref{lem:shiyang}(i). In the following, we take $ d =r\in S$. 
	
	We have $ r,br,ry\in\Omega_{ N} $ and $ b\in\Omega_{<N} $.  Then $ \mc{R}(br)=\{r\} =\mc{L}(ry)$; otherwise, $ br $ or $ ry \in \Omega_{>N}$.
	In other words, any reduced expression of $ br $ (resp. $ ry $) ends (resp. begins) with $ r $.
	
	Assume that $ 	bry<b\bd r\bd y $. Then we can find reduced expressions $ b=s_ps_{p-1}\cdots s_1 $, $ y=t_1t_2\cdots t_q $ and $  i,j\geq 1 $ such that \[
	s_is_{i-1}\cdots s_1r t_1t_2\cdots t_j=w_{rs}, \quad i+j+1=m_{rs}
	\]
	for some  $ s\in S$ with $3\leq  m_{rs}<\infty $.
	Since $ b\in\Omega_{<N} $, we have $ L(s)<L(r) $ and  $ i=1 $.   These imply that $$  ry=w(r,m_{rs}-1)\bd z  $$ for some $ z\in W $. 
 Since $ w(r,m_{rs}-1)\in D_{>N} $, we have $ry\in \Omega_{>N} $, a contradiction. Hence $bry=b\bd r\bd y$.
\end{proof}

\begin{prop}\label{prop:leq}
	If  $ d\in D_N $, $ x\in U_d^{-1} $, $ y\in U_d $, $ w\leq d $, then
\begin{equation}\label{eq:leq}
	\deg(\nt_{xw}\nt_y) \leq-\deg p_{w,d}.
\end{equation}
If moreover $ b\in B_d $, $ w<d $, then
	\begin{equation}\label{eq:lin}
\deg(\nt_{bw}\nt_y) <-\deg p_{w,d}.
\end{equation}
\end{prop}
\begin{proof}
	\textbf{Case (i)}: $ l(w)\geq3 $. We have $ T_{xw}T_y=T_{xwy} $ by Lemma \ref{lem:shiyang} (i), and hence $ \deg(\nt_{xw}\nt_y)\leq 0 \leq-\deg p_{w,d}$, i.e. \eqref{eq:leq} holds. If moreover $ w<d $, then $ \deg p_{w,d}<0 $, and \eqref{eq:lin} holds.
	
	\textbf{Case (ii)}: $l(w)=2  $. Let $ w=st $ with $ I=\{s,t\}\subset S $. By the proof of case (i),  we only need to consider the case of  $ xwy<x\bd w\bd y$. In this situation, we have $ \mc{R}(xs)=\{s,r\} $, $ \mc{L}(ty)=\{t,r\} $ for some $ r\in S\setminus I $, and \(
\deg (T_{xw}T_y)=L(r)
	\). 
	
	Now the proof is divided into the following two cases.
	
	\begin{itemize}
\item [(1):] $ d=w_I $. Since $ d=w_I\in D_N $, we have  \begin{equation}\label{eq:str1}
	 \frac12L(s)+L(t)\leq \frac12N  
	\end{equation} 
Since $ \mc{R}(xs)=\{s,r\} $ and $ xw_I\in\Omega_{ N} $, we have $ L(w_{s,r}) \leq N$, and in particular, 
	\begin{equation}\label{eq:str2}
		\frac12L(s)+L(r)\leq \frac12N .
	\end{equation} If moreover $ x\in B_d $, then $ xs\in \Omega_{<N} $, and hence \begin{equation}\label{eq:str3}
		\frac12L(s)+L(r)<\frac12 N .
	\end{equation} 
By \eqref{eq:str1} and \eqref{eq:str2}, we have $ L(s)+L(t)+L(r)\leq N $, or equivalently
\begin{equation}\label{eq:lr}
	L(r)\leq N-L(s)-L(t)=L(d)-L(w)=-\deg p_{w,d}.
\end{equation}
If moreover $ b\in B_d $, then by \eqref{eq:str3} the above inequality is strict. The proposition follows in this case.


	\item [(2):] $ d=d_I \in W_I$, $ I=\{s_1,s_2\}=\{s,t\}$ with  $ L(s_1)<L(s_2) $. Since $ s_2d<d $, $ ds_2<d $, $ s_2\in\{s,t\} $, $ \mc{R}(xs)=\{s,r\} $, $ \mc{L}(ty)=\{t,r\} $ and $ xd,dy\in\Omega_{ N} $, we have $ L(w_{s_2,r})\leq N $. Hence $ L(r)+L(s_2)-L(s_1)<N $, which is equivalent to\[
L(r)<N-(L(s_2)-L(s_1))=L'(d)-L'(w)=-\deg p_{w,d}.
\]
\end{itemize}
	
	\textbf{Case (iii)}: $ l(w)=1 $, i.e. $ w=r $ for some $ r\in S $. By claim \eqref{eq:strict}, we have $ \deg(\nt_{xr}\nt_{ry})\leq N $ and $ \deg(\nt_{x}\nt_{ry})\leq N $. By \eqref{eq:tr}, we have $$  	\deg(\xi_r\nt_{xr}\nt_y) \leq N  .$$ This implies $$  	\deg(\nt_{xr}\nt_y) \leq N -L(r)\leq -\deg p_{r,d} .$$

Assume  $ x\in B_d $, $ w<d $.  If $ rd<d $, then $ x,xr\in \Omega_{<N} $. By Proposition \ref{prop:key1},  all the inequalities in the last paragraph are strict. If $ rd>d $, then $ d=d_I $ for some $ I=\{r,r'\}\subset S $ with $ L(r)<L(r') $, in which case we have $$  -\deg p_{r,d} =L'(d_I)-L'(r)=N+L(r)>\deg(\nt_{xr}\nt_y) .$$
	
	\textbf{Case (iv)}: $ l(w)=0 $. By Proposition \ref{prop:key1}, we have $ \deg (\nt_{x}\nt_{y}) \leq N$; if moreover $ x\in B_d $,  the inequality is strict since  $ x\in\Omega_{<N} $.
\end{proof}

\subsection{One-sided factorization}\label{subsec:onesided}
\begin{prop}\label{prop:predec}
	Let $ d\in D_N $, $ y\in U_d $. Then\[
	\nc_{dy}=\nc_d\nF_y
	\]
for some $ \nF_y=\sum_{y'\in U_d} g_{y',y} \nt_{y'}\in \hn $ such that \begin{itemize}
	\item $ g_{y',y}\neq 0 $ only if $ y'\leq y $, 
	\item $ g_{y,y}=1 $, 
	\item and $ \deg g_{y',y}<0 $ for $ y'<y $.
\end{itemize}

Similarly, for $ x\in U_{d}^{-1} $, we have 
\[
	\nc_{xd}=\nE_x\nc_d,
\]
where $ \nE_x =(\nF_{x^{-1}})^\flat$ and $ {\cdot}^\flat $ is an $ \mc{A} $-linear \textit{anti}-involution of $ \hn $ such that $ (\nt_z)^\flat= \nt_{z^{-1}}$. 
\end{prop}

This proposition will follow from Lemma \ref{lem:predec1}.


\begin{lem}\label{lem:independent1}
	The elements $ \nc_d\nt_y,  y\in U_d $ of $ \hn $ are $ \mc{A} $-linearly independent.
\end{lem}
\begin{proof}
	Assume $$  \sum_{y\in U_d } a_y \nc_d\nt_y=0 $$ with $ y_0 $  the maximal element in $ U_d $ such that $ a_{y_0}\neq0 $. Then we have \[
	a_{y_0}\nt_{dy_0}+\sum_{z<dy_0}b_z\nt_z=0
	\] for some $ b_z \in \mc{A}$. Since  $ \nt_z\in \bigoplus_{\substack{z'\leq z\\z'\in\wn}}\mc{A}\nt_{z'} $ (using \eqref{eq:minus}), we have
	\[
	a_{y_0}\nt_{dy_0}+\sum_{\substack{z'<dy_0\\z'\in\wn}}b_{z}'\nt_{z'}=0
	\]  for some $ b'_z \in \mc{A}$.
	By Lemma \ref{lem:independent}, we have $ a_{y_0} =0$, a contradiction. This proves the lemma.
\end{proof}

Let $$  Y_d=\{ w\in W\mid vw=v\bd w\text{ for any }v\leq d\} . $$  For any $ w\in Y_d $, we have \begin{equation}\label{eq:mod1}
\nc_d\nt_w\equiv \nt_{dw}\mod{(\hn)_{<0}}  .
\end{equation} Note that $ U_d\subset Y_d $, see \eqref{eq:ddd}.


\begin{lem}\label{lem:predec1}
	Let $ w\in Y_d $. There exists a unique element $ \nF_w\in \hn $ such that  $ \nc_d\nF_w $ is bar invariant, and   \begin{equation}\label{key}
	\nF_w =\nt_w+\sum_{\substack{y<w\\y\in U_d}}g_{y,w} \nt_y.
\end{equation}
for some $ g_{y,w}\in\mc{A}_{<0} $.
\end{lem}

In this paragraph, we assume that the above lemma holds for a fixed $ w\in Y_d $. By  \eqref{eq:mod1},  $ \nc_d \nF_w \equiv \nt_{dw} \mod ({\hn})_{<0}$. If $ w\in U_d $, then by Lemma \ref{lem:char}, we have $ \nc_{dw}=\nc_d\nF_w $. Thus this lemma implies Proposition \ref{prop:predec}.
If   $ dw\in\Omega_{>N} =\wnn$, we actually have $ \nc_d \nF_w \equiv \nt_{dw}\equiv 0 \mod ({\hn})_{<0} $ (see \eqref{eq:degree}). By \eqref{claim:uniqueness}, we have $ \nc_d\nF_w =0$. In particular, if $ dw\in\Omega_{>N}  $,
\begin{equation}\label{eq:dw}
\nc_d\nt_w\in \bigoplus_{\substack{y<w\\y\in U_d}}\mc{A}_{<0} \nc_d \nt_y.
\end{equation} 

\begin{proof}[Proof of Lemma \ref{lem:predec1}]
	We prove it by induction on $ l(w) $. It is obvious for $ w=e $.  Fix a $ w\in Y_d $ with $ l(w)>1 $ and assume that  we have proved the lemma for all $ w' \in Y_d$ such that $ l(w')<l(w) $. 
	
It is well-known that $ \overline{T_u}=T_u+\sum_{z<u} R_{z,u}T_z $ for some $ R_{z,u}\in\mc{A} $. 
Using \eqref{eq:sdi} and  \eqref{eq:dw}, for any $ u\in Y_d $ with $ u\leq w $, we have
\begin{equation}\label{eq:bar}
	\overline{\nc_d\nt_u}= \nc_d\nt_u+\bigoplus_{\substack{y<u\\y\in U_d}} r_{y,u} \nc_d\nt_y \text{ for some }r_{y,u}\in\mc{A}.
\end{equation}
	Then it is a routine to prove that there are unique $ g_{y,w} \in \mc{A}$ with $ y\in U_d $, $ y< w $ such that  $ \nc_d\nF_w $ is bar invariant and \begin{equation}
	\nF_w =\nt_w+\sum_{\substack{y<w\\y\in U_d}}g_{y,w} \nt_y;
	\end{equation}
	see for example \cite[Theorem 5.2]{lusztig2003hecke}. 
For this, we need to use  Lemma \ref{lem:independent1} and  \eqref{eq:bar}. This completes the proof of this lemma and Proposition \ref{prop:predec}.
\end{proof}

\begin{cor}\label{cor:right}
	For any  $ d\in D_N $, we have \[
	\nc_d\hn=\bigoplus_{y\in U_d}\mc{A}\nc_{d}\nt_y= \bigoplus_{y\in U_d}\mc{A}\nc_{dy} ,
	\]
	Hence $ \{\nc_d\nt_y\mid y\in U_d \} $, $ \{\nc_{dy}\mid y\in U_d \} $ are  two $ \mc{A} $-basis of the right ideal of $ \hn $ generated by $ \nc_d $.
\end{cor}
\begin{proof}
	For any $ z\in W $,
	\begin{align*}
	\nc_d\nt_z&\in \mc{A}\nc_d\nt_w \text{ for some } w\in Y_d\\
	&\subset\bigoplus_{y\in U_d}\mc{A}\nc_d\nt_y \text{ by \eqref{eq:dw}}\\
	& = \bigoplus_{y\in U_d}\mc{A}\nc_{dy} \text{ by Proposition \ref{prop:predec}}.
	\end{align*}
	Thus $$  \nc_d\hn\subset\bigoplus_{y\in U_d}\mc{A}\nc_{d}\nt_y= \bigoplus_{y\in U_d}\mc{A}\nc_{dy}   .$$ It is obvious that $ \nc_d\hn\supseteq\bigoplus_{y\in U_d}\mc{A}\nc_{d}\nt_y$. Now the corollary follows.
\end{proof}

\subsection{Factorization formula and its corollaries}

For $ d\in D_N $, let $ \eta_d=h_{d,d,d} $. By Lemma \ref{lem:eta}, $ \deg \eta_d=N $.

\begin{thm}[Factorization formula]\label{thm:dec}
	For $ b\in B_d $, $ d\in D_N $, $ y\in U_d $, we have
	\begin{equation}\label{eq:dec}
	\nc_{bdy}=\nE_b\nc_d\nF_y \text{ in }\hn.
	\end{equation}
Hence $ \nc_{bd}\nc_{dy} =\eta_d \nc_{bdy}$.
\end{thm}

\begin{proof} We have

\begin{equation}
\nE_b \nc_d \nF_y=\sum_{\substack{x'\leq b,\,y'\leq y\\x'\in U_d^{-1},y'\in U_d \\w\leq d}} g_{x',b}p_{w,d}g_{y',y}\nt_{x'}\nt_{w}\nt_{y'}.
\end{equation}
If $ x'\neq b $ or $ y'\neq y $, then $ \deg (g_{x',b}g_{y',y})<0 $, and  $ \deg(p_{w,d}\nt_{x'}\nt_{w}\nt_{y'}) \leq 0$ by  \eqref{eq:leq}.
Hence if $ x'\neq b $ or $ y'\neq y $, we have$$  \deg (g_{x',b}p_{w,d}g_{y',y}\nt_{x'}\nt_{w}\nt_{y'}) <0 .$$ 
 If $ x'=b,y'=y $ and $ w\neq d $, by  \eqref{eq:lin} we have $ \deg(p_{w,d}\nt_{x'}\nt_{w}\nt_{y'}) < 0$ and  $$  \deg (g_{x',b}p_{w,d}g_{y',y}\nt_{x'}\nt_{w}\nt_{y'}) <0 .$$ 
Hence \begin{align*}
\nE_b \nc_d \nF_y
&\equiv \nt_{b}\nt_d\nt_y \mod{(\hn)_{<0}}\\
&\equiv \nt_{bdy}  \mod{(\hn)_{<0}}\quad 
\end{align*}
by Lemma \ref{lem:length}.
By Proposition \ref{prop:predec},  $$ \nc_{bd}\nc_{dy}=E_b \nc_d\nc_d \nF_y=\eta_d \nE_b \nc_d \nF_y  .$$ Therefore, $ \nE_b \nc_d \nF_y $ is bar invariant. If $ bdy\in\wn $, then by Lemma \ref{lem:char} we have $\nE_b\nc_d\nF_y= \nc_{bdy}$. If $ bdy\in \wnn $, we have $ \nt_{bdy}\equiv 0 \mod{(\hn)_{<0}}$ (see \eqref{eq:degree}),  and by claim \eqref{claim:uniqueness}  we have $\nE_b\nc_d\nF_y=0= \nc_{bdy}$. This completes the proof. (In the next theorem, we will see that $ bdy $ is alway in $ \wn $.)
\end{proof}

\begin{ass}\label{ass5}
In the rest of  section \ref{sec:main}, we fix an integer $ N>0 $ such that  $ \wnn=\Omega_{>N} $ and  $ (\text{P1,P4,P8})_{>N}  $ hold. 
\end{ass}
Note that this assumption  implies Assumption \ref{ass4}.

%


\begin{thm}
\label{thm:cor}
The following properties hold.
\begin{itemize}
\item [(i)] For $ d\in D_N $, $ b\in B_d $, we have $ \nc_{bd}\hn= \bigoplus_{y\in U_d}\mc{A}\nc_{bdy}  $.
\item [(ii)]For $ d\in D_N $, $ b\in B_d $, the subset $ \Phi_{b,d}:=bdU_d $ (resp. $ \Gamma_{b,d}:=\Phi_{b,d}^{-1} $) is a right (resp. left) cell of $ W $.
\item 	[(iii)]We have $W_N= \Omega_{ N}=\bigsqcup_{\substack{d\in D_N\\b\in B_d}}\Phi_{b,d} $.  In particular, the set of right (resp. left) cells in $ W_N $ is in bijection with the set $ \bigsqcup_{d\in D_N} B_d d $.
\item [(iv)]The subset $ W_{\geq N} $ is $ \prec_{ {LR}} $ closed, and $ W_N $ is a union of some two-sided cells.
\item[(v)] If $ w_1,w_ 2\in W_N$ with $ w_1\prec_R w_2 $, then $ w_1\sim_R w_2 $.
%
	\end{itemize}
\end{thm}
\begin{proof} 
	Assertion (i) follows immediately from the factorization formula (Theorem \ref{thm:dec}) and Corollary \ref{cor:right}. 
	
	Fix $ d\in D_N $, $ b\in B_d $. We claim that if  $ u\in W_{\leq N} $ and $ C_u $ appears in a product $ C_{v} h$ for some $ v\in bdU_d $ and $ h\in \mc{H} $, then $ u\in bdU_d $.
Let $ v=bdy_1 $ with $ y_1\in U_d $.
 By the factorization formula and  (i),  $ \nc_u $ appears in \[
\nc_{bdy_1} \tilde{h}=\nc_{bd}\nF_{y_1}\tilde{h}\in  \nc_{bd}\hn= \bigoplus_{y\in U_d}\mc{A}\nc_{bdy},
\]
where $ \tilde{h} $ is the image of $ h $ in $ \mc{H} _{\leq N}$.
Then by Lemma \ref{lem:char}, we have $ u\in bdU_d $.
	
Now we claim that for any $ d\in D_N $, $ b\in B_d $,
\begin{equation}\label{eq:claim2}
bdU_d\cup W_{>N}\text{ is } \prec_R \text{ closed}.
\end{equation}
Since $ W_{>N} $ is $ \prec_{LR }$ closed, it suffices to prove that  
\begin{equation}\label{eq:claim3}
\text{if }u \in W_{\leq N} \text{ and } u\prec_{R}  v \text{ with }v\in bdU_d, \text{ then }u \in bdU_d.
\end{equation}
Since $ u\prec_R v $, we can find a sequence of elements \[
u=w_1, w_{2},\cdots, w_{k-1},w_k=v \text{ in } W
\]
such that, for all $ 1\leq i\leq k-1 $, $ C_{w_{i}} $ appears in $ C_{w_{i+1}}h_i $ for some $ h_i\in\mc{H} $. Then we have 
$ w_i\in W_{\leq N} $ for all $ 1\leq i\leq k $, because $ w_{i}\prec_R w_{i+1} $, $ u=w_1\in W_{\leq N} $ and $ W_{>N} $ is $ \prec_{R} $ closed. Now inductively applying the claim in the last paragraph, we have all $ w_i $ and in particular $ u $ belong to $ bdU_d $. This proves \eqref{eq:claim3} and  \eqref{eq:claim2}.

Now we prove assertion (ii).
	Suppose $ w\sim_{ {R}} bd$. If $ w\in \wnn$, then  $ bd\in\wnn $ since $ \wnn $ is $ \prec_{ {LR}} $
closed, a contradiction with $ bd\in\Omega_{ N}\subset W_{\leq N} $. Thus $ w\in\wn $.	Then by \eqref{eq:claim3}, we have $ w\in bdU_d $. Therefore the right cell containing $ bd $ is a subset of $  bdU_d $. 

Fix an arbitrary $ y\in U_d $. By the factorization formula, $  h_{d,dy,dy}$ is  equal to $h_{d,d,d}  $, which has degree $ N $.
 Thus $ \af(dy)\geq N $. If $ \af(dy)>N $, then $ dy \in \wnn=\Omega_{>N} $, which contradicts with \eqref{eq:ud}. Hence $ \af(dy)=N $, and  $ \gamma_{d,dy,(dy)^{-1}}\neq 0 $. By Lemma \ref{lem:cyc} (iv), we have $ d\sim_R dy $. In particular, $ d\prec_R dy $. Hence
 we can find a sequence of elements 
 $ d=w_1, w_{2},\cdots, w_k=dy $
 such that, for all $ 1\leq i\leq k-1 $, $ C_{w_{i}} $ appears in $ C_{w_{i+1}}h_i $ for some $ h_i\in\mc{H} $. In the same way as the proof of \eqref{eq:claim3}, one can see that, for all $ 1\leq i\leq k $,  $ w_i =dy_i $ for some $ y_i\in U_d $, and in particular $ w_i\in W_{\leq N} $.
Hence  for all $ 1\leq i\leq k-1 $, 
\begin{equation}\label{key}
0\neq \nc_{w_{i}} \text{ appears in }\nc_{w_{i+1}}\tilde{h}_i \in \bigoplus_{y\in U_d}\mc{A}\nc_{dy},
\end{equation}
 where $ \tilde{h}_i $ is the image of $ h_i $ in $ \mc{H}_{\leq N} $. 
 By left multiplying by $ \nE_b $ and using the factorization formula, one can see that   for all $ 1\leq i\leq k-1 $, 
 \begin{equation}\label{key}
\text{ if } \nc_{bw_{i}}  \neq 0, \text{ then }  \nc_{bw_{i}}\text{ appears in }\nc_{bw_{i+1}}\tilde{h}_i .
 \end{equation}
In particular, if $ \nc_{bw_{i}}  \neq 0 $, then  $ \nc_{bw_{i+1}}\neq 0  $.
Since $ bw_1=bd\in W_{\leq N} $, we actually have $ \nc_{bw_{i}}  \neq 0 $ for all $ i $. Hence $ C_{bw_{i}} $ appears in $ C_{bw_{i+1}}h_i  $ for all $ 1\leq i\leq k-1 $.
Thus $ bd\prec_R bdy $.
 Since $ bdy \prec _R bd $ is obvious,  we can conclude that $$   bdy\sim_R bd \in W_{\leq N}=\Omega_{ \leq N} . $$
 Since $ bdy=b\bd d\bd y\in \Omega_{ \geq N} $, we actually have 
 \begin{equation}\label{eq:contain}
 bdU_d\subset \Omega_{ N}.
 \end{equation}
Now we have proved that $ bdU_d $ is contained in a right cell. 
 Combined with the last paragraph, assertion (ii) follows.


By \eqref{eq:contain} and \eqref{eq:omega}, $ \Omega_{ N}=\bigcup_{d\in D_N}B_d dU_d  $ . By the factorization formula we have $ h_{bd,dy,bdy}=h_{d,d,d}  $ with degree $ N $, which implies that $ \af(bdy)\geq N $. Thus $ \Omega_{ N}\subset W_{\geq N} $. Since $ \Omega_{ N}\subset \Omega_{ \leq N}= \wn $, we have $ \Omega_{ N}\subset W_{N}   $.

Take any  $ x\in W_N $. By Lemma \ref{lem:compute},   $ \deg \nt_x\nt_{z}=N $  for some $ z\in W_{\leq N} $.
By Proposition \ref{prop:key1},  we have $ x\in \Omega_{ N} $. Therefore, $ \Omega_{ N}\supseteq W_{N}   $. Now we have  proved that $ W_N=\Omega_{ N}  $.

For (iii),
it remains to prove that the union is disjoint. In other words, we need to prove that if $ \Phi_{b,d}\cap\Phi_{b',d'} \neq \emptyset$ for some $ d,d'\in D_N $, $ b\in B_d $ and $ b'\in B_{d'} $, we have $ d=d' $ and $ b=b' $. Since $ \Phi_{b,d} $, $ \Phi_{b',d'}  $ are right cells, we have $ \Phi_{b,d}=\Phi_{b',d'}  $. Thus $ bd=b'd'y' $ for some $ y'\in U_{d'} $. Since $ b\in B_d $,  the definition of $ B_d $ implies $ y'=e $. Thus $ bd=b'd' $.
Then $ d\sim_L bd=b'd'\sim_L d' $. By assertion (ii) for left cells, we have $ dU_d=d'U_{d'} $. Then $ d=d'\bd y_1 $ and $ d'=d\bd y_2 $ for some $ y_1\in U_{d'},y_2\in U_{d} $. This forces $ d=d' $, and then $ b=b' $. This proves (iii).

By assertion (iii) and  claim \eqref{eq:claim2}, $ W_{\geq N} $ is $ \prec_{R} $ closed. It is also $ \prec_L $ closed since $ W_{\geq N} $ is stable by taking inverse. Thus $ W_{\geq N} $ is $ \prec_{LR} $ closed. Then $ W_N $ is a union of some two-sided cells, since $ W_{>N} $ is also $ \prec_{ {LR}} $ closed. This proves (iv).

Now we prove (v). By (iii), $ w_2 \in  bdU_d$ for some $ d\in D_N $ and $ b\in B_d $. Using claim \eqref{eq:claim3}, we have $ w_1\in bdU_d $. Thus $ w_1\sim_R w_2 $ by (ii).
\end{proof}

\begin{thm}\label{thm:cor2} 
	Recall that $ \Phi_{b,d }=bdU_d $, $ \Gamma_{b,d}=\Phi_{b,d }^{-1} $ and write $ \Phi_{d}=\Phi_{e,d} $, $ \Gamma_d=\Phi_d^{-1} $.
	For  $ d,d'\in D_N $, denote by $P_{d,d'}:=  \Phi_d\cap\Gamma_{d'}$.
	
	\begin{itemize}
		\item [(i)] 
		For  $ b\in B_d $, $ b'\in B_{d'} $, we have $ \Phi_{b,d}\cap \Gamma_{b',d'}=bP_{d,d'}b'^{-1} $, and\begin{equation}\label{key}
		W_N=\bigsqcup_{\substack{d,d'\in D_N \\b\in B_d ,b'\in B_{d'} }}bP_{d,d'}b'^{-1} ,
		\end{equation}
		i.e. for any $ w\in W_N $, there is a unique  factorization  $ w=b\bd p_w\bd b'^{-1} $ such that  $ d,d'\in D_N$,  $ b\in B_d $, $ b' \in B_{d'}$, $ p_w\in P_{d,d'} $. We have   \begin{equation}\label{eq:dec2}
		\nc_w=\nE_b\nc_{p_w}\nF_{b'^{-1}}.
		\end{equation}
		
		\item [(ii)] Let $ x,y,z\in W_N $ such that \[
		x=b_1p_xb_2^{-1},\quad y=b_3p_yb_4^{-1},\quad z=b_5p_zb_6^{-1},
		\]where $ d_i\in D_N $, $ b_i\in B_{d_i} $ for  $ 1\leq i\leq 6 $ and  $ p_x,p_y, p_z $ are  given by the factorization  in (vi).  If $ \gamma_{x,y,z}\neq 0 $, then $ (b_2,d_2)=(b_3,d_3) $, $ (b_4,d_4)=(b_5,d_5)$, $ (b_6,d_6)=(b_1,d_1)$, and \[
		\gamma_{x,y,z}=\gamma_{p_x,p_y,p_z}.
		\]
	\end{itemize}
\end{thm}
\begin{proof}
Let $ w\in \Phi_{b,d}\cap\Gamma_{b',d'} $. Then $ w=bdy $ for some $ y\in U_d $. Since $ bdy\prec_{ {L}} dy$,  by Theorem \ref{thm:cor} (iii) and (v) we have $ bdy\sim_{ {L}} dy $. Thus $ dy$ belongs to the left cell  $\Gamma_{b',d'} $, i.e. $ dy=u\bd b'^{-1} $ for some  $ u\in \Gamma_{d'} $. Since $ u\in W_N $, $ dy\prec_{ {R}} u $, by Theorem \ref{thm:cor} (v) again, we have  $ u\sim_{ {R}} dy $, and hence $ u\in\Phi_d $. Take $ p_w=u $. Then we have proved $ w=b\bd p_w\bd b'^{-1} $, $ p_w\in \Phi_{d}\cap\Gamma_{d'}=P_{d,d'} $, and by the factorization formula, $ \nc_w=\nE_b\nc_{dy} =\nE_b\nc_{u}\nF_{b'^{-1}}$. Then we have $ \Phi_{b,d}\cap \Gamma_{b',d'}=bP_{d,d'}b'^{-1} $. By Theorem \ref{thm:cor} (iii), \[
W_N=\left( \bigsqcup_{\substack{d\in D_N\\b\in B_d}}\Phi_{b,d}\right) \bigcap\left( \bigsqcup_{\substack{d'\in D_N\\b'\in B_{d'}}}\Gamma_{b',d'}\right) =\bigsqcup_{\substack{d,d'\in D_N \\b\in B_d ,b'\in B_{d'} }}\Phi_{b,d}\cap\Gamma_{b',d'}.
\]Then (i) follows.

Now we prove (ii).
By Lemma \ref{lem:cyc} (iv), if $ \gamma_{x,y,z}\neq 0 $ for $ x,y,z\in W_N $, then $ x\sim_{ {L}} y^{-1}  $, $ y\sim_{ {L}} z^{-1} $, $ z\sim_{ {L}} x^{-1} $, and hence $ (b_2,d_2)=(b_3,d_3) $, $ (b_4,d_4)=(b_5,d_5)$, $ (b_6,d_6)=(b_1,d_1)$ by Theorem \ref{thm:cor} (ii) and (iii).

By Theorem \ref{thm:cor} (iv) and (v), we have \[
\nc_{db^{-1}}\nc_{bd}=\sum_{p\in P_{d,d}} h_{db^{-1}, bd, p}\nc_p.
\]
For any  $ p\in P_{d,d}$, we have $ \nc_{p^{-1}} \nc_{db^{-1}}=\eta_d \nc_{p^{-1}b^{-1}}$ by Theorem \ref{thm:dec}. This implies that if $ \gamma_{ p^{-1}, db^{-1}, bd}\neq 0  $, then $$  p=d , \gamma_{ p^{-1}, db^{-1}, bd}=\gamma_{d,d,d} .$$ 
If $ \gamma_{db^{-1}, bd, p^{-1}}\neq 0 $, then by Lemma \ref{lem:cyc} (iv) we have $ \gamma_{ p^{-1}, db^{-1}, bd}= \gamma_{db^{-1}, bd, p^{-1}}\neq 0 $. Therefore $ \gamma_{db^{-1}, bd, p^{-1}}\neq 0 $ implies that  $ p=d $ and $  \gamma_{db^{-1}, bd,p^{-1}}=\gamma_{d,d,d}$.  

Keep notations of assertion (ii) and assume that  $ \gamma_{x,y,z}\neq 0 $.
Let $ x_1=b_1p_x $, $ y_1= p_yb_4^{-1}$. By the factorization formula,
\begin{align*}
\eta_{d_2}^2 \nc_x\nc_y&= \nc_{x_1}\nc_{d_2b_2^{-1}}\nc_{b_2d_2}\nc_{y_1}\\
&=\sum_{p\in P_{d_2,d_2}}h_{d_2b_2^{-1},b_2d_2,p}\nc_{x_1}\nc_p\nc_{y_1}.
\end{align*}
Hence $ \eta_{d_2}^2 h_{x,y,z^{-1}}=\sum_{\substack{p\in P_{d_2,d_2}\\u\in \wn}} h_{d_2b_2^{-1},b_2d_2,p} h_{x_1,p,u}h_{u,y_1,z^{-1}} $. Using the fact that  $$  \deg h_{w_1,w_2,w_3}\leq N  $$ for $ w_1,w_2\in W $ and  $ w_3\in\wn $, one can see that 
\begin{align*}
\gamma_{d_2,d_2,d_2}^2\gamma_{x,y,z}
&=\sum_{\substack{p\in P_{d_2,d_2}\\u\in \wn}} \gamma_{d_2b_2^{-1},b_2d_2,p^{-1}}\gamma_{x_1,p,u^{-1}}\gamma_{u,y_1,z}\\
&=\sum_{u\in \wn}\gamma_{d_2,d_2,d_2}\gamma_{x_1,d_2,u^{-1}}\gamma_{u,y_1,z}\quad \text{(by the last paragraph)}\\
&=\gamma_{d_2,d_2,d_2}^2\gamma_{x_1,y_1,z}\quad \text{(since $ \nc_{x_1}\nc_{d_2}=\eta_{d_2} \nc_{x_1} $)}\\
&=\gamma_{d_2,d_2,d_2}^2\gamma_{p_x,p_y,p_z} \quad \left(\text{\begin{minipage}{7cm} by  the factorization formula, $$  \nc_{b_1p_x}\nc_{p_yb_4^{-1}}=\sum_{v\in P_{d_1,d_4}}h_{p_x,p_y,v}\nc_{b_1vb_4^{-1}} $$
\end{minipage}}\right).
\end{align*}
Then we have $ \gamma_{x,y,z} =\gamma_{p_x,p_y,p_z}$.
\end{proof}

\subsection{Preparation for the proof of P1-P15}\label{subsec:prepare}
Keep Assumption  \ref{ass5}. 
Recall notations $ \Delta(z) $, $ n_z $, and  $ \mc{D}_N $ from section \ref{subsec:basis}.

\begin{prop} \label{prop:p4} The following properties hold.
\begin{itemize} 
	\item [(i)] For $ z\in W_N $, we have $\af(z)\leq \Delta(z)  $, and \[
	\mc{D}_N=\{z\in W_N\mid \af(z)=\Delta(z) \}=\{bdb^{-1}\mid d\in D_N, b\in B_d \}.
	\]
	In particular, $ z^2=e $ for any $ z\in\mc{D}_{N} $, and  every left cell in $ W_N $ contains a unique element in $ \mc{D}_{N} $.
	
	\item [(ii)] Let $ x,y\in \wn $ and $ z\in \mc{D}_N $. Then $ \gamma_{x,y,z}\neq 0 $ if and only if $ x=y^{-1} $ and $ y\sim_{ {L}} z $. Moreover in this case $ \gamma_{x,y,z}=n_z =\pm1$.
\end{itemize}
\end{prop}

\begin{proof}
	Let $ z=bdy $ for some $ d\in D_N $, $ b\in B_d $, $ y\in U_d $. We have $ \nc_{bd}\nc_{dy}=\eta_d\nc_{z} $, and 
	\begin{align*}
	\eta_d p_{e,z}&=\tau(\eta_dC_{z})\\
	&\equiv \ntau(\eta_d\nc_z) \mod \mc{A}_{<0} \quad\text{ by \eqref{eq:tau} }\\
	&=\ntau (\nc_{bd}\nc_{dy})\\
	&\equiv \delta_{b,y^{-1}} \mod \mc{A}_{<0} \quad\text{ by \eqref{eq:tau1}}.
	\end{align*}
Hence $$  \Delta(z)=-\deg p_{e,z}\leq \deg \eta_d=N =\af(z) ,$$ and the equality holds if and only if $ b=y^{-1} $, i.e. $ z=bdb^{-1} $. Then (i) follows.

Let $ x,y\in \wn $, $ z\in \mc{D}_N $ with $ \gamma_{x,y,z}\neq 0 $. By Lemma \ref{lem:cyc}(iv), $$  x\sim_R z=z^{-1}\sim_L y  .$$ By Theorem \ref{thm:cor}, we have $ x,y\in W_N $ and $ \nc_x\nc_y=\sum_{w\in W_N}h_{x,y,w}\nc_w $. We have
\begin{align*}
\ntau(\nc_x\nc_y)&=\ntau(\sum_{w\in W_N}h_{x,y,w}\nc_w)\\
&\equiv \tau(\sum_{w\in W_N}h_{x,y,w}C_w) \mod{\mc{A}_{<0}}\text{ by \eqref{eq:tau}}\\
&=\sum_{w\in W_N}h_{x,y,w}p_{e,w}.
\end{align*}
Then by \eqref{eq:tau1} we have \begin{equation}\label{eq:mod}
 \sum_{w\in W_N}h_{x,y,w}p_{e,w}\equiv \delta_{x,y^{-1}} \mod \mc{A}_{<0}.
\end{equation}By (i), we have $$  \deg h_{x,y,w}p_{e,w}\leq 0  .$$ If $ \deg h_{x,y,z}p_{e,z}= 0 $, then $ \gamma_{x,y,z^{-1}}\neq 0 $ and $ z\in\mc{D}_N $. By (i), $ z $ must be the unique  element   that contained in  $ \mc{D}_N $  and   the left cell of $ y  $.
Hence for $ w\neq z $, $ \deg h_{x,y,w}p_{e,w}< 0  $.
By \eqref{eq:mod}, we have $ \gamma_{x,y,z}n_z=\delta_{x,y^{-1}} $. Thus $ \gamma_{x,y,z}\neq 0 $ implies that $ x=y^{-1} $ and  $ \gamma_{x,y,z}=n_z=\pm 1$. 

Conversely, for any $ x=y^{-1}\sim_{ {R}} z \in \mc{D}_N$,  the above arguments also show that $ \gamma_{x,y,z}n_z=1 $, and hence $  \gamma_{x,y,z}=n_z=\pm 1 $. This proves (ii).
\end{proof}

\begin{prop}\label{prop:p15}
 Let $ \mc{M} $ be an $ \mc{A}\otimes_{\mathbb{Z}}\mc{A} $ module with basis $ \{ m_w\mid w\in W_N \} $. 
	Let $ \hn $ act on $ \mc{M} $ on the left via\[
	(f\nc_{x}). (m_w)=\sum_{z\in W_N}(fh_{x,w,z}\otimes 1) m_z \text{ for } x\in \wn, f\in \mc{A}
	\]
	and
	on the right via\[
	(m_w).(f\nc_{x})=\sum_{z\in W_N}(1\otimes fh_{w,x,z}) m_z \text{ for } x\in \wn, f\in \mc{A}
	\]
	Then these two actions  commute with each other.
\end{prop}

\begin{proof}
We  abbreviate $$ h'_{x,y,z}=h_{x,y,z}\otimes1\in\mc{A}\otimes\mc{A}, \quad h''_{x,y,z}=1\otimes h_{x,y,z}\in\mc{A}\otimes\mc{A},$$
\[
\eta_d'=h'_{d,d,d}, \quad  \eta_d''=h''_{d,d,d}  .
\]
	We  first claim that   
	\begin{equation}\label{claim:comm}
	(\nc_{xd}m_{d})\nc_{dy}=\nc_{xd}(m_{d}\nc_{dy}) \text{ for }d\in D_N,  x^{-1}, y\in U_d.
	\end{equation}
	
	By the factorization formula, $ \nc_{xd}\nc_{d}= \eta_d\nc_{xd} $ and hence $\nc_{xd}m_{d}=\eta_d'm_{xd}$. Similarly,  $m_{d}\nc_{dy}=\eta_d''m_{dy}$.
	
	Let  $\nE_x\nc_{d}\nF_{y}=\sum_{z\in \wn}b_z\nc_z$ for some $ b_{z}\in\mc{A} $. Then $ b_z $ is bar invariant and $h_{xd,dy,z}=b_z\eta_d$. Since $ \deg h_{xd,dy,z} \leq N  $ and $ \gamma_{d,d,d}=n_d=\pm 1 $, we have $ b_z\in\mathbb{Z} $ and $ b_z=\gamma_{xd,dy,z^{-1}}n_d $.
	So 
	\begin{align*}
	(\nc_{xd}m_{d})\nc_{du}&= \eta_d' m_{xd}\nc_{dy}\\ 
	&=\eta_d'\eta_d'' n_d\sum_{z\in\wn }\gamma_{xd,dy,z^{-1}}m_z.
	\end{align*}
	Similar computations show that $ \nc_{xd}(m_{d}\nc_{du})=\eta_d'\eta_d'' n_d\sum_{z\in\wn }\gamma_{xd,dy,z^{-1}}m_z $. Then claim \eqref{claim:comm} follows.

	Let $u,v\in \wn$,  $w\in W_N$ with $ w=bdy $, $ d\in D_N $, $b\in B_d$, $y\in U_d$. We have
	\begin{align*}
\eta_d'\eta_d''	\left(\nc_um_{bdy}\right)\nc_v&=\eta_d''\left(\nc_u\left(\nc_{bd}m_{dy}\right)\right)\nc_v\\
	&=\eta_d''\left(\left(\nc_u\nc_{bd}\right)m_{dy}\right)\nc_v\\ 
	&=\eta_d''\left(\sum_{x'\in U_d^{-1}}h'_{u,bd,x'd}\nc_{x'd}m_{dy}\right)\nc_v \\
	&=\left(\sum_{x'\in U_d^{-1}}h'_{u,bd,x'd}\nc_{x'd}\left(m_{d}\nc_{dy}\right)\right)\nc_v\\
	&=\left(\sum_{x'\in U_d^{-1}}h'_{u,bd,x'd}\left(\nc_{x'd}m_{d}\right)\nc_{dy}\right)\nc_v \text{ by \eqref{claim:comm}}\\
	&=\left(\sum_{x'\in U_d^{-1}}h'_{u,bd,x'd}\left(\nc_{x'd}m_{d}\right)\right)\left(\nc_{dy}\nc_v\right) \\
	&=\sum_{\substack{x'\in U_d^{-1}\\y'\in U_d}} h'_{u,bd,x'd}h''_{dy,v,dy'}\left(\nc_{x'd}m_{d}\right)\nc_{dy'}.
	\end{align*}
	
	Similar computations show that \[\eta_d'\eta_d''\nc_u(m_{bdy}\nc_v)= \sum_{\substack{x'\in U_d^{-1}\\y'\in U_d}} h'_{u,bd,x'd}h''_{dy,v,dy'}\nc_{x'd}(m_{d}\nc_{dy'}).\] Using \eqref{claim:comm} again, we have $(\nc_xm_{bdy})\nc_{y}=\nc_x(m_{bdy}\nc_y)$. This completes the proof.
\end{proof}

\section{Proof of P1-P15}\label{sec:proof}

\begin{ass}
	In this section, $ (W,S) $ is a Coxeter group with complete graph with a positive weight function $ L $.
\end{ass}

\begin{prop}\label{prop:ind}
If $ (\text{P1,P4,P7,P8})_{>N} $ and  $ W_{>N}=\Omega_{>N} $  hold for $ (W,S) $,  then   $ (\text{P1-P11})_{N} $,  $ (\text{P13-P15})_{N} $  and $ W_{N}=\Omega_{N} $ hold.
\end{prop}

\begin{proof}
	First, $ W_{N}=\Omega_{N} $ follows from Theorem \ref{thm:cor} (iii).
	
	By Proposition \ref{prop:p4} (i) we have $ (\text{P1})_{N} $, $ (\text{P6})_{ N} $.  By Theorem \ref{thm:cor} (iv), if $ w\in W_N $ and $ w'\prec_{ {LR}} w $, then $ w'\in W_{\geq N} $, and hence $ \af(w')\geq N=\af(w) $. This proves $ (\text{P4})_{N}  $. By Theorem \ref{thm:cor} (v) we have $ (\text{P9,P10})_{ N} $. 
	
	If $ w' \prec_{ {LR}} w $ with $ w\in W_N $ and $ \af(w')=\af(w) $, then by definition we have a sequence  $ w'=w_1,w_2, \cdots, w_n=w$ such that $ w_i\prec_{ {L}} w_{i+1} $ or $ w_i\prec_{R} w_{i+1} $. By  $ (\text{P4})_{\geq N}  $, we have $$  N=\af(w_1)\geq \af(w_2)\geq \cdots \geq \af(w_n) =N .$$ By $  (\text{P9,P10})_{ N}  $ we have $ w'\sim_{ {LR}} w $. This proves $  (\text{P11})_{ N}  $.
	
By Theorem \ref{thm:cor} (iii), if $ w\in W_N $, then $ w=bdy $ for some $ d\in D_N $, $ b\in B_d $ and $ y\in U_d $. By Theorem \ref{thm:cor} (iii), $w\sim_{ {LR}} d$. Then $ w^{-1}\sim_{ {LR}} d^{-1}=d$, and hence $ w\sim_{ {LR}} w^{-1} $. This proves  $ (\text{P14})_{ N} $.

		Now we prove $ (\text{P8})_{N} $ and use notations in $ (\text{P8})_{\geq N} $. By $(\text{P8})_{>N}   $, we can assume that $ x,y,z\in \wn $ and at least one of them lies in $ W_N $. We claim that we always have $ z\in W_N $ no matter which one  belongs to $ W_N $. For example, if $ x\in W_N $, then $ \gamma_{x,y,z}\neq 0 $ implies that $ z^{-1}\prec_{ {LR}} x$, and by Theorem \ref{thm:cor} (iv), we have $ \af(z)\geq \af(x)=N $. But $ z\in \wn $. Thus $ z\in W_N $. Then
		using Lemma \ref{lem:cyc} (iv), we have $ x\sim_{ {L}} y^{-1} $, $ y\sim_{ {L}} z^{-1} $,  $ z\sim_{ {L}} x^{-1} $. This proves $ (\text{P8})_{N} $.
		
		Now we prove $ (\text{P7})_{N} $ and use notations in $ (\text{P7})_{\geq N} $. By $(\text{P7})_{>N}   $, we can assume that $ x,y,z\in \wn $ and one of them is in $ W_N $. If $  \gamma_{x,y,z},\gamma_{y,z,x},\gamma_{z,x,y} $ are all zero, then $\gamma_{x,y,z}=\gamma_{y,z,x}=\gamma_{z,x,y}$ is obvious.
		If one of $  \gamma_{x,y,z},\gamma_{y,z,x},\gamma_{z,x,y} $ is nonzero, say $  \gamma_{z,x,y}\neq 0 $, then by  arguments in the previous paragraph, we have $ y\in W_N $, and then applying Lemma \ref{lem:cyc} (iv) we have 
		$ \gamma_{x,y,z}=\gamma_{y,z,x}=\gamma_{z,x,y}$.   This proves $ (\text{P7})_{N} $.

	By   Proposition \ref{prop:p4} (i) and (ii), we have $ (\text{P13})_{ N}  $. 
	
	Let $ z\in \mc{D}_N $ with $ \gamma_{x,y,z}\neq 0 $. Then by $ (\text{P8,P4})_{N} $, we have $ x,y\in W_N $. By Proposition \ref{prop:p4} (ii), we have $ x=y^{-1} $. This proves $  (\text{P2})_{N}  $.
	
	Let $ y\in W_N $ and $ z\in\mc{D} $ with $ \gamma_{y^{-1},y,z} \neq 0$. Then by $ (\text{P8,P4})_{N} $, we have $ z\in \mc{D}_N $, and $ y\sim_{ {L}} z^{-1} $. By $ (\text{P6})_{N} $, $ z^{-1}=z $. Hence $ z $ is contained in the left cells that contains $ y $. By $ (\text{P13})_{N} $, $ z $ is unique. This proves $ (\text{P3})_{N} $.
	
	Let $ z\in \mc{D}_N $ with $ \gamma_{y^{-1},y,z}\neq 0 $. Then $ y\in W_N $ and by  Proposition \ref{prop:p4} (ii), $  \gamma_{y^{-1},y,z}=n_z=\pm1 $. This proves $ (\text{P5})_{ N}  $. 
	
Use notations in $ (\text{P15})_{N} $ and Proposition \ref{prop:p15}. We have  $ w,w' \in W$ and $ x,y\in W_N $.
If $ w $ or $ w' $  are in $ \wnn $, then both sides of the equation in $ (\text{P15})_{N} $ are  $  0  $ by $ (\text{P4})_{> N}  $. Assume now that $ w , w'\in \wn $.  If $ h_{w,x,z}\otimes h_{z,w',y}\neq 0 $, then $ z\prec_{ {LR}} x $ and $ y\prec_{ {LR}} z $. By  $ (\text{P4})_{\geq N}  $, $$  \af(z)\geq \af(x)=N=\af(y)\geq \af(z) ,$$ and hence $ z\in W_N $.
Thus the left side of the equation in $ (\text{P15})_{N} $ is the coefficient of $ m_y $ in $ (\nc_wm_x)\nc_{w'} $, and similarly the right side is that of $ \nc_w(m_x\nc_{w'}) $. They are equal by Proposition \ref{prop:p15}.	This proves  $ (\text{P15})_{N} $.
\end{proof}

\begin{thm}\label{thm:last}
		 For any Coxeter group with complete graph $ (W,S) $, conjectures P1-P15 hold, and $ W_{ N}= \Omega_{ N}$ for any $ N $.
\end{thm}
\begin{proof}
	
	Since the $ \af $-function is bounded for Coxeter groups with complete graph,   we have $ W_{>N_0}=\Omega_{>N_0}=\emptyset $ for large enough $ N _0$, and hence $ (\text{P1-P11})_{>N_0} $ and $ (\text{P13-P15})_{>N_0} $ hold trivially. 
	Then by Proposition  \ref{prop:ind} and  decreasing induction on $ N $, one can  prove $ (\text{P1-P11})_{\geq  N} $, $ (\text{P13-P15})_{\geq N} $ and $ W_{ N}= \Omega_{ N}$ for all $ N $. Thus P1-P11, P13-P15 and $ W_{N}=\Omega_{N} $ hold for  any $ N $ and any Coxeter group with complete graph.
	
	At last we prove P12. Let $ w\in W_I$ with $ I\subset S $. Assume  $ \af(w)=N $.
	By the last paragraph, $ W_N=\Omega_{ N} $, and $ (W_I)_N=(\Omega_I)_{N} $ (note that $ W_I $ is also a Coxeter group with complete graph). Hence $$  w\in W_N\cap W_I=\Omega_{ N} \cap W_I= (\Omega_I)_{ N}=(W_I)_N $$ (see \eqref{eq:omegaj}).
	This proves P12.
%
%
\end{proof}

\begin{cor}
 Assumptions \ref{ass4} and \ref{ass5}  and all the results in section \ref{sec:main}  hold for all  $ N\in\mathbb{N} $.
\end{cor}

\section{Cells}\label{sec:cells}

\begin{ass}
In this section, $ (W,S,L) $  is  a positively weighted Coxeter group with complete graph.
\end{ass}

Define $ \mathbb{A}=\{N\in\mathbb{N}\mid W_N\neq\emptyset \} $.

\begin{thm} \label{th:cell}
We have 
\[ W=\bigsqcup_{\substack{N\in\mathbb{A},d\in D_N\\b\in B_d,y\in U_d}} bdU_d \]
is the  partition into right cells and 
 \[W=\bigsqcup_{N\in \mathbb{A}}\Omega_{ N} \] is the partition into  two-sided cells.
\end{thm}

\begin{proof}
The right cell partition has been given in  Theorem \ref{thm:cor} (ii) and (iii).

By Theorem \ref{thm:cor}, the  subset $ W_N=\Omega_{ N} $ is a union of   two-sided cells with $ \af $-values $ N $.
For $ w\in \Omega_{ N} $, we have $ w=bdy $ for some $ d\in D_N $, $ b\in B_d $, $ y\in U_d $, and hence $ w\sim_{ {LR}} d $. Thus any element $ w \in W_N$ is $ \sim_{ {LR}} $ equivalent to some $ d\in D_N $. To determine the two-sided cells in $ W_N $, it is enough to determine the restriction of $ \sim_{ {LR}} $ on $ D_N $.

Let $ d_1, d_2\in D_N $. We find an element $ x\in \Omega_{ N} $ as follows.

 If one of $ d_1,d _2  $,  say $ d_2 $, is the longest element $ w_J $ of  some finite parabolic subgroup $ W_J $, then we take $ x $ to be the longest element of the coset $ d_1W_J $.  

Assume that  $ d_1=d_{J_1} $, $ d_2=d_{J_2} $ for some $ J_1, J_2 \subset S$ (see section \ref{sec:rank2} for  definition of $ d_J $). There are three  cases:
\begin{itemize}
\item If  $ J_1\cap J_2=\emptyset $, then we take  $ x=d_1d_2 $.
		
\item If $ J_1\subset J_2 $ or $ J_2\subset J_1 $, then we must have $ J_1=J_2 $ since $ |J_1|=|J_2|=2 $. In this case,  we take $ x=d_1=d_2$.
			
\item If $ J_1=\{s,r\}$ or $ J_2=\{t,r\}$,  then we take $ x=d_{J_1}trsd_{J_2} $. 
\end{itemize}

Using the the fact that $ m_{t_1t_2}\geq 3 $ for any $ t_1,t_2\in S $, one can check directly that  the element $ x $ constructed as above belongs to $ \Omega_{ N} $, and hence belongs to $ d_1U_{d_1}\cap U_{d_2}^{-1}d_2 $. Therefore,\[
d_1\sim_{ {R}} x\sim_{ {L}} d_2.
\]
Hence $ D_N $ is $ \sim_{ {LR}} $ connected. This completes the proof.
\end{proof}

%

\begin{cor}~
\begin{itemize}
\item [(i)]Each two-sided cell of $ W $ has a nonempty intersection with a finite parabolic subgroup $ W_I $  of $ W $, and hence there are only finitely many two-sided cells in $ W $. 
\item [(ii)]Assume $ J\subset S $. If $ W_ N\cap W_J\neq \emptyset$, then $ W_ N\cap W_J$ is a two-sided cell of $ W_J $. Similar, if $ \Phi $ is a right cell of $ W $ and $ \Phi\cap W_J\neq \emptyset$, then $ \Phi\cap W_J $ is a right cell of $ W_J $.
\item [(iii)] If $ d\in D_N $, $ y\in U_d  $  and $ y=y_1\bd y_2 $ for some $ y_1,y_2\in W $, then $ y_1\in U_d $. This implies that  the right cell $ bdU_d $ is right-connected in the sense that for any $ u,v\in bdU_d $ we have a sequence $u= w_1, w_2,\cdots, w_k =v$ such that $ w_i^{-1}w_{i+1}\in S $ for all $ 1\leq i<k $.
\end{itemize}
\end{cor}
\begin{proof}
It follows immediately from Theorem \ref{th:cell}.
\end{proof}

Note that it is possible that there are infinitely many right cells in $ W $, i.e. $ B_d $ is an infinite set for some $ d\in D $, (see \cite{bedard1986infinity-leftcell,xi2012afunction}).
%
%
%
%
%
%
%
%

\section{Possible generalizations}\label{sec:generalized}

In this section, we study Coxeter groups such that \begin{itemize}
	\item [$ (*) $]all irreducible components of finite parabolic subgroups of $ W $ have rank at most 2.
\end{itemize} For example,  the following Coxeter groups satisfy the above condition:
\begin{itemize}
	\item Coxeter groups with complete graph,
\item 	right-angled Coxeter groups,
\item infinite Coxeter groups of rank 3,
\item Coxeter groups satisfying $ m_{st}\neq 3 $ for all $ s\neq t\in S $.
\end{itemize}
Note that the boundedness conjecture for the above examples has been proved, see \cite{xi2012afunction,Bel2004,zhou,gao2016rank3,li-Shi-non3edge}.

We expect that methods of this paper for the proof of P1-P15 for Coxeter groups with complete graph  also works for the Coxeter groups satisfying $ (*) $. 
This is based on the following  observations.
The results in sections \ref{sec:rank2} and \ref{sec:complete} are used only in subsection \ref{subsec:degest}.
In the arguments of subsections \ref{subsec:onesided}-\ref{subsec:prepare} and section \ref{sec:proof}, we actually does not use any particular properties of Coxeter groups with complete graph except the one that their finite dihedral groups have rank 1 or 2. Let us give more details in the following.

Fix a Coxeter group $ (W,S) $ satisfying $ (*) $.
Define the set $ D $ to be the set of elements $ d $ of the form\[
d=d_1d_2\cdots d_k \text{ for some } k\in \mathbb{N},
\]
where $ d_j $, for each $ 1\leq j\leq k $, is 
\begin{itemize}
	\item [(i)] either the longest element of an irreducible finite parabolic subgroup $ W_{I_j} $,
	\item [(ii)] or the element $ d_{I_j} $ when $ I_j=\{s,t\}  $ for some $ s,t\in S $ satisfying $ 4\leq m_{st}<\infty $, $ L(s)\neq L(t) $,
\end{itemize}
and we require that $ I_1, I_2, \cdots, I_k $ have no intersection with each other.
Note that $ d_1,d_2,\cdots, d_k $ commute with each other, and  when $ k=0 $ we have $ d=e $. Define a function $ \af': D\to \mathbb{N} $ by\[
\af'(d)=\sum _{1\leq j\leq k} \af'(d_j),
\]
where $  \af'(d_j)=L(d_{I_j}) $ in case (i) and $  \af'(d_j)=L'(d_{I_j}) $ in case (ii). Define \[
D_{\geq N}=\{d\in D\mid \af'(d)\geq N \},
\]
and similarly $ D_{> N} $, $ D_{N} $. Obviously, these generalize the corresponding notations in section \ref{sec:main} for Coxeter groups with complete graph.
Using $ D $ as above, one can naturally generalize the notations
 $ \Omega_{\geq N} $, $ \Omega_{>N} $,  $ \Omega_{\leq N} $, $ \Omega_{N}$, $ U_d $ and $ B_d $   ($ d\in D_N $)  in  section \ref{sec:main} to the present case for Coxeter groups satisfying $ (*) $, see \eqref{eq:omegaN}, \eqref{eq:ud} and \eqref{eq:bd}.
 
 \begin{lem}
 	Let $ (W,S) $ be a Coxeter group satisfying $ (*) $.\begin{itemize}
 		\item [(i)] Let $ d\in D_N $. For any $ y\in U_d $, $ w\leq d $, we have $ wy=w\bd y $. 
 		\item [(ii)] For any  $ w\in \Omega_{ N} $,  we have $ w=b\bd d\bd y $ for some $ d\in D_N $, $ b\in B_d $ and $ y\in U_d $.
 	\end{itemize}
 \end{lem}

This is a  generalization of Lemma \ref{lem:bdy-basic}. The proof is almost the same.

The following proposition generalizes
the results in subsection \ref{subsec:onesided}.

\begin{prop}
	Let $ (W,S) $ be a Coxeter group satisfying $ (*) $. Let $ N $ be an integer such that $ W_{>N} =\Omega_{> N}$ and $ W_{>N} $ is $ \prec_{LR} $ closed.
	
	For $ d\in D_N $, $ y\in U_d $, we have\[
	\nc_{dy}=\nc_d\nF_y,
	\]
	where 
	$ \nF_y=\sum_{y'\in U_d} g_{y',y} \nt_{y'} $ is an element of $\hn $ such that\begin{itemize}
		\item $ g_{y',y}\neq 0 $ only if $ y'\leq y $, 
		\item $ g_{y,y}=1 $, 
		\item and $ \deg g_{y',y}<0 $ for $ y'<y $.
	\end{itemize}
Similarly, for $ x\in U_d^{-1} $, we have $ \nc_{xd} =\nE_x \nc_d$ with $ \nE_x=(\nF_{x^{-1}})^{\flat} $.

Moreover,   we have
\[
\nc_d\hn=\bigoplus_{y\in U_d}\mc{A}\nc_{d}\nt_y= \bigoplus_{y\in U_d}\mc{A}\nc_{dy}.
\]
\end{prop}

For the proof, see Proposition \ref{prop:predec} and Corollary \ref{cor:right}.

\begin{conj}\label{conj1}
Let $ (W,S) $ be a Coxeter group satisfying $ (*) $.  Let $ N $ be an integer such that $ W_{>N} =\Omega_{> N}$ and $ W_{>N} $ is $ \prec_{LR} $ closed.
\begin{enumerate}
	\item For all $ x,y\in \wn $, the equality $\deg \nt_x\nt _y =N $ holds  only if $ x,y\in \Omega_{ \geq N} $.
	\item For $ d\in D_N $, $ b\in B_d $, $ y\in U_d $, we have $ bdy=b\bd d \bd y $.
	\item 
		If  $ d\in D_N $, $ x\in U_d^{-1} $, $ y\in U_d $, $ w\leq d $, then
		\begin{equation}
		\deg(\nt_{xw}\nt_y) \leq-\deg p_{w,d}.
		\end{equation}
		If moreover $ b\in B_d $, $ w<d $, then we have a strict inequality
		\begin{equation}
		\deg(\nt_{bw}\nt_y) <-\deg p_{w,d}.
		\end{equation}
\end{enumerate}
\end{conj}
These three statements of this conjecture correspond to Proposition \ref{prop:key1}, Lemma \ref{lem:length} and Proposition \ref{prop:leq}, respectively. Thus this conjecture holds for Coxeter groups with complete graph.
   This conjecture is exactly what we need in the proof of the factorization formula.

\begin{thm}\label{thm:gen1}
Let $ (W,S) $ be a Coxeter group satisfying $ (*) $.  Let $ N $ be an integer such that $ W_{>N} =\Omega_{> N}$, $ W_{>N} $ is $ \prec_{LR} $ closed, and  Conjecture \ref{conj1} holds.

For  $ d\in D_N $, $ b\in B_d $, $ y\in U_d $, we have
\begin{equation}
\nc_{bdy}=\nE_b\nc_d\nF_y \text{ in }\hn.
\end{equation}
\end{thm}

The proof is  repeating that of Theorem \ref{thm:dec}.

\begin{thm}\label{thm:gen2}
Let $ (W,S) $ be a Coxeter group satisfying $ (*) $.  Let $ N $ be an integer such that $ W_{>N} =\Omega_{> N}$, $ (\text{P1},\text{P4}, \text{P8})_{>N} $  and  Conjecture \ref{conj1} hold. Then we have the following properties.
\begin{enumerate}
	\item We have $$  W_{N}=\Omega_{ N}=\bigsqcup_{d\in D_N, b\in B_d} bdU_d ,$$ which is a decomposition into right cells.
	\item The set $ W_{\geq N} $ is $ \prec_{LR} $ closed.
	\item  If $ w_1,w_2\in W_N $ with $ w_1\prec_{R} w_2 $, then $ w_1\sim_{ {R}} w_2 $.
	
	\item The statements in Propositions \ref{prop:p4} and \ref{prop:p15} hold in the present case.
\end{enumerate}
\end{thm}

The proof of this theorem is completely similar to those of Theorem \ref{thm:cor},  Propositions \ref{prop:p4} and \ref{prop:p15} .

\begin{thm}\label{thm:generalized}
Let $ (W,S) $ be a Coxeter group satisfying $ (*) $. Assume that  the boundedness conjecture and Conjecture \ref{conj1} hold for all parabolic subgroups of $ (W,S) $. Then conjectures P1-P15 hold.
\end{thm}

\begin{proof}[Sketch of  proof]
We need first prove that
$ W_{\geq N}=\Omega_{ \geq N} $,
and  $ (\text{P1-P11})_{\geq N} $ and  $  (\text{P13-P15})_{\geq N}  $ hold, by applying decreasing induction on $ N $. When $ N $ is large enough, they hold trivially by the boundedness conjecture. The details of induction is completely similar to the proof of Proposition \ref{prop:ind}, and we need Theorems \ref{thm:gen1} and  \ref{thm:gen2}.

By the above arguments, we obtain that $ W_{ N}=\Omega_{N} $ for all $ N $, and P1-P11 and P13-P15 hold. These also hold for all parabolic subgroups of $ W $ by our assumptions. Then one can prove P12 using the method in the proof of Theorem \ref{thm:last}
\end{proof}

\appendix
\section{Conjectures P1-P15 for finite dihedral groups}\label{ap:dihedral}

The aim of this section is to give a new  proof of  P1-P15 for  (irreducible) finite dihedral groups, based on some straightforward  computations. Note that the original proof is due to \cite{geck2011rank2}.

Let us   explain our strategy.

In the proof of P1-P15 for Coxeter groups with complete graph, we used P1-P15 and  cell partitions for finite dihedral groups, see Lemma \ref{lem:ppd}.   They are only used in the proofs of Lemma \ref{lem:degI}, Proposition \ref{prop:key} and equation \eqref{eq:woi}. Both Lemma \ref{lem:degI} and Proposition \ref{prop:key} are only used in the proof of Proposition \ref{prop:key1}. The equation \eqref{eq:woi} is only used to show that $ \nf_{u,v,w} $ that computed in the finite dihedral groups coincides with that computed in the whole group, see the paragraph before subsection \ref{subsec:degest}. This claim is also only used in the proof of Proposition \ref{prop:key1}. In other words, we didn't use Lemma \ref{lem:ppd} after Proposition \ref{prop:key1}. Hence if we can prove Proposition \ref{prop:key1} for finite dihedral groups, then the arguments after Proposition \ref{prop:key1} also work since finite dihedral groups are special Coxeter groups with complete graph.

By the last paragraph, to prove P1-P15 for finite dihedral groups, it is enough to prove  Proposition \ref{prop:key1} for for finite dihedral groups. This follows from the following stronger proposition.

In the following,  $ W_I $ is a dihedral group with $ I=\{s,t\} $, $ 3\leq m_{st}<\infty $.

\begin{prop}\label{prop:rank2}
	Let $ N\in\mathbb{N} $. For $ x,y\in \wn $, the equality $ \deg\nt_x\nt_y =N$ holds only if $ y\in (\Omega_I)_{ N} $. Moreover, $ (W_I)_N= (\Omega_I)_{ N} $ and $ (W_I)_{\geq N} $ is $ \prec_{LR} $ closed.
\end{prop}

We prove this proposition by decreasing induction on $ N $.


The following lemma is easy.
\begin{lem}\label{lem:low}
	For any $ x,y\in W_I $, we have\[
	\deg T_xT_y\leq L(w_I).
	\]
	The equality holds if and only if $ x=y=w_I $.
\end{lem}


By  Lemmas \ref{lem:low} and \ref{lem:compute}, we have $ (W_I)_N=\emptyset =(\Omega_I)_{ N} $ for $ N>L(w_I) $, and $ (W_I)_N=\{w_I\} =(\Omega_I)_{ N} $ for $ N=L(w_I) $. Using $$  T_sC_{w_I} =q^a C_{w_I}, \quad T_tC_{w_I} =q^b C_{w_I}  ,$$ one can see that $ \{ w_I\} $ is $ \prec_{ {LR}} $ closed. For $ N>L(w_I) $, the equality $ \deg\nt_x\nt_y =N$ cannot hold.  For $ N=L(w_I) $, the equality $ \deg\nt_x\nt_y =N$  hold only if $ y=w_0 $.  
Therefore Proposition \ref{prop:rank2} holds for $ N\geq L(w_I) $.

\begin{lem}\label{lem:like}
	Assume that $ a\leq b $,  $ N=L(w_I)-1 $. We have $(W_I)_{\leq N} =W_I\setminus\{w_I\}$.
	For any $ u,v\in  (W_I)_{\leq N}  $, we have	\begin{equation}\label{eq:deguv}
	\deg  \nt_u\nt _v \leq b+\left\lfloor \frac{l(u)-1}{2}\right\rfloor(b-a).
	\end{equation}
	
	If $ s\in \mc{R}(u) $, then we have \begin{equation}\label{eq:deguv2}
	\deg  \nt_u\nt _v \leq b+\left\lfloor \frac{l(u)-2}{2}\right\rfloor(b-a).
	\end{equation}
	If  $ \mc{R}(u)\cap \mc{L}(v)= \emptyset $, then we have  
	\begin{equation}\label{eq:deguvba1}
	\deg  \nt_u\nt _v \leq \left\lfloor \frac{l(u)}{2}\right\rfloor(b-a).
	\end{equation}
\end{lem}

	Note that   \begin{equation}\label{eq:lu}
\left\lfloor \frac{l(u)}{2}\right\rfloor(b-a)< b+\left\lfloor \frac{l(u)-2}{2}\right\rfloor(b-a)\leq  b+\left\lfloor \frac{l(u)-1}{2}\right\rfloor(b-a).
\end{equation}
\begin{proof}
	Due to Lemma \ref{lem:low} and $ N=L(w_I)-1 $, we have  $(W_I)_{\leq N} =(W_I)_{<L(w_I)} =W_I\setminus\{w_I\}$.
	
	We prove \eqref{eq:deguv}, \eqref{eq:deguv2} and \eqref{eq:deguvba1} by induction on $ l(u) $.
	If $ l(u)=0 $, then they clearly hold. Assume now that $ l(u)>0 $ and that they hold for elements of  length less than $ l(u)$. 
	
	Assume that $ \mc{R}(u)\cap \mc{L}(v)= \emptyset $.  If $uv=u\bd v $, then the result is obvious. If $ uv<u\bd v$, then there exist $ u',u'' $ such that $ u=u'\bd u'' $, $ u''\neq e $ and $ w_I=u''\bd v $. Since $ \nt_{w_I}=-\sum _{y\neq w_I} p_{y,w_I}\nt_y $, we have \begin{align*}
	\nt_u\nt_v&=\nt_{u' }\nt_{w_I} \\
	&=-\sum_{y\neq w_I} p_{y,w_I}\nt_{u'}\nt_y.
	\end{align*}
	Since $ l(u')<l(u) $,	by the induction hypothesis, we have $$  \deg 	\nt_{u'}\nt_y\leq b+\left\lfloor \frac{l(u')-1}{2}\right\rfloor(b-a)\leq b+ \left\lfloor \frac{l(u)-2}{2}\right\rfloor(b-a).  $$ 
	Since $ \deg p_{y,w_I}\leq -a $ for $ y\neq w_I $, we have
	\[
	\deg \nt_u\nt_v \leq -a+b+\left\lfloor \frac{l(u)-2}{2}\right\rfloor(b-a)=\left\lfloor \frac{l(u)}{2}\right\rfloor(b-a). 
	\] 
	This proves \eqref{eq:deguvba1}. 
	
	
	%
	
	Assume that $ r\in \mc{R}(u)\cap \mc{L}(v)\neq \emptyset $. Then\[
	\nt_{u}\nt_{v}=\nt_{ur}\nt_{rv}+\xi_r \nt_{ur}\nt_{v}.
	\]
	Note that $ \mc{R}(ur)\cap \mc{L}(v)= \emptyset $. Applying the induction hypothesis,
	we have \[
	\deg \nt_{ur}\nt_{rv} \leq b+\left\lfloor \frac{l(ur)-1}{2}\right\rfloor(b-a)=a+\left\lfloor \frac{l(u)}{2}\right\rfloor(b-a),
	\]
	\[
	\deg	\xi_r \nt_{ur}\nt_{v}\leq L(r)+\deg	\nt_{ur}\nt_{v}\leq L(r)+\left\lfloor \frac{l(ur)}{2}\right\rfloor(b-a)=L(r)+\left\lfloor \frac{l(u)-1}{2}\right\rfloor(b-a).
	\]
	Since $ L(r)\leq b $ and $ a+\left\lfloor \frac{l(u)}{2}\right\rfloor(b-a)\leq b+\left\lfloor \frac{l(u)-1}{2}\right\rfloor(b-a) $, we have 
	 \[ \deg	\nt_{u}\nt_{v}\leq b+\left\lfloor \frac{l(u)-1}{2}\right\rfloor(b-a).\]
	If $ s\in \mc{R}(u) $, then $ r=s $, $ L(r)=a $, and hence 
	\[\deg	\nt_{u}\nt_{v}\leq a+\left\lfloor \frac{l(u)}{2}\right\rfloor(b-a)=b+\left\lfloor \frac{l(u)-2}{2}\right\rfloor(b-a)\]
	using $  a+\left\lfloor \frac{l(u)-1}{2}\right\rfloor(b-a)\leq a+\left\lfloor \frac{l(u)}{2}\right\rfloor(b-a) $.
	This completes the proof.
\end{proof}

	\begin{cor}\label{cor:equal}
		Assume that $ a= b $,  $ N=L(w_I)-1 $.
For any $  u,v\in (W_I)_{\leq N}  =W_I\setminus\{ w_I\}$, we have  \begin{equation}\label{eq:uvba}
\deg \nt_u\nt _v\leq a.
\end{equation}
The equality holds  only if $ u,v\neq e$.
	\end{cor}
This follows immediately from Lemma \ref{lem:like}.

Assume $ a=b $. By the above corollary and Lemma \ref{lem:compute}, for $a< N<L(w_I) $, $ (W_I)_N=\emptyset $, and  $ (W_I)_a\subset (\Omega_I)_a=W_I\setminus \{ e,w_I \} $. Conversely, for $ v\in  (\Omega_I)_a$ and $ r\in \mc{L}(v) $, we have $ \deg\nt_r\nt_v=a $ (with $ N=a $), which implies that $ v\in (W_I)_a $ by Lemma \ref{lem:compute} again. Thus $ (W_I)_a=(\Omega_I)_a $. It is easy to see that $ (W_I)_{\geq N} $ is $ \prec_{LR} $ closed for $ a\leq N<L(w_I) $. Therefore, Proposition \ref{prop:rank2} holds for $ a\leq N<L(w_I) $ in the case of $ a=b $.

\begin{cor}\label{cor:unequal}
	Assume that $ a<b $,  $ N=L(w_I)-1 $.
	For any $  u,v\in (W_I)_{\leq N}  $, we have  \begin{equation}
	\deg \nt_u\nt _v\leq L'(d_I)= mb-(m-1)a.
	\end{equation}
	The equality holds  if and only if $ u=v=d_I$.
%
\end{cor}

\begin{proof}
	Since $ l(u)\leq 2m-1 $, by \eqref{eq:deguv}, we have $  \deg \nt_u\nt _v\leq mb-(m-1)a $. 
	Assume that $  \deg \nt_u\nt_v= mb-(m-1)a $ holds. Then by \eqref{eq:deguv}, $ l(u)=2m-1 $,  and by  \eqref{eq:deguv2}, we have $ ut<u $. These  imply that $ u=d_I $, and similarly $ v=d_I $. 
	
We claim that for  $ u_k=w(t,2k-1)$ with $ 1\leq  k\leq m $ we have 
	\begin{equation}\label{eq:udi}
	\deg \nt_{u_k}\nt_{d_I}= L'(u_k)=kb-(k-1)a,
	\end{equation}
	which implies that $  \deg \nt_{d_I}\nt_{d_I}= mb-(m-1)a  $.

	We prove \eqref{eq:udi} by induction on $ k $. If $ k=1 $, it is obvious. Assume that $ k\geq 2 $ and  that we have proved it for elements $ u_{k'} $ with $ k'<k $.
	
	We have
	\begin{align*}
	\nt_{u_k}\nt_{d_I}
	&=\nt_{u_{k-1}st}\nt_{d_I}\\
	&=\nt_{u_{k-1}s}\nt_{td_I}+\xi_t  \nt_{u_{k-1}s}\nt_{d_I}\\
	&=\nt_{u_{k-1}}\nt_{std_I}+\xi_s \nt_{u_{k-1}}\nt_{td_I}+ \xi_t  \nt_{u_{k-1}}\nt_{w_I}\\
	&=\nt_{u_{k-1}}\nt_{std_I}+\xi_s \nt_{u_{k-1}}\nt_{td_I}\\
	&\phantom{=}-\sum_{y\neq w_I, d_I}\xi_t p_{y,w_I} \nt_{u_{k-1}}\nt_{y}-\xi_t p_{d_I,w_I}\nt_{u_{k-1}}\nt_{d_I}.
	\end{align*}
	By \eqref{eq:deguv},  we have $$  \deg (\nt_{u_{k-1}}\nt_{std_I})\leq b+(k-2)(b-a) <b+(k-1)(b-a)  .$$
	By \eqref{eq:deguvba1}, we have $$ \deg (\xi_s \nt_{u_{k-1}}\nt_{td_I}) \leq a+(k-2)(b-a)<b+(k-1)(b-a) . $$
	By \eqref{eq:deguv} and the fact that  $ \deg  \xi_t p_{y,w_I} \leq 0$ for $ y\neq w_I,d_I $, we have $$ \deg  (\xi_t p_{y,w_I} \nt_{u_{k-1}}\nt_{y} )<b+(k-1)(b-a)   .$$ By the induction hypothesis, we have $$  \deg \xi_t p_{d_I,w_I}\nt_{u_{k-1}}\nt_{d_I}= b+(k-1)(b-a) .$$
	Hence $$   \deg \nt_{u_k}\nt_{d_I}= b+(k-1)(b-a)  .$$ This proves claim \eqref{eq:udi}.
\end{proof}

Assume $ a<b $. As above, by  Corollary \ref{cor:unequal} and Lemma \ref{lem:compute}, for $ L'(d_I)<N<L(w_I) $, $ (W_I)_{N}=\emptyset=(\Omega_I)_N $, and for $ N=L'(d_I) $, $ (W_I)_N=\{d_I\}=(\Omega_I)_N $. By \eqref{eq:sdi}, the set $ \{d_I,w_I\} $ is $ \prec_{LR} $ closed. Therefore, Proposition \ref{prop:rank2} holds for $L'(d_I) \leq N<L(d_I) $ in the case of $ a<b $.

\begin{lem}\label{lem:unequal}
	 Assume that $ a<b $, and $ N=L'(d_I)-1 $. Then $ (W_I)_{\leq N}=W_I\setminus\{d_I,w_I\} $. For any $ u,v\in (W_I)_{\leq N} $, we have\[
	 \deg\nt_u\nt_v\leq b.
	 \]
	 If moreover $  \mc{R}(u)\cap\mc{L}(v)=\emptyset $,
	 \[
	 \deg\nt_u\nt_v\leq 0.
	 \]
	 The equality $  \deg\nt_u\nt_v=b $ holds  only if  $ t $ appears in reduced expressions of $ u $ and $ v $.
\end{lem}
\begin{proof}
	We prove it by induction on $ l(u) $. If $ l(u)=0 $, it is obvious. Assume that $ l(u)>0 $ and that it is known for elements with smaller length.

	Assume that $ \mc{R}(u)\cap \mc{L}(v)= \emptyset $.  If $uv=u\bd v$, then the result is obvious. If $ uv<u\bd v $, then there exist $ u',u'' $ such that $ u=u'\bd u'' $, $ u''\neq e $ and $ w_I=u''\bd v  $. Since 
	\begin{align*}
	  \nt_{w_I}&=-\sum _{y\neq d_I, w_I} p_{y,w_I}\nt_y  -p_{d_I,w_I}\nt_{d_I}\\
	  &=-\sum _{y\neq d_I, w_I} p_{y,w_I}\nt_y  +q^{-a}\sum _{z<d_I}p_{z,d_I}\nt_z,
	\end{align*}
we have \begin{align*}
	\nt_u\nt_v&=\nt_{u' }\nt_{w_I} \\
&=-\sum _{y\neq d_I, w_I} p_{y,w_I}\nt_{u'}\nt_y  +q^{-a}\sum _{z<d_I}p_{z,d_I}\nt_{u'}\nt_z.
	\end{align*}
	Note that $ \deg   p_{y,w_I}\leq -b $ for $ y\neq d_I, w_I $, $\deg p_{z,d_I}\leq-(b-a) $ for $ z<d_I $ (see Lemma \ref{lem:degd}). Applying induction hypothesis, $\deg \nt_{u'}\nt_y \leq b $, $  \deg \nt_{u'}\nt_z \leq b$, and hence $ \deg \nt_u\nt_v\leq 0 $.
	
	If $r\in \mc{R}(u)\cap\mc{L}(v)  $, then
\begin{equation}\label{eq:urv}
	\nt_{u}\nt_{v}=\nt_{ur}\nt_{rv}+\xi _r\nt_{ur}\nt_{v}.
\end{equation}

	Note that $ \mc{R}(ur)\cap \mc{L}(v)= \emptyset $. Applying induction hypothesis, we have $\nt_{ur}\nt_{rv}\leq b$ and $ \deg \nt_{ur}\nt_{v} \leq 0$, and hence $ \deg \nt_{u}\nt_{v}\leq b $.
	
Now we prove the last sentence of the lemma. If $ t $ does not appear in the (unique) reduced expression of $ u $, then $ u=s $, and obviously $ \deg\nt_u\nt_v=a <b$. Hence $  \deg\nt_u\nt_v=b $ implies that $ t $ appears the reduced expression of u, and similar of $ v $.
\end{proof}

Assume $ a<b $. As above, by  Lemmas \ref{lem:unequal} and \ref{lem:compute}, for $ b<N<L'(d_I) $, $ (W_I)_N=\emptyset=(\Omega_I)_N $, and $ (W_I)_b\subset W_I\setminus\{ e,s, d_I,w_I\} =(\Omega_I)_b$.  If $ v\in  (\Omega_I)_b$ and $ t\in\mc{L}(v) $, then $ \deg \nt_t\nt_{v}=b $; if $ v\in  (\Omega_I)_b$ and $ s\in \mc{L}(v) $, then  $ v=st\bd y$ for some $ y $ and
\[
\nt_{ts}\nt_{v}=\xi_a\nt_{tv}+\xi_b\nt_{ty}+\nt_y
\]
with $ N=b $, which implies that $ \deg \nt_{ts}\nt_{v}=b $. Thus by Lemma \ref{lem:compute}, $ (\Omega_I)_b \subset (W_I)_b$. Then we have proved $ (\Omega_I)_b = (W_I)_b $.

For $ u\in (W_I)_b $, we have $ u=x\bd d\bd y $ with $ d=t$, $ x\in B_d= \{e,s\} $,  $ y\in U_d $. Then we have a factorization formula 
$ \nc_u=\nE_x\nc_d\nF_{y} $
with $ N=b $ as in Theorem \ref{thm:dec}. Then using  arguments in Theorem \ref{thm:cor}, one can prove that $ (W_I)_{\geq b}  $ is $ \prec_{LR} $ closed.

Then   Proposition \ref{prop:rank2} follows for $ b\leq N<L'(d_I) $ in the case of $ a<b $.

%
%

The following lemma is easy.
\begin{lem}
 Assume that $ a<b $, and $ N=b-1 $.  We have $ (W_I)_{\leq N} =\{ e, s\}$. For $ u,v\in  (W_I)_{\leq N}   $, we have\[
\deg\nt_u\nt_v\leq a.
\]
The equality  holds if and  only if   $ u = v=s $.
\end{lem}

As above, one can see that Proposition \ref{prop:rank2} holds for $ a\leq N<b $ in the case of $ a<b $.

Using $ (W_I)_{<a}=\{e\} $, one can easily see that Proposition \ref{prop:rank2}  for $ 0\leq  N<a $. This completes the proof of Proposition \ref{prop:rank2}.

\section{Right-angled Coxeter groups}\label{ap:ra}

	Throughout this section,  $ (W,S) $ is a right-angled Coxeter group, i.e. $ m_{st}=2 $ or $ \infty $ for any $ s\neq t\in S $. Let $ L $ be a fixed positive weight function on $ W $.

The goal of this section is to prove P1-P15 for $ (W,S,L) $.


Let $ D $ be the set of elements $ w $ such that $w= s_1 s_2\cdots s_p $ for some $ s_i\in S $ with $ m_{s_is_j} =2$ for any $ i\neq j $. Define \[
D_N=\{w\in D\mid L(w)=N\}\text{ and } D_{\geq 
	N}=\{w\in D\mid L(w)\geq N\}\text{ etc.}
\]
Define $ \Omega_{\geq N} $ be the set of elements $ w=x\bd d\bd y $ for some $ d\in D_{\geq N} $, $ x,y\in W $. Let $ \Omega_{ N}=\Omega_{\geq N}\setminus \Omega_{>N} $, and $ \Omega_{<N}=W\setminus \Omega_{\geq N} $. For $ d\in D_N $,  $ U_d $ and $ B_d $ are defined as in \eqref{eq:ud} and \eqref{eq:bd}. 

Let $ N_0 =\max\{N\mid D_N\neq \emptyset \}$. Such an integer exists because $ S $ is a finite set.

By Theorem \ref{thm:generalized}, to prove P1-P15 for $ (W,S) $, we only need to check the following theorem. 
\begin{thm}\label{thm:right-angled}
	Let $ (W,S) $ be a right-angled Coxeter group. 
	
	The boundedness conjecture holds for $ W $, i.e.  $ \deg T_x T_y\leq N_0 $ for any $ x,y\in W $.

	Fix $ N\in\mathbb{N} $ such that $ W_{>N}=\Omega_{>N} $ is $ \prec_{LR} $ closed. We have the following properties.
	\begin{enumerate}
		\item  For $ x,y\in W_{\leq N} $, the equality $ \deg\nt_x\nt_y=N $ holds  only if $ y\in\Omega_{ N} $.
		\item For $ d\in D_N $, $ x\in U_d^{-1} $, $ y\in U_d $, we have $ xdy=x\bd d\bd y $.
		\item For  $ d\in D_N $,  $ v<d $, $ y\in U_d $,
		we have
		\begin{align}\label{key}
		\deg \nt_x\nt_v\nt_y&\leq N-L(v) \text{ if } v\leq d, x\in U_d^{-1},\\
		\deg \nt_b\nt_v\nt_y&< N-L(v)  \text{ if } v< d, b\in B_d.
		\end{align}
	\end{enumerate}
\end{thm}


The rest of this section is devoted to the proof of this theorem.

Recall that the exchange condition of a Coxeter group says that if $ w=s_1s_2\cdots s_n$ is a reduced expression and $ sw<w $, then there exists some $ i $ such that $ sw=s_1\cdots \hat{s}_i\cdots s_n $, where $ \hat{s}_i $ denotes deleting $ s_i $.

\begin{lem}\label{lem:comm}
	Assume that $ sw=wt>w $ for some $ w\in W $ and $ s,t\in S $, and $ w=s_1s_2\cdots s_n $ is a reduced expression. Then we have $ s=t $, and $ s $ commutes with all $ s_i $, $ i=1\cdots n $.
\end{lem}
\begin{proof}
	Note that $ s,s_1\in\mc{L}(sw) $. Then $ m_{ss_1}<\infty $ and hence $ m_{ss_1}=2  $, i.e. $ s $ commutes with $ s_1 $. Thus we obtain  $ ss_2s_3\cdots s_n=s_2s_3\cdots s_n t $. Then one can use induction on $ n $ to prove the lemma.
\end{proof}

\begin{lem}\label{lem:ra-com}
	Assume that $ x=s_ps_{p-1}\cdots s_1 $ and $ y=t_1t_2\cdots t_q $ are reduced expressions, and $ xy<x\bd y $. Let $ i $ be  the integer such that \[
	l(s_{i-1}\cdots s_1y)=l(y)+i-1, \text{ and } s_is_{i-1}\cdots s_1y<s_{i-1}\cdots s_1y.
	\]
	Let $ x_1=s_is_{i-1}\cdots s_1 $, and $ j $  be the integer such that \[
	l(x_1t_1t_2\cdots t_{j-1})=l(x_1)+j-1, \text{ and } x_1t_1t_{2}\cdots t_j<x_1t_1t_{2}\cdots t_{j-1}.
	\] Then $ s_i=t_j $ and $ s_i $ commutes with  $ s_{i'} $, $ t_{j'} $  for all $ i'<i $, $ j'<j $.
\end{lem}
\begin{proof}
Take
	$$  s=s_i , t=t_j , \text{ and }w= s_{i-1}\cdots s_2s_1t_1t_2\cdots t_{j-1} ,$$
	and then apply Lemma \ref{lem:comm}.
\end{proof}

\begin{lem}\label{lem:xdy}
Let $ x,y\in W $, $ d\in D $ such that $xd=x\bd d$, $ dy=d\bd y $, $ xdy<x\bd d\bd y$.
Then there exists $ s\in S$, $ x',y'\in W$ such that
 $$  x=x'\bd s , \quad y=s\bd y' ,\quad sd=ds >d .$$ This implies  \[
T_{x}T_dT_y=T_{x'}T_dT_{y'}+\xi_s T_{x'}T_{sd}T_{y'}.
\]
\end{lem}

\begin{proof}
Assume that $ xdy<x\bd d\bd y$,   and $$  x=s_p\cdots s_1 , \quad y=t_1\cdots t_q, \quad d=r_1r_2\cdots r_k  $$ are reduced expressions. Let $ i $ be the integer such that \[
l(s_{i-1}\cdots s_1dy)=l(dy)+i-1, \text{ and } s_is_{i-1}\cdots s_1dy<s_{i-1}\cdots s_1dy.
\]
Let $ x_1=s_is_{i-1}\cdots s_1 $, and  $ j $  be the integer such that \[
l(x_1dt_1t_2\cdots t_{j-1})=l(x_1d)+j-1, \text{ and } x_1dt_1t_{2}\cdots t_j<x_1dt_1t_{2}\cdots t_{j-1}.
\] By Lemma \ref{lem:ra-com}, $ s_i=t_j $ and $ s_i $ commutes with  $ s_{i'} $, $ t_{j'} $, $ r_{k'} $  for all $ i'<i $, $ j'<j $, $ k'\leq k$.
Take $$  s=s_i=t_j ,\quad x'=s_p\cdots \hat{s}_i\cdots s_1  ,\quad y'=t_1\cdots \hat{t}_j\cdots t_q .$$
 We have $ x=x's $, $ y=sy' $, $ sd=ds >d$. Then  \[
T_{x}T_dT_y=T_{x'}T_dT_{y'}+\xi_s T_{x'}T_{sd}T_{y'}.
\]
\end{proof}


\begin{lem}\label{lem:bound-ra}
	Let $ x,y\in W $, $ d\in D $ such that $xd=x\bd d$, $ dy=d\bd y $. Then
	\begin{equation}\label{eq:degc}
	\deg T_xT_dT_y\leq N_0-L(d).
	\end{equation}
\end{lem}

\begin{proof}
We use induction on $l (x) $. If $ l(x)=0 $,  the lemma is obvious. Assume that $ l(x)>0 $ and the lemma holds for all $ x' $ such that $ l(x')<l(x) $.

If $ xdy=x\bd d \bd y $,  the lemma is obvious. If $ xdy<x\bd d \bd y $, then by Lemma \ref{lem:xdy},  there exists $ s\in S$, $ x',y'\in W$ such that
 \[
T_{x}T_dT_y=T_{x'}T_dT_{y'}+\xi_s T_{x'}T_{sd}T_{y'}
\]
with  $  x=x'\bd s ,  y=s\bd y' , sd=ds >d .$ By the induction hypothesis,  $$  \deg T_{x'}T_dT_{y'}\leq N_0-L(d), \quad\deg  \xi_s T_{x'}T_{sd}T_{y'}\leq L(s)+N_0-L(sd)  .$$
Hence $  \deg T_{x}T_dT_y\leq N_0-L(d)$.
\end{proof}

\begin{cor}\label{cor:bouned}
For any $ x,y\in W $, we have $ \deg T_x T_y\leq N_0 $.
\end{cor}

\begin{proof}
	Take $ d=e $ in Lemma \ref{lem:bound-ra}.
\end{proof}

In the following, $ N $ is a fixed positive integer such that $ \wnn=\Omega_{>N} $ is $ \prec_{LR} $ closed. Then we can use results in subsection \ref{subsec:basis}.

\begin{lem}\label{lem:lengthadd}
	For  $ d\in D_N,$  $ x\in U_d^{-1} ,$  $y\in U_d$, we have $ xdy=x\bd d\bd y  $.
\end{lem}

\begin{proof}
	Otherwise, by Lemma  \ref{lem:xdy}, there exists some $ s\in S $, $ y'\in W $ such that  $ y=s\bd y' $, $ sd=ds >d$. This implies that $ dy =d\bd s\bd y'\in \Omega_{>N}$, a contradiction with $ y\in U_d $.
\end{proof}

\begin{lem}\label{lem:leqN}
For $ d\in D_{>N} $, we have\[
\nt_d=\sum_{v\in D_{\leq N}, v<d}g_{v,d}\nt_v
\]
for some $ g_{v,d}\in \mc{A} $ such that $ \deg g_{v,d}\leq L(v)-L(d) $.
\end{lem}

\begin{proof}
	Since $ C_d=\sum_{v\leq d} (q^{L(v)-L(d)} )T _v $ and $ d\in D_{>N} $, we have\[
	\nt_d=-\sum_{v<d} (q^{L(v)-L(d)})\nt _v.
	\]
	Note that $ v<d $ implies $ v\in D $, and
	it is possible that $ v\in D_{>N} $. One can use induction on length to prove that \[
	\nt_d=\sum_{v\in D_{\leq N}, v<d}g_{v,d}\nt_v
	\]
	for some $ g_{v,d}\in \mc{A} $ such that $ \deg g_{v,d}\leq L(v)-L(d) $.
\end{proof}

\begin{prop}\label{prop:bound}
	Let $ x,y\in W $, $ d\in D  $  such that $xd=x\bd d$, $ dy=d\bd y $. Then 
	\begin{equation}\label{key}
	\deg \nt_x\nt_d\nt_y\leq N-L(d),
	\end{equation}
 and the equality holds only if $ xd,dy\in\Omega_{\geq N} $.
\end{prop}\label{pro:ra-main}
\begin{proof}
	 We prove it  by induction on the length of $ x $.
	 
	Assume $ l(x)=0 $  and $ d\in D_{\leq N} $. Then
	\[
	\deg \nt_x\nt_d\nt_y =	\deg \nt_{dy}\leq 0\leq N-L(d).
	\]
	(When $ dy\in W_{>N} $, we need use  \eqref{eq:degree}.)
If the equality $ \deg \nt_x\nt_d\nt_y =  N-L(d)$ holds, then $ d\in D_N $ and $ xd,dy\in \Omega_{ \geq N} $.
	 
	 Assume $ l(x)=0 $  and $ d\in D_{> N} $. Using Lemma \ref{lem:leqN}, \[
	 	 \nt_x\nt_d\nt_y =\sum _{v\in D_{\leq N},v< d}	g_{v,d} \nt_{vy}
	 \]
	with $ g_{v,d} \in\mc{A}$ such that $ \deg g_{v,d} \leq L(v)-L(d) $. Since $$ \deg 	(g_{v,d} \nt_{vy})\leq L(v)-L(d) \leq N-L(d) ,$$ we have
$ 	\deg \nt_x\nt_d\nt_y\leq   N-L(d). $ Since $ d\in D_{>N} $, we always have $ xd,dy\in \Omega_{\geq N} $.

	 Assume now that $ l(x)>1 $ and that the lemma holds for all $ x' $ with $ l(x')<l(x) $.

	Assume that $ xdy<x\bd d\bd y $. By Lemma \ref{lem:xdy}, there exists some $ s\in S $, $ x',y'\in W $ such that $ x=x's $, $ y=sy' $, $ sd=ds >d$, and  \[
	\nt_{x}\nt_d\nt_y=\nt_{x'}\nt_d\nt_{y'}+\xi_s \nt_{x'}\nt_{sd}\nt_{y'}.
	\]
	By the induction hypothesis, we  have $$  \deg \nt_{x'}\nt_d\nt_{y'} \leq N-L(d) \text{ and }\deg  \nt_{x'}\nt_{sd}\nt_{y'}  \leq N-L(sd) ,$$ and hence $ \deg \nt_{x}\nt_d\nt_y\leq N-L(d) $. 
	Suppose that  $ \deg \nt_{x}\nt_d\nt_y= N-L(d) $. Then  $$  \deg \nt_{x'}\nt_d\nt_{y'}= N-L(d) \text{ or }\deg  \nt_{x'}\nt_{sd}\nt_{y'}  = N-L(sd) ,$$ which implies that $ x'd\in \Omega_{\geq N} $ or $ x'sd\in \Omega_{\geq N} $ by the induction hypothesis for $ x' $. Hence $ xd=x'sd=x'ds\in \Omega_{\geq N} $.  Similarly, $ y\in\Omega_{\geq N} $.	
	
		Assume  $xdy=x\bd d \bd y$ and $ d\in D_{\leq N} $. Then $\deg \nt_x\nt_d\nt_y \leq 0\leq N-L(d)$.   If  $ \deg \nt_x\nt_d\nt_y= N-L(d) =0$, then $ d\in D_N $ and  $ xd, dy\in\Omega_{\geq N} $.
	
	Assume  $xdy=x\bd d \bd y$ and $ d\in D_{> N} $.   Using Lemma \ref{lem:leqN}, \[
	\nt_x\nt_d\nt_y =\sum _{v\in D_{\leq N},v< d}	g_{v,d} \nt_x\nt_v\nt_{y}
	\]
	with $ g_{v,d} \in\mc{A}$ such that $ \deg g_{v,d} \leq L(v)-L(d) $. By the last two paragraphs, we have 	\[
	\deg	(g_{v,d} \nt_x\nt_v\nt_{y})\leq L(v)-L(d)+N-L(v)=N-L(d).
	\]
	Hence $ \deg \nt_x\nt_d\nt_y \leq N-L(d) $. If the  equality holds, then $ xv,vy\in \Omega_{ \geq N} $ for some $ v $, which implies that $xd,dy \in \Omega_{\geq N} $.
	
	This completes the proof.
\end{proof}

\begin{cor}\label{cor:bounded1}
	For any $ x,y\in\wn $, the equality  $ \deg \nt_x\nt_y= N $  holds only if $ x,y\in\Omega_{ N} $. 
\end{cor}
\begin{proof}
	Take $ d=e $ in Proposition \ref{prop:bound}, and note that $ x,y\in\wn= \Omega_{ \leq N} $.
\end{proof}

\begin{cor}\label{cor:bounded2}
Let   $ d\in D_N $, $ y\in U_d $.
\begin{itemize}
	\item [(i)] For $v\leq d,$  $ x\in U_d^{-1},$ we have $  \deg \nt_x\nt_v\nt_y\leq N-L(v)$.
	\item [(ii)] For $v< d,$  $ x\in B_d,$ we have $  \deg \nt_x\nt_v\nt_y< N-L(v)$.
\end{itemize}
\end{cor}
\begin{proof}
First note that for $ v<d $ we have $ v\in D_{<N} $.

Assertion (i) follows immediately from Proposition \ref{prop:bound}.

In assertion (ii), since  $v< d,$  $ x\in B_d$, we have $ xv\in \Omega_{<N} $.This  implies that $  \deg \nt_x\nt_v\nt_y< N-L(v)$ by Proposition \ref{prop:bound}.
\end{proof}

Now we have proved Theorem \ref{thm:right-angled}, see Corollaries \ref{cor:bouned}, \ref{cor:bounded1}, \ref{cor:bounded2} and Lemma \ref{lem:lengthadd}. Then P1-P15 follow from Theorem \ref{thm:generalized}.

%
%


\subsection*{Acknowledgments}
This work is completed during a visit to  the University of Sydney, and it was inspired when preparing a talk for an informal seminar organized by Geordie Williamson.
The author would like to thank Geordie Williamson for helpful discussions,
 and   the University of Sydney for  hospitality.

 The  author is  supported by NSFC Grants No.11601116 and No. 11801031, and Beijing Institute of Technology Research Fund Program for Young Scholars.

\bibliography{complete2}

\begin{thebibliography}{BGIL10}

\bibitem[AHR18]{achar2018cells}
P.~N. Achar, W.~Hardesty, and S.~Riche.
\newblock Conjectures on tilting modules and antispherical p-cells.
\newblock {\em arXiv:1812.09960v1}, 2018.

\bibitem[B{\'e}d86]{bedard1986infinity-leftcell}
R.~B{\'e}dard.
\newblock Cells for two {C}oxeter groups.
\newblock {\em Comm. Algebra}, 14(7):1253--1286, 1986.

\bibitem[B{\'e}d88]{bedard1988lowest}
R.~B{\'e}dard.
\newblock The lowest two-sided cell for an affine {W}eyl group.
\newblock {\em Comm. Algebra}, 16(6):1113--1132, 1988.

\bibitem[{Bel}04]{Bel2004}
Mikhail {Belolipetsky}.
\newblock {Cells and representations of right-angled Coxeter groups.}
\newblock {\em {Sel. Math., New Ser.}}, 10(3):325--339, 2004.

\bibitem[BGIL10]{bonnafe-geck-iancu-lam2010}
C.~Bonnaf\'{e}, M.~Geck, L.~Iancu, and T.~Lam.
\newblock On domino insertion and {K}azhdan-{L}usztig cells in type {$B_n$}.
\newblock In {\em Representation theory of algebraic groups and quantum
  groups}, volume 284 of {\em Progr. Math.}, pages 33--54.
  Birkh\"{a}user/Springer, New York, 2010.

\bibitem[BI03]{bonnafe-iancu2003leftcell}
C.~Bonnaf\'{e} and L.~Iancu.
\newblock Left cells in type {$B_n$} with unequal parameters.
\newblock {\em Represent. Theory}, 7:587--609, 2003.

\bibitem[Bla09]{blasiak2009factorization}
J.~Blasiak.
\newblock A factorization theorem for affine {K}azhdan–{L}usztig basis
  elements.
\newblock {\em arXiv:0908.0340v1}, 2009.

\bibitem[Bon06]{bonnafe2006twosidedcell}
C.~Bonnaf\'{e}.
\newblock Two-sided cells in type {$B$} (asymptotic case).
\newblock {\em J. Algebra}, 304(1):216--236, 2006.

\bibitem[Bon17]{bonnafe2017book}
C.~Bonnaf\'{e}.
\newblock {\em Kazhdan-{L}usztig cells with unequal parameters}, volume~24 of
  {\em Algebra and Applications}.
\newblock Springer, Cham, 2017.

\bibitem[Bre97]{bremke1997lowest}
K.~Bremke.
\newblock On generalized cells in affine {W}eyl groups.
\newblock {\em J. Algebra}, 191(1):149--173, 1997.

\bibitem[EW14]{elias-williamson2014positivity}
B.~Elias and G.~Williamson.
\newblock The {H}odge theory of {S}oergel bimodules.
\newblock {\em Ann. of Math. (2)}, 180(3):1089--1136, 2014.

\bibitem[Gao16]{gao2016rank3}
J.~Gao.
\newblock The boundness of weighted coxeter groups of rank 3.
\newblock {\em arXiv:1607.02286}, 2016.

\bibitem[Gec02]{geck2002leading}
M.~Geck.
\newblock Constructible characters, leading coefficients and left cells for
  finite {C}oxeter groups with unequal parameters.
\newblock {\em Represent. Theory}, 6:1--30, 2002.

\bibitem[Gec06]{geck2006relative}
M.~Geck.
\newblock Relative {K}azhdan-{L}usztig cells.
\newblock {\em Represent. Theory}, 10:481--524, 2006.

\bibitem[Gec11]{geck2011rank2}
M.~Geck.
\newblock On {I}wahori-{H}ecke algebras with unequal parameters and {L}usztig's
  isomorphism theorem.
\newblock {\em Pure Appl. Math. Q.}, 7(3, Special Issue: In honor of Jacques
  Tits):587--620, 2011.

\bibitem[GI06]{geck-iancu2006afunction}
M.~Geck and L.~Iancu.
\newblock Lusztig's {$a$}-function in type {$B_n$} in the asymptotic case.
\newblock {\em Nagoya Math. J.}, 182:199--240, 2006.

\bibitem[GP19a]{guilhot-parkinson2019C2}
J.~Guilhot and J.~Parkinson.
\newblock Balanced representations, the asymptotic {P}lancherel formula, and
  {L}usztig's conjectures for $\tilde{C}_2$.
\newblock {\em To appear in Algebraic Combinatorics}, 2019.

\bibitem[GP19b]{guilhot-parkinson2019G2}
J.~Guilhot and J.~Parkinson.
\newblock A proof of {L}usztig's conjectures for affine type $\tilde{G}_2$ with
  arbitrary parameters.
\newblock {\em To appear in Proceedings of the London Mathematical Society},
  2019.

\bibitem[Gui08a]{guilhot2008lowest}
J.~Guilhot.
\newblock On the lowest two-sided cell in affine {W}eyl groups.
\newblock {\em Represent. Theory}, 12:327--345, 2008.

\bibitem[Gui08b]{guilhot2008computedrank2}
J.~Guilhot.
\newblock Some computations about kazhdan-lusztig cells in affine weyl groups
  of rank 2.
\newblock {\em arxiv.org/0810.5165}, 2008.

\bibitem[Gui10]{guilhot2010cells}
J.~Guilhot.
\newblock Kazhdan-{L}usztig cells in affine {W}eyl groups of rank 2.
\newblock {\em Int. Math. Res. Not. IMRN}, (17):3422--3462, 2010.

\bibitem[Hum02]{Hum2002modular}
J.~E. Humphreys.
\newblock Analogues of {W}eyl's formula for reduced enveloping algebras.
\newblock {\em Experiment. Math.}, 11(4):567--573 (2003), 2002.

\bibitem[KL79]{kazhdan_lusztig79representation}
D.~Kazhdan and G.~Lusztig.
\newblock Representations of {C}oxeter groups and {H}ecke algebras.
\newblock {\em Invent. Math.}, 53(2):165--184, 1979.

\bibitem[KL80]{kazhdan-lusztig1980positivity}
D.~Kazhdan and G.~Lusztig.
\newblock Schubert varieties and {P}oincar\'{e} duality.
\newblock In {\em Geometry of the {L}aplace operator ({P}roc. {S}ympos. {P}ure
  {M}ath., {U}niv. {H}awaii, {H}onolulu, {H}awaii, 1979)}, Proc. Sympos. Pure
  Math., XXXVI, pages 185--203. Amer. Math. Soc., Providence, R.I., 1980.

\bibitem[LS19]{li-Shi-non3edge}
Y.~Li and J.~Shi.
\newblock The boundedness of a weighted {C}oxeter group with
  non-3-edge-labeling graph.
\newblock {\em J. Algebra Appl.}, 18(5):1950085, 43, 2019.

\bibitem[Lus83]{lusztig1983leftCell}
G.~Lusztig.
\newblock Left cells in {W}eyl groups.
\newblock In {\em Lie group representations, {I} ({C}ollege {P}ark, {M}d.,
  1982/1983)}, volume 1024 of {\em Lecture Notes in Math.}, pages 99--111.
  Springer, Berlin, 1983.

\bibitem[Lus84]{lusztig1984char}
G.~Lusztig.
\newblock {\em Characters of reductive groups over a finite field}, volume 107
  of {\em Annals of Mathematics Studies}.
\newblock Princeton University Press, Princeton, NJ, 1984.

\bibitem[Lus85]{lusztig1985cellsI}
G.~Lusztig.
\newblock Cells in affine {W}eyl groups.
\newblock In {\em Algebraic groups and related topics ({K}yoto/{N}agoya,
  1983)}, volume~6 of {\em Adv. Stud. Pure Math.}, pages 255--287.
  North-Holland, Amsterdam, 1985.

\bibitem[Lus87a]{lusztig1987cellsII}
G.~Lusztig.
\newblock Cells in affine {W}eyl groups. {II}.
\newblock {\em J. Algebra}, 109(2):536--548, 1987.

\bibitem[Lus87b]{lusztig1987III}
G.~Lusztig.
\newblock Cells in affine {W}eyl groups. {III}.
\newblock {\em J. Fac. Sci. Univ. Tokyo Sect. IA Math.}, 34(2):223--243, 1987.

\bibitem[Lus89]{lusztig1989cellsIV}
G.~Lusztig.
\newblock Cells in affine {W}eyl groups. {IV}.
\newblock {\em J. Fac. Sci. Univ. Tokyo Sect. IA Math.}, 36(2):297--328, 1989.

\bibitem[Lus03]{lusztig2003hecke}
G.~Lusztig.
\newblock {\em Hecke algebras with unequal parameters}, volume~18 of {\em CRM
  Monograph Series}.
\newblock American Mathematical Society, Providence, RI, 2003.

\bibitem[Shi87]{shi1987lowest-I}
J.~Shi.
\newblock A two-sided cell in an affine {W}eyl group.
\newblock {\em J. London Math. Soc. (2)}, 36(3):407--420, 1987.

\bibitem[Shi88]{shi1988lowest-II}
J.~Shi.
\newblock A two-sided cell in an affine {W}eyl group. {II}.
\newblock {\em J. London Math. Soc. (2)}, 37(2):253--264, 1988.

\bibitem[Shi15]{shi-2015-reduced-strict}
J.~Shi.
\newblock The reduced expressions in a {C}oxeter system with a strictly
  complete {C}oxeter graph.
\newblock {\em Adv. Math.}, 272:579--597, 2015.

\bibitem[Shi18]{Shi2016reduced}
J.~Shi.
\newblock Reduced expressions in a {C}oxeter system with a complete {C}oxeter
  graph.
\newblock {\em J. Combin. Theory Ser. A}, 159:240--268, 2018.

\bibitem[SY15]{Shi-Yang2015universal}
J.~Shi and Gao Y.
\newblock The weighted universal {C}oxeter group and some related conjectures
  of {L}usztig.
\newblock {\em J. Algebra}, 441:678--694, 2015.

\bibitem[SY16]{Shi-Yang2016}
J.~{Shi} and G.~{Yang}.
\newblock {The boundedness of the weighted Coxeter group with complete graph.}
\newblock {\em {Proc. Am. Math. Soc.}}, 144(11):4573--4581, 2016.

\bibitem[Xi90]{xi1990based_ring}
N.~Xi.
\newblock The based ring of the lowest two-sided cell of an affine {W}eyl
  group.
\newblock {\em J. Algebra}, 134(2):356--368, 1990.

\bibitem[Xi94a]{xi1994based}
N.~Xi.
\newblock The based ring of the lowest two-sided cell of an affine {W}eyl
  group. {II}.
\newblock {\em Ann. Sci. {\'E}c. Norm. Sup{\'e}r. (4)}, 27:47--61, 1994.

\bibitem[Xi94b]{xi1994book}
N.~Xi.
\newblock {\em Representations of affine {H}ecke algebras}, volume 1587 of {\em
  Lecture Notes in Mathematics}.
\newblock Springer-Verlag, Berlin, 1994.

\bibitem[Xi07]{xi2007rep}
N.~Xi.
\newblock Representations of affine {H}ecke algebras and based rings of affine
  {W}eyl groups.
\newblock {\em J. Amer. Math. Soc.}, 20(1):211--217, 2007.

\bibitem[Xi12]{xi2012afunction}
N.~Xi.
\newblock Lusztig's {$A$}-function for {C}oxeter groups with complete graphs.
\newblock {\em Bull. Inst. Math. Acad. Sin. (N.S.)}, 7(1):71--90, 2012.

\bibitem[Xie15]{xie2015dec-rank2}
X.~Xie.
\newblock A decomposition formula for the kazhdan-lusztig basis of affine hecke
  algebras of rank 2.
\newblock {\em arXiv:1509.05991}, 2015.

\bibitem[Xie17a]{xie2015lowest}
X.~Xie.
\newblock The based ring of the lowest generalized two-sided cell of an
  extended affine {W}eyl group.
\newblock {\em J. Algebra}, 477:1--28, 2017.

\bibitem[Xie17b]{xie2017complete}
X.~Xie.
\newblock The lowest two-sided cell of a {C}oxeter group with complete graph.
\newblock {\em J. Algebra}, 489:38--58, 2017.

\bibitem[Zho13]{zhou}
P.~Zhou.
\newblock Lusztig's {$a$}-function for {C}oxeter groups of rank 3.
\newblock {\em J. Algebra}, 384:169--193, 2013.

\end{thebibliography}

\end{document}